\documentclass{article}
\usepackage{graphicx} 
\usepackage{amsmath}
\usepackage{amssymb}
\usepackage{amsthm}
\usepackage{multirow}
\usepackage{rotating}
\usepackage{algorithm,algorithmicx}
\usepackage{algpseudocode}
\usepackage{bm}
\usepackage{tabularx}
\usepackage{array}
\usepackage{tikz}
\usepackage{authblk}
\usepackage{url}
\usepackage{xcolor}

\newcolumntype{Y}{>{\centering\arraybackslash}X}

\newtheorem{corollary}{Corollary}
\newtheorem{proposition}{Proposition}
\theoremstyle{remark} 
\newtheorem{remark}{Remark}

\title{Single-level reformulations for network pricing under Wardrop equilibrium with separable affine costs}
\author[1,2]{Miguel F. Anjos}
\author[3]{Jos\'e Neto}
\affil[1]{Group for Research in Decision Analysis (GERAD), Montr\'eal, Qu\'ebec, Canada}
\affil[2]{School of Mathematics and Maxwell Institute for Mathematical Sciences,
University of Edinburgh, Edinburgh, United Kingdom}
\affil[3]{SAMOVAR, T\'el\'ecom SudParis, Institut Polytechnique de Paris,
91120 Palaiseau, France}
\date{October 17, 2026}

\begin{document}

\maketitle
\begin{abstract}
We consider the problem of setting tolls on a subset of the arcs of a congested network
so as to maximize the revenue of a network authority, anticipating that users route
according to a Wardrop equilibrium. 
This problem is naturally cast as a bilevel program in which the leader
sets tolls while the followers respond by choosing an equilibrium routing. 

We focus on the important special case of separable affine arc cost functions. Under this
assumption, Wardrop equilibria admit Beckmann's convex characterization, from which we
derive three new single-level reformulations of the bilevel problem, exploiting both
Beckmann's formulation and its dual: an arc-based formulation, which we solve exactly, and
two path-based formulations, solved by column-and-constraint generation.
Beyond the reformulations,  
we further derive from Beckmann's
potential an upper bound on the optimal revenue that can be used to certify optimality gaps 
across all formulations.

An extensive computational study on instances derived from the Sioux Falls and Eastern
Massachusetts networks compares our formulations with the recent mixed-integer nonlinear
formulation of Beck, Labb\'e and Schmidt~\cite{BLS_2024}. 
 The arc-based formulation solves substantially more instances to proven optimality (within a time limit of 1800\,s) than the formulation of~\cite{BLS_2024} and, unlike it, returns a feasible solution on every instance, remaining 
effective on the largest 74-node network, on which~\cite{BLS_2024} solves none. 
The magnitude of the upper bounds on the tolls proves to be a major driver of tractability: 
the genuinely hard core of the benchmark emerges once the regulatory cap on tolls is removed. 
The two path-based formulations, used as heuristics, return a solution within 1\% of the optimum on 
every instance whose optimal value is known, often in seconds where exact formulations reach 
the time limit. 
\end{abstract}

\section{Introduction}
Pricing mechanisms play a central role in transportation and telecommunication
networks, both as a source of revenue and as a means of influencing users'
routing decisions in order to improve the overall performance of the system
\cite{ABEJW_2006,vH_2008}.
In many practical situations, however, the travel cost experienced by users also
depends on congestion effects generated by the aggregate traffic,
a phenomenon well documented in the road traffic literature \cite{BMW_1956,W_52}.
As a consequence, pricing decisions and traffic equilibria become strongly coupled,
leading naturally to bilevel optimization models in which a central authority
sets tolls anticipating the equilibrium response of the network users
\cite{BLMS_2001,LMS_1998}.

In this work, we consider a toll-setting problem in a congested network.
The network is represented by a directed graph $G=(V,A)$,
where $V$ denotes the set of nodes and $A$ the set of arcs.
For each arc $a\in A$,
the unit transportation cost is given by a flow-dependent function $c_a(\cdot)$,
and no capacity constraints are imposed on the arcs.
All arc cost functions are assumed to be nonnegative, continuous,
nondecreasing, and separable,
meaning that $c_a$ depends only on the total flow routed through arc $a$.

A set $\mathcal{K}\subseteq V\times V$ of origin-destination (OD) pairs is given.
For each commodity $k\in\mathcal{K}$,
the pair $(o_k,s_k)$ is associated with a demand $d_k>0$
that must be fully routed from $o_k$ to $s_k$.
We denote by $\mathcal{P}_k$ the set of elementary paths from $o_k$ to $s_k$ in $G$,
and assume throughout that $\mathcal{P}_k\neq\emptyset$ for every $k\in\mathcal{K}$.

A leader seeks to determine nonnegative tolls
$\tau=(\tau_a)_{a\in A}$
belonging to a prescribed compact and convex set
$\mathcal{T}\subseteq\mathbb{R}_+^A$.
We further assume that $\mathcal{T}$ is \emph{downward closed},
meaning that if $\tau\in\mathcal{T}$ and
$0\le\tau'\le\tau$ componentwise, then $\tau'\in\mathcal{T}$.
This assumption is satisfied in particular when $\mathcal{T}$
is defined by independent upper bound constraints on each toll, 
which is 
a common setting in the literature. 
These tolls are added to the flow-dependent arc costs,
yielding effective arc costs
\[
\widehat{c}_a(f_a,\tau_a) = c_a(f_a) + \tau_a, \qquad a\in A,
\]
where $f_a$ denotes the total flow routed through arc $a$.
The followers are nonatomic users who route themselves selfishly
according to Wardrop's equilibrium principle \cite{W_52,WW_52}:
for each commodity, flow is distributed only along minimum-cost paths
with respect to the effective costs $\widehat{c}_a$.

The resulting optimization problem can be formulated as the bilevel program
\begin{equation}
\label{eq:BW}
\tag{\tt{BW}}
\begin{array}{ll}
\max_{\tau\in\mathcal{T},\,f} & F(f,\tau) \\[0.2cm]
\mathrm{s.t.} & f\in\mathcal{W}(\tau),
\end{array}
\end{equation}
where $\mathcal{W}(\tau)$ denotes the set of Wardrop equilibrium flows
associated with the toll vector $\tau$,
and $F$ represents the leader's objective function.
Depending on the application,
$F$ may model revenue maximization, network efficiency,
or other system-level performance criteria \cite{DZ_2012,HARKS2019947}.

Problem~\eqref{eq:BW} is inherently difficult to solve,
combining a bilevel structure, equilibrium constraints,
and a 
dependence of arc costs on the equilibrium flows.
In this work, we focus on the practically and theoretically important
special case of separable affine arc cost functions \cite{LMP_1975,PAT_1994}.
Under this assumption,
Wardrop equilibria admit a characterization through Beckmann's convex potential formulation \cite{BMW_1956},
which provides a powerful and tractable framework for deriving exact single-level reformulations.

The main contribution of this paper is the derivation of new exact single-level
reformulations of problem~\eqref{eq:BW} in the affine separable case, exploiting both
Beckmann's formulation and its dual. We introduce three such reformulations --- an
arc-based formulation, which we solve exactly, and two path-based formulations, which
we solve heuristically by column-and-constraint generation --- and compare them with
the recent formulation proposed by Beck et al.~\cite{BLS_2024}. Beyond the
reformulations, we show that the common assumption that every commodity admits a
toll-free path induces finite implicit upper bounds on the tolls
(Proposition~\ref{prop:toll_upper_bound}), allowing the unbounded case to be treated
without arbitrary bounds, and we derive from Beckmann's potential an upper bound on the
optimal revenue that certifies optimality gaps uniformly across all formulations. An
extensive computational study on the Sioux Falls and Eastern Massachusetts networks
then shows that the arc-based formulation solves substantially more instances to proven
optimality and remains effective on the largest, $74$-node network, on which the
formulation of Beck et al. solves none, while the two path-based heuristics recover the
optimum --- or a solution within less than 1\% of it --- wherever it is known.

The remainder of the paper is organized as follows.
Section~\ref{sec:related_works} reviews the related literature.
Section~\ref{sec:wardrop} recalls Beckmann's potential formulation for Wardrop equilibria
and derives the dual formulation used throughout the paper.
Single-level reformulations are then introduced and analyzed in Section~\ref{sec:slr}.
Bounds and a reduction of the model size are discussed in Section 
\ref{sec:comput_matters}, 
and computational experiments are reported in Section~\ref{sec:computational_experiments}.
Concluding remarks are given in Section~\ref{sec:conclusion}.

\section{Related works}
\label{sec:related_works}

Network pricing problems have been extensively studied in the literature,
especially in the context of transportation and telecommunication networks;
see \cite{ABEJW_2006,vH_2008} for surveys covering a broad range of models and applications.
A large part of the literature considers the case where arc costs are fixed,
that is, independent of congestion effects.
In this setting, the followers' problem reduces to a collection of independent
shortest path computations, one for each origin-destination pair.
When the leader seeks to maximize toll revenues,
the resulting bilevel problem corresponds to the classical
\emph{network pricing problem},
investigated for instance in \cite{BLMS_2001,BCN_2024,LMS_1998}.
Complexity aspects of this problem have received considerable attention: 
 \cite{BCN_2024} establishes asymmetry results
in the complexity of the multi-commodity case,
and polynomial special cases have been identified in \cite{vHvdKMOB_2003}.

The framework considered in the present paper differs fundamentally from
these models because arc costs depend on the equilibrium flow.
Consequently, the lower-level problem is no longer a collection of independent
shortest path problems, but rather the computation of a Wardrop equilibrium,
which couples all commodities through shared congestion effects.

Another classical line of research concerns congestion pricing
and marginal-cost pricing.
Under suitable assumptions on the arc cost functions,
it is well known that marginal-cost tolls can induce a Wardrop equilibrium
that minimizes the total travel cost,
\[
\sum_{a\in A} f_a\, c_a(f_a),
\]
thereby achieving a socially optimal outcome.
Such results are classical in transportation science
and algorithmic game theory \cite{AGT_2007,BMW_1956}.
However, these approaches focus on achieving socially optimal equilibria
through analytically derived tolls,
whereas problem~\eqref{eq:BW} considers a general bilevel optimization framework
in which the leader may pursue arbitrary objectives,
including revenue maximization or system efficiency,
subject to a feasible toll set $\mathcal{T}$.
The structure of $\mathcal{T}$ itself has been the subject of dedicated studies:
Harks et al.~\cite{HARKS2019947} investigate the impact of toll caps
in privatized road networks, showing that restricting the toll space
can significantly affect equilibrium outcomes and network performance.

Several works have investigated pricing models involving equilibrium constraints
under assumptions that differ from those considered here.
In \cite{BDW_2002}, the authors consider fair resource allocation problems
in capacitated networks, followed by a revenue optimization phase
based on Lagrange multipliers associated with capacity constraints.
Ba\c{s}ar and Srikant \cite{BS_2002} study a Stackelberg game
under Nash equilibrium constraints,
where followers correspond to finitely many users
choosing among predetermined routes,
and where congestion and pricing jointly influence user utilities.
Dempe and Zemkoho \cite{DZ_2012} investigate theoretical properties
and optimality conditions for bilevel road pricing problems
with separable arc costs, providing a rigorous analytical foundation
for this class of models.

To the best of our knowledge,
the work of Beck et al.~\cite{BLS_2024}
is currently the only contribution proposing both
single-level reformulations and computational experiments
for a bilevel toll-setting problem with Wardrop equilibrium constraints
in a setting directly comparable to ours.
Notably, Beck et al.\ also consider a robust variant of the problem
under budgeted uncertainty  on the travel costs,
whereas the present paper focuses on the deterministic case.
We build upon the affine separable structure of the arc cost functions
to derive alternative exact single-level reformulations
based on Beckmann's potential formulation \cite{BMW_1956}
and duality arguments,
 introduce bounds, a model size reduction and provide a  
computational study 
of the resulting formulations.

\section{Wardrop equilibria and Beckmann reformulation} 
\label{sec:wardrop}
In this section, we recall classical results on Wardrop equilibria
that will be used throughout the paper.
We first present Beckmann's convex  
formulation \cite{BMW_1956}
for the computation of Wardrop equilibria under separable arc cost functions. 
We then specialize to the case of affine arc costs
and derive the corresponding dual formulation,
which plays a central role in the single-level reformulations
introduced in Section~\ref{sec:slr}.

\subsection{Beckmann's formulation and KKT conditions}
\label{subsec:convex_reform}
In the context of congested traffic networks,
a flow is said to be a \emph{Wardrop equilibrium} if,
for each commodity $k\in\mathcal{K}$,
flow is routed only along paths of minimum effective cost
from $o_k$ to $s_k$ \cite{W_52,WW_52}.
In other words, no user can unilaterally reduce their travel cost
by switching to a different route.

Given a flow $f\in\mathbb{R}_{+}^{A}$ and a toll vector $\tau\in\mathbb{R}^{A}_{+}$,
the effective cost of path $p\in\mathcal{P}_k$ for commodity $k\in\mathcal{K}$
is defined as
\[
\widehat{c}_{p,k}(f,\tau) = \sum_{a\in p}\widehat{c}_a(f_a,\tau_a),
\]
and we denote by $\pi_k(f,\tau)$ the minimum effective cost
of an $(o_k,s_k)$-path.
For each commodity $k\in\mathcal{K}$ and arc $a\in A$,
let $f_{p,k}$ denote the flow routed through path $p\in\mathcal{P}_k$,
and $f_{a,k}$ the corresponding arc flow for commodity $k$.
The total arc flow is then
\[
f_a = \sum_{k\in\mathcal{K}} f_{a,k}, \qquad a\in A.
\]
Throughout this section and the remainder of the paper,
the letter $f$ may refer either to the path-flow vector
$\left(f_{p,k}\right)_{k\in\mathcal{K},\, p\in\mathcal{P}_k}$,
to the arc-flow vector $\left(f_a\right)_{a\in A}$,
or to both simultaneously, with the intended meaning clear from context.

Using the notation above,
a flow $f$ is a Wardrop equilibrium with respect to a fixed toll vector
$\tau\in\mathbb{R}_{+}^{A}$ if and only if it satisfies the following system: 
\begin{subequations}
\label{eq:wardrop_syst1}
\begin{align}
&\sum_{p\in\mathcal{P}_k}{f_{p,k}} = d_k & k\in \mathcal{K}
\label{eq:wardrop_syst1_1}\\
&\widehat{c}_{p,k}\left(f,\tau\right)-\pi_k
\geq 0 &p\in \mathcal{P}_k,\ k\in \mathcal{K}
\label{eq:wardrop_syst1_2}\\
&f_{p,k}\left(\widehat{c}_{p,k}\left(f,\tau\right)-\pi_k
\right)=0 &p\in \mathcal{P}_k,\ k\in \mathcal{K}
\label{eq:wardrop_syst1_3}\\
&f_{p,k}\geq 0
& p\in \mathcal{P}_k,\ k\in \mathcal{K}
\label{eq:wardrop_syst1_4}\\
&\pi_k\geq 0 & \ k\in \mathcal{K}
\label{eq:wardrop_syst1_5}
\end{align}
\end{subequations} 
Constraints \eqref{eq:wardrop_syst1_1} ensure that all demands are fully satisfied. 
Constraints \eqref{eq:wardrop_syst1_2} impose that, for each commodity $k\in\mathcal K$, every path in $\mathcal P_k$ has length at least $\pi_k$. 
Complementarity constraints \eqref{eq:wardrop_syst1_3} ensure that only shortest paths may carry a strictly positive amount of flow. 
Constraints~\eqref{eq:wardrop_syst1_4}--\eqref{eq:wardrop_syst1_5}
enforce nonnegativity of path flows and shortest-path costs, respectively.

System~\eqref{eq:wardrop_syst1} corresponds exactly to the KKT optimality conditions
of Beckmann's  formulation \cite{BMW_1956},
given by the following convex program:
\begin{subequations}
\label{eq:wardrop_syst2}
\begin{align}
\min_{f}\ \  &B(f,\tau)=\sum_{a\in A}{\int_{0}^{f_a}{\widehat{c}_a\left(x,\tau_a\right)dx}}\\
\mbox{s.t.} & \sum_{p\in\mathcal{P}_k}{f_{p,k}} = d_k & k\in \mathcal{K}\\
&\sum_{k\in\mathcal{K}}{\sum_{\substack{p\in\mathcal{P}_k \colon \\ a\in p}}{f_{p,k}}} 
=f_a & a\in A \label{eq_wardrop_syst2_2}\\
&f_{p,k} \geq 0 
&k\in \mathcal{K}, p\in\mathcal{P}_k\\
&f_a \geq 0 & a \in A
\end{align}
\end{subequations} 

Under the assumptions above, a Wardrop equilibrium always exists \cite{BMW_1956,PAT_1994}.
Moreover, if the arc cost functions are strictly increasing,
the equilibrium arc flow vector is unique \cite{BMW_1956}.
When arc costs are only nondecreasing,
multiple Wardrop equilibria may exist and may yield different values
for the leader's objective in~\eqref{eq:BW}.
 In this work, we adopt the standard \emph{optimistic} assumption, meaning that 
among all lower-level optimal solutions,  one maximizing the leader's objective is selected. 

\subsection{Affine arc cost functions and dual formulation}
In this subsection, we specialize to the case of affine arc cost functions
$c_a\colon \mathbb{R}_+\rightarrow \mathbb{R}$ of the form
\[
c_a(f_a) = \alpha_a \cdot f_a + \beta_a, \qquad (\alpha_a, \beta_a)\in\mathbb{R}_{+}^{2},
\]
which are nonnegative and nondecreasing by assumption.
Under this structure, Beckmann's formulation~\eqref{eq:wardrop_syst2}
takes the following explicit form for a fixed toll vector $\tau\in\mathbb{R}_{+}^{A}$. 
\begin{subequations}
\label{eq:wardrop_syst_aff}
\begin{align}
\min_{f\geq 0}\ \  &B(f,\tau)=\sum_{a\in A}{\left(\frac{\alpha_a}{2}\cdot f_a^2 + \left(\beta_a+\tau_a\right)\cdot f_a\right)}\label{eq:wardrop_syst_aff_obj}\\
\mbox{s.t.} & \sum_{p\in\mathcal{P}_k}{f_{p,k}} = d_k, & k\in \mathcal{K} \label{eq:wardrop_syst_aff_1}\\
&\sum_{k\in\mathcal{K}}{\sum_{\substack{p\in\mathcal{P}_k \colon \\ a\in p}}{f_{p,k}}} 
-f_a\leq 0, & a\in A \label{eq:wardrop_syst_aff_2}
\end{align}
\end{subequations} 
Compared with~\eqref{eq:wardrop_syst2},
the flow conservation equality~\eqref{eq_wardrop_syst2_2} has been relaxed to inequality ~\eqref{eq:wardrop_syst_aff_2}.
This relaxation is valid because, under nonnegative and nondecreasing arc costs, there always exists 
an optimal solution of~\eqref{eq:wardrop_syst_aff} satisfying the equality at optimality:
increasing $f_a$ beyond the actual routed flow may only 
increase the objective.

We now derive the dual of~\eqref{eq:wardrop_syst_aff}
by forming the Lagrangian with multipliers
$\pi_k\in\mathbb{R}$ for the demand constraints
and $\lambda_a\geq 0$ for the  
inequalities~\eqref{eq:wardrop_syst_aff_2}.
The Lagrangian is given by 
\begin{align*}
&\mathcal{L}\left(f,\pi,\lambda\right)=\\
&B\left(f,\tau\right)+
\sum_{k\in\mathcal{K}}{\pi_k \cdot\left(\sum_{p\in\mathcal{P}_k}{f_{p,k}}-d_k\right)} 
+\sum_{a\in A}{\lambda_a\cdot\left(\sum_{k\in\mathcal{K}}{\sum_{\substack{p\in\mathcal{P}_k \colon \\ a\in p}}{f_{p,k}}} 
-f_a\right)}
\end{align*}
with $f\geq 0$, $\lambda\geq 0$. 

The Lagrangian decomposes separably over arcs and commodities as follows: 
\begin{align*}
&\mathcal{L}\left(f,\pi,\lambda\right)=
\sum_{a\in A}{\mathcal{L}_{1,a}\left(f_a,\lambda_a\right)} +\sum_{k\in\mathcal{K}}{\mathcal{L}_{2,k}\left(f_{\cdot,k},\pi_k,\lambda\right)} -
\sum_{k\in\mathcal{K}}{\pi_k \cdot d_k} 
\end{align*}
where $f_{\cdot,k} = (f_{p,k})_{p\in\mathcal{P}_k}$ denotes the path-flow vector
for commodity $k$, and 
\begin{align*}
\mathcal{L}_{1,a}\left(f_a,\lambda_a\right)=\frac{\alpha_a}{2}\cdot f_a^2 + \left(\beta_a+\tau_a-\lambda_a\right)\cdot f_a\\
\mathcal{L}_{2,k}\left(f_{\cdot,k},\pi_k,\lambda\right)=
\sum_{p\in\mathcal{P}_k}{
\left(\pi_k+\left(\sum_{a\in p}{\lambda_a}\right)\right)\cdot f_{p,k}}.
\end{align*}

The dual function of the Lagrangian is defined by 
$g\left(\pi,\lambda\right)=\min_{f\geq 0}{\mathcal{L}\left(f,\pi,\lambda\right)}$. 
Let $A_0 = \{a\in A \colon \alpha_a = 0\}$ denote the set of arcs
with constant cost function, and let $A_1 = A\setminus A_0$.
This partition arises naturally when minimizing the Lagrangian,
as the two cases yield 
different expressions. 
Indeed, minimizing the Lagrangian term by term over $f\geq 0$, we obtain: 
\begin{align*}
\min_{f_a\geq 0}{\mathcal{L}_{1,a}\left(f_a,\lambda_a\right)}=
\begin{cases}
0 \mbox{ if } \beta_a+\tau_a-\lambda_a\geq 0\\
-\infty \mbox{ otherwise }
\end{cases}
\mbox{ if }a\in A_0\\
\min_{f_a\geq 0}{\mathcal{L}_{1,a}\left(f_a,\lambda_a\right)}=
\begin{cases}
0 \mbox{ if } \beta_a+\tau_a-\lambda_a\geq 0\\
-\frac{\left(\beta_a+\tau_a-\lambda_a\right)^2}{2\cdot \alpha_a} \mbox{ otherwise }
\end{cases}
\mbox{ if }a\in A_1\\
\min_{f_{\cdot,k}\geq 0}{\mathcal{L}_{2,k}\left(f_{\cdot,k},\pi_k,\lambda\right)}=
\begin{cases}
0 \mbox{ if }-\pi_k\leq \sum_{a\in p}{\lambda_a}, \forall p\in \mathcal{P}_k\\
-\infty \mbox{ otherwise}.
\end{cases}
\end{align*}
We thus deduce, after the change of variables $\pi_k\leftarrow-\pi_k$, the dual problem
\begin{subequations}
\label{eq:wardrop_syst_aff_dual}
\begin{align}
\max_{\lambda\geq 0,\pi} & \sum_{k\in\mathcal{K}}{\pi_k\cdot d_k}
-\sum_{a\in A_1}{\frac{1}{2\cdot\alpha_a} \left[\min\left( 0,\ \beta_a+\tau_a-\lambda_a\right)\right]^2}\\
\mbox{s.t. }& \pi_k\leq \sum_{a\in p}{\lambda_a}, &k\in \mathcal{K}, p\in \mathcal{P}_k,\\
& \beta_a+\tau_a-\lambda_a \geq 0
& a\in A_0 
\end{align}
\end{subequations}

\begin{remark}
\label{rk:dual_simplification}
Without loss of generality, one may restrict attention to optimal solutions
$(\lambda^*,\pi^*)$ of~\eqref{eq:wardrop_syst_aff_dual}
satisfying $\beta_a + \tau_a - \lambda_a^* \leq 0$ for all $a\in A_1$.
Indeed, if this condition is violated for some arc $\tilde{a}\in A_1$,
then increasing $\lambda_{\tilde{a}}^*$ up to $\beta_{\tilde{a}}+\tau_{\tilde{a}}$
preserves feasibility and leaves the objective value unchanged,
since $[\min(0, \beta_{\tilde{a}}+\tau_{\tilde{a}}-\lambda_{\tilde{a}})]^2 = 0$
throughout this range; repeating this for every such arc yields an optimal
solution with the required property. 
\end{remark}

In light of Remark~\ref{rk:dual_simplification},
the objective of~\eqref{eq:wardrop_syst_aff_dual} simplifies,
yielding the equivalent dual formulation: 
\begin{subequations}
\label{eq:wardrop_syst_aff_dual2}
\begin{align}
\max_{\lambda\geq 0,\pi} & \sum_{k\in\mathcal{K}}{\pi_k\cdot d_k}
-\sum_{a\in A_1}{\frac{1}{2\cdot\alpha_a} \left(\beta_a+\tau_a-\lambda_a\right)^2}\\
\mbox{s.t. }& \pi_k\leq \sum_{a\in p}{\lambda_a}, &k\in \mathcal{K}, p\in \mathcal{P}_k\label{eq:wardrop_syst_aff_dual2_3}\\
& \beta_a+\tau_a-\lambda_a \geq 0 & a\in A_0. 
\end{align}
\end{subequations}
The strong duality between~\eqref{eq:wardrop_syst_aff} 
and~\eqref{eq:wardrop_syst_aff_dual2} stated in the next proposition 
will play a central role in the single-level reformulations 
derived in Section~\ref{sec:slr}.
\begin{proposition}
\label{prop:strong_duality}
Strong duality holds between~\eqref{eq:wardrop_syst_aff} 
and~\eqref{eq:wardrop_syst_aff_dual2}.
\end{proposition}
\begin{proof}
By assumption, $\mathcal{P}_k\neq\emptyset$ for all $k\in\mathcal{K}$,
so the feasible set of~\eqref{eq:wardrop_syst_aff} is nonempty.
Although the feasible set is unbounded due to the relaxation
of~\eqref{eq_wardrop_syst2_2} to an inequality, 
an optimal solution exists that satisfies the latter with equality, 
using the fact that the objective $B(f,\tau)$ is nondecreasing in $f_a$. 
Therefore, 
we may restrict the minimization
to the set
\[
\mathcal{F} = \left\{f \geq 0 \colon
\eqref{eq:wardrop_syst_aff_1},\ \eqref{eq:wardrop_syst_aff_2},\ 
f_a \leq \sum_{k\in\mathcal{K}}d_k,\ a\in A
\right\},
\]
which is nonempty, compact, and polyhedral.
Since $B(f,\tau)$ is continuous, it attains its minimum over $\mathcal{F}$,
and an optimal primal solution $f^\star$ exists. We consider such an optimal solution 
satisfying~\eqref{eq:wardrop_syst_aff_2}   
with equality.  

Define the candidate dual solution by
\begin{subequations}
\label{eq:candidate_dual}
\begin{align}
\lambda_a^\star &= \widehat{c}_a(f_a^\star,\tau_a) 
= \alpha_a f_a^\star + \beta_a + \tau_a,
\quad a\in A,\\
\pi_k^\star &= \min_{p\in\mathcal{P}_k}
\left\{\sum_{a\in p}\widehat{c}_a(f_a^\star,\tau_a)\right\},
\quad k\in\mathcal{K}.
\end{align}
\end{subequations}
We show that $(\lambda^\star,\pi^\star)$ is feasible
for~\eqref{eq:wardrop_syst_aff_dual2}
and that the primal and dual objective values coincide, thus  
establishing strong duality.

\medskip
\noindent\textit{Dual feasibility.}
Since arc costs are nonnegative,
$\lambda_a^\star = \widehat{c}_a(f_a^\star,\tau_a)\geq 0$ for all $a\in A$.
For each $k\in\mathcal{K}$ and $p\in\mathcal{P}_k$,
\[
\sum_{a\in p}\lambda_a^\star
= \sum_{a\in p}\widehat{c}_a(f_a^\star,\tau_a)
= \widehat{c}_{p,k}(f^\star,\tau)
\geq \pi_k^\star,
\]
so constraint~\eqref{eq:wardrop_syst_aff_dual2_3} is satisfied.
For arcs $a\in A_0$, we have $\alpha_a=0$, hence
$\lambda_a^\star = \beta_a+\tau_a$,
so the constraint $\beta_a+\tau_a-\lambda_a^\star\geq 0$
holds with equality.
Thus $(\lambda^\star,\pi^\star)$ is feasible for~\eqref{eq:wardrop_syst_aff_dual2}.

\medskip
\noindent\textit{Equality of objective values.}
For $a\in A_1$, since $\lambda_a^\star = \alpha_a f_a^\star+\beta_a+\tau_a$,
we have $(\beta_a+\tau_a)f_a^\star
= \lambda_a^\star f_a^\star - \alpha_a(f_a^\star)^2$,
while for $a\in A_0$, $\alpha_a=0$ gives
$(\beta_a+\tau_a)f_a^\star = \lambda_a^\star f_a^\star$.
Therefore the primal objective expands as 
\begin{align*}
B(f^\star,\tau)
&= \sum_{a\in A}\left(
    \frac{\alpha_a}{2}(f_a^\star)^2 + (\beta_a+\tau_a)f_a^\star
  \right)\\
&= \sum_{a\in A}\lambda_a^\star f_a^\star
  - \sum_{a\in A_1}\frac{\alpha_a}{2}(f_a^\star)^2\\
&= \sum_{a\in A}\lambda_a^\star f_a^\star
  - \sum_{a\in A_1}\frac{(\lambda_a^\star-\beta_a-\tau_a)^2}{2\alpha_a}.
\end{align*}
It remains to show that
$\sum_{a\in A}\lambda_a^\star f_a^\star
= \sum_{k\in\mathcal{K}}\pi_k^\star d_k$.
Using the Wardrop complementarity conditions~\eqref{eq:wardrop_syst1_3},
which state that $f_{p,k}^\star > 0$ only if
$\widehat{c}_{p,k}(f^\star,\tau) = \pi_k^\star$,
together with the demand constraints~\eqref{eq:wardrop_syst1_1},
we obtain
\begin{align*}
\sum_{k\in\mathcal{K}}\pi_k^\star d_k
&= \sum_{k\in\mathcal{K}}\sum_{p\in\mathcal{P}_k}
    \pi_k^\star f_{p,k}^\star
= \sum_{k\in\mathcal{K}}\sum_{p\in\mathcal{P}_k}
    \widehat{c}_{p,k}(f^\star,\tau)\cdot f_{p,k}^\star\\
&= \sum_{k\in\mathcal{K}}\sum_{p\in\mathcal{P}_k}
    \sum_{a\in p}\widehat{c}_a(f_a^\star,\tau_a)\cdot f_{p,k}^\star
= \sum_{a\in A}\lambda_a^\star
    \sum_{k\in\mathcal{K}}\sum_{\substack{p\in\mathcal{P}_k\colon\\a\in p}}f_{p,k}^\star
= \sum_{a\in A}\lambda_a^\star f_a^\star,
\end{align*}
where the last equality uses the fact that
$f_a^\star = \sum_{k\in\mathcal{K}}\sum_{\substack{p\in\mathcal{P}_k\colon\\a\in p}}f_{p,k}^\star$. 
Hence $B(f^\star,\tau)$ equals the objective value of~\eqref{eq:wardrop_syst_aff_dual2}
at $(\lambda^\star,\pi^\star)$, which, combined with weak duality, proves the claim.
\end{proof}

The following corollary is an immediate consequence of 
Proposition~\ref{prop:strong_duality} and its proof.

\begin{corollary}
\label{cor:dual_from_primal}
Given an optimal solution $f^\star$ of the primal problem~\eqref{eq:wardrop_syst_aff},
the pair $(\lambda^\star,\pi^\star)$ defined by~\eqref{eq:candidate_dual}
is an optimal solution of the dual problem~\eqref{eq:wardrop_syst_aff_dual2}.
Moreover, the dual optimal solution  
admits the following interpretation:
$\lambda_a^\star$ equals the effective arc cost at equilibrium,
and $\pi_k^\star$ equals the shortest-path cost for commodity $k$
under the effective arc costs induced by $f^\star$.
\end{corollary}

For affine nondecreasing arc cost functions, a Wardrop equilibrium can be computed 
by
the Frank--Wolfe algorithm~\cite{PAT_1994} or by convex quadratic programming. We adopt the
latter and solve Beckmann's problem~\eqref{eq:wardrop_syst_aff} directly.
In our implementation, the exponentially many
path-flow variables $\left(f_{p,k}\right)$ of~\eqref{eq:wardrop_syst_aff} are replaced by the
equivalent arc-commodity variables $\left(f_{a,k}\right)$ together with the associated flow
conservation constraints, which yields a formulation of polynomial size.

The primal and dual formulations derived in this section
provide the foundation for the single-level reformulations
of problem~\eqref{eq:BW} developed in the next section.

\section{Single-level reformulations of problem~\eqref{eq:BW}}
\label{sec:slr}
In this section, we derive several single-level reformulations of
problem~\eqref{eq:BW} in the affine separable case, focusing on the
profit-maximization objective $F(f,\tau)=\sum_{a\in A}f_a\tau_a$,
which is the objective considered in~\cite{BLS_2024}
and adopted throughout the remainder of this paper. 
We first recall the mixed-integer nonlinear programming formulation 
recently proposed by Beck et al.~\cite{BLS_2024} 
(Section~\ref{sec:bls_formulation}),
which serves as a benchmark in our computational experiments.
We then introduce three new single-level reformulations.
The first two exploit the primal-dual structure of Wardrop equilibria
established in Section~\ref{sec:wardrop}:
enforcing primal feasibility, dual feasibility, and strong duality simultaneously
eliminates the lower-level equilibrium problem and yields an equivalent 
single-level program. 
An arc-flow formulation and a path-flow formulation are derived in this way,
differing in the choice of flow variables and the structure of the resulting constraints.
The third reformulation takes a different route,
directly encoding the Wardrop equilibrium conditions
through the complementarity system~\eqref{eq:wardrop_syst1},
leading to a formulation that can be handled 
through SOS1 constraints in modern mixed-integer optimization solvers.

\subsection{MINLP reformulation proposed by Beck et al.~\cite{BLS_2024}}
\label{sec:bls_formulation}
Beck et al.~\cite{BLS_2024} propose the following mixed-integer nonlinear
single-level reformulation of the toll-setting problem. 
The formulation introduces, for each commodity $k\in\mathcal{K}$ and arc $a\in A$,
a binary variable $z_{a,k}$ indicating whether arc $a$ may belong to a shortest path
used by commodity $k$, and a variable $t_{i,k}$ representing the shortest-path distance
from node $i$ to the destination $s_k$ for commodity $k$.

\begin{subequations}
\small
\label{eq:form_bls}
\begin{align}
    \max_{\tau,f,t,z}\  &\sum_{a\in A}{f_a\cdot \tau_a}\\
    \mathrm{s.t. } &\ N{f_{\cdot,k}} = b_k, & k\in \mathcal{K}\\
&\sum_{k\in\mathcal{K}}{f_{a,k}} 
=f_a, & a\in A \\
&t_{i,k}-t_{j,k}\leq \alpha_a\cdot f_a + \beta_a + \tau_a & a=(i,j)\in A,\ k\in \mathcal{K} \label{eq:form_bls3}\\
&t_{i,k}-t_{j,k}\geq \alpha_a\cdot f_a + \beta_a + \tau_a - M_{a,k}\left(1-z_{a,k}\right)& a=(i,j)\in A,\ k\in \mathcal{K}\\
&t_{s_k,k}=0&k\in \mathcal{K}\\
&f_{a,k}\leq d_{k}\cdot z_{a,k} &  a\in A,\ k\in \mathcal{K}\label{eq:form_bls5}\\
&\tau \in\mathcal{T},\ f\geq 0,\ t\geq 0 \label{eq:form_bls6}\\
&z_{a,k}\in\{0,1\}&a\in A, k\in \mathcal{K}\label{eq:form_bls7}
\end{align}
\end{subequations}
where $N\in\mathbb{R}^{V\times A}$ is the node-arc incidence matrix of $G$, 
and $b_k\in\mathbb{R}^V$ is the demand vector for commodity $k$, 
with $\left(b_k\right)_{o_k}=d_k$, $\left(b_k\right)_{s_k}=-d_k$, 
and all other entries equal to zero.
The constants 
$M_{a,k}$ are big-M parameters,
for which explicit sufficiently large values are provided in~\cite{BLS_2024}.
Constraints~\eqref{eq:form_bls3}--\eqref{eq:form_bls7} 
ensure that positive flow for commodity $k$ 
can only be routed along shortest $(o_k,s_k)$-paths.

\subsubsection*{Variable bounds and big-M constants}
Due to the nonnegativity of the costs on the arcs, problem~\eqref{eq:form_bls} admits an optimal solution routing
no flow on circuits, which allows several variables to be bounded or fixed
without loss of optimality. We record here 
the bounds used in the sequel and 
 valid inequalities additionally proposed by Beck et al.~\cite{BLS_2024}.

Flow entering an origin or leaving a destination may be set to
zero:
\begin{equation}
\label{eq:flow_bounds}
f_{a,k}=0, \qquad  a\in \delta^{-}\left(o_k\right)\cup\delta^{+}\left(s_k\right),\ k\in \mathcal{K},
\end{equation}
where $\delta^{-}(v)$ (resp.\ $\delta^{+}(v)$) denotes the set of arcs having
node $v$ as terminal (resp.\ initial) endpoint. Combining~\eqref{eq:flow_bounds}
with the demand bound $f_{a,k}\le d_k$ implied by~\eqref{eq:form_bls5} (since
$z_{a,k}\le 1$), the equilibrium arc flow on every arc $a\in A$ satisfies
\begin{equation}
\label{eq:fa_bounds}
0\le f_a\le f_a^{\tt up}
:=\sum_{\substack{k\in\mathcal K\,:\\ a\notin\delta^{-}(o_k)\cup\delta^{+}(s_k)}} d_k .
\end{equation}
The quantity $f_a^{\tt up}$ serves throughout as
the reference arc-flow upper bound, both in~\eqref{eq:form_bls} and in the other
formulations of Section~\ref{sec:slr}.
These 
bounds also yield a finite upper bound on the shortest-path variables:
telescoping~\eqref{eq:form_bls3} along any elementary $(i,s_k)$-path and using
$t_{s_k,k}=0$ together with $f_a\le f_a^{\tt up}$ and
$\tau_a\le\tau^{\tt up}_a:=\max_{\tau\in\mathcal{T}}\tau_a$, one obtains, for
every commodity $k\in\mathcal{K}$ and node $i\in V$,
\begin{equation}
\label{eq:bounds_tik}
0\leq t_{i,k}\leq
\min_{p\in\mathcal{P}^{k}_{i}}\sum_{a\in p}\left(\alpha_a\, f^{\tt{up}}_{a}
+\beta_a+\tau_{a}^{\tt{up}}\right),
\end{equation}
where $\mathcal{P}^{k}_{i}$ denotes the set of elementary $(i,s_k)$-paths in
$G$. These bounds are used to set the big-M constants $M_{a,k}$
in~\eqref{eq:form_bls}, similar to~\cite{BLS_2024}.

\subsubsection*{Valid inequalities for the routing binary variables}

Beck et al.~\cite{BLS_2024} 
 suggest to add two types of inequalities to  \eqref{eq:form_bls}, 
 to improve computational performance. 
The first is
satisfied by every optimal solution, the second holds without loss of
optimality. We report them hereafter. 
For each pair of reverse arcs $(i,j),(j,i)\in A$, every optimal solution
satisfies
\begin{equation}
\label{eq:bls_inverse_arcs}
z_{(i,j),k}+z_{(j,i),k}\leq 1, \qquad k\in\mathcal{K}.
\end{equation}
Moreover, there exists an optimal solution satisfying
\begin{equation}
\label{eq:toll_lb}
\tau_a \geq \tau^{\tt up}_{a}\Bigl(1-\sum_{k\in\mathcal{K}} z_{a,k}\Bigr),
\qquad a\in A,
\end{equation}
i.e., an arc
that no commodity routes along a shortest path (so $z_{a,k}=0$ for all $k$) is
assigned the maximal toll $\tau^{\tt up}_a$.
\subsection{Strong duality based formulations}
\label{sec:sd_formulations}
The strong duality result established in Proposition~\ref{prop:strong_duality}
provides a natural route to eliminate the lower-level equilibrium problem: 
the lower-level constraint $f\in\mathcal{W}(\tau)$ 
is equivalently reformulated by jointly imposing
primal feasibility of~\eqref{eq:wardrop_syst_aff},
dual feasibility of~\eqref{eq:wardrop_syst_aff_dual2},
and 
strong duality, 
yielding an equivalent single-level reformulation.
We present two variants of this approach,
differing in their choice of flow variables:
the first relies on arc-flow variables and flow-conservation constraints,
while the second additionally introduces path-flow variables,
leading to a formulation that is structurally closer to the original Wardrop conditions.

\subsubsection{Arc-based formulation}
Applying Proposition~\ref{prop:strong_duality} to problem~\eqref{eq:BW}
and expressing flow conservation through arc-flow variables
yields the following single-level reformulation.
\begin{subequations}
\label{eq:form_sd_arc}
\begin{align}
    \max_{\tau,f,\pi,\lambda}\  &  \sum_{a\in A}{f_a\cdot \tau_a}  \\
    \mathrm{s.t. } &\ N{f^{k}} = b^k, & k\in \mathcal{K}\\
&\sum_{k\in\mathcal{K}}{f_{a,k}} 
=f_a, & a\in A \\
& \sum_{a\in A}{\left(\frac{\alpha_a}{2}\cdot f_a^2 + \left(\beta_a+\tau_a\right)\cdot f_a\right)}\leq \notag \\ 
&\qquad \sum_{k\in\mathcal{K}}{\pi_k\cdot d_k}
-\sum_{a\in A_1}{\frac{1}{2\cdot\alpha_a} \left(\beta_a+\tau_a-\lambda_a\right)^2} \\
& \pi_k\leq \sum_{a\in p}{\lambda_a}, &k\in \mathcal{K}, p\in \mathcal{P}_k,  \label{eq:sd_arc_path_ineq_gen}\\ 
& \beta_a+\tau_a-\lambda_a \geq 0 & a\in A_0 \\
& 0\leq f_{a,k}\leq d_k & k\in \mathcal{K}, a\in A \\
&\tau \in\mathcal{T}, \lambda\geq 0
\end{align}
\end{subequations}

Corollary~\ref{cor:dual_from_primal} shows that,
at optimality, the dual variables $\lambda_a$
coincide with the corresponding equilibrium arc costs
$\alpha_a f_a + \beta_a + \tau_a$.
These variables can therefore be eliminated from
formulation~\eqref{eq:form_sd_arc}
through the substitutions
\[
\lambda_a = \alpha_a f_a + \beta_a + \tau_a, \qquad a\in A,
\]
leading to the more compact formulation~\eqref{eq:form_sd_arc_small}.

\begin{subequations}
\label{eq:form_sd_arc_small}
\begin{align}
    \max_{\tau,f,\pi,t}\  &t\\
    \mathrm{s.t. } &\ N{f^{k}} = b^k, & k\in \mathcal{K}\\
&\sum_{k\in\mathcal{K}}{f_{a,k}} 
=f_a, & a\in A \\
&t= \sum_{a\in A}{f_a\cdot \tau_a} \label{eq:epi_con}\\
& t+\sum_{a\in A}{\left(\alpha_a\cdot f_a^2+\beta_a\cdot f_a\right)} 
\leq 
\sum_{k\in\mathcal{K}}{\pi_k\cdot d_k} \label{eq:sd_flow_ineq}\\
&\pi_k\leq \sum_{a\in p}{\left(\alpha_a\cdot f_a+\beta_a+\tau_a\right)}, &k\in \mathcal{K}, p\in \mathcal{P}_k,  \label{eq:sd_arc_path_ineq}\\ 
& 0\leq f_{a,k}\leq d_k & k\in \mathcal{K}, a\in A \\
&\tau \in\mathcal{T},t\geq 0
\end{align}
\end{subequations}
A notable structural feature of formulation~\eqref{eq:form_sd_arc_small}
is that all nonconvexity is concentrated in the single bilinear
constraint~\eqref{eq:epi_con}. 
The remaining constraints are either linear
or convex quadratic (constraint~\eqref{eq:sd_flow_ineq}).

\subsubsection{Path-based formulation}
An alternative formulation is obtained by replacing the arc-flow 
conservation constraints by a path-flow representation.
Although the number of paths may be exponentially large,
this formulation can serve as the basis for a column-and-constraint generation approach:
one starts with a restricted set of paths and iteratively adds
new paths together with their associated constraints,
with the aim of identifying high-quality solutions
while keeping the size of the formulation manageable.
The performance of such an approach will be discussed in
Section~\ref{sec:computational_experiments}. 

\begin{subequations}
\label{eq:form_sd_paths_small}
\begin{align}
    \max_{\tau,f,\pi,t}\  &t \\ 
    \mathrm{s.t.} &  \sum_{p\in\mathcal{P}_k}{f_{p,k}} = d_k, & k\in \mathcal{K}\\
&\sum_{k\in\mathcal{K}}{\sum_{\substack{p\in\mathcal{P}_k \colon \\ a\in p}}{f_{p,k}}} 
-f_a= 0, & a\in A \\
&t= \sum_{a\in A}{f_a\cdot \tau_a} \label{eq:path_epi_con}\\
& t+\sum_{a\in A}{\left(\alpha_a\cdot f_a^2+\beta_a\cdot f_a\right)} 
\leq 
\sum_{k\in\mathcal{K}}{\pi_k\cdot d_k} \label{eq:path_sd_flow_ineq}\\
&\pi_k\leq \sum_{a\in p}{\left(\alpha_a\cdot f_a+\beta_a+\tau_a\right)}, &k\in \mathcal{K}, p\in \mathcal{P}_k,  \label{eq:sd_paths_arc_path_ineq}\\ 
&\tau \in\mathcal{T},f\geq 0,t\geq 0
\end{align}
\end{subequations}
\subsection{Complementarity-based formulation}
\label{sec:cs_formulations}
The previous formulations exploit Beckmann's formulation and
strong duality to eliminate the lower-level problem.
An alternative approach consists in directly enforcing
the Wardrop equilibrium conditions through the complementarity
system~\eqref{eq:wardrop_syst1}. 
Recall that a Special Ordered Set of type 1 (SOS1) constraint 
imposed on a pair of variables $(x,y)$ enforces that at most one 
of the two variables is nonzero, which is precisely the structure
of the complementarity constraints~\eqref{eq:cs_compl}. 
SOS1 constraints are supported by several modern 
mixed-integer optimization solvers. 

\begin{subequations}
\label{eq:form_cs}
\begin{align}
    \max_{\tau,f,\pi}\  &\sum_{a\in A}{f_a\cdot \tau_a}\\
    \mathrm{s.t.} & \sum_{p\in\mathcal{P}_k}{f_{p,k}} = d_k & k\in \mathcal{K}\\
&\sum_{k\in\mathcal{K}}{\sum_{\substack{p\in\mathcal{P}_k \colon \\ a\in p}}{f_{p,k}}}
-f_a= 0, & a\in A \\
&\widehat{c}_{p,k}\left(f,\tau\right)-\pi_k\geq 0 &p\in \mathcal{P}_k,\ k\in \mathcal{K}\\
&f_{p,k}\left(\widehat{c}_{p,k}\left(f,\tau\right)-\pi_k\right)=0 &p\in \mathcal{P}_k,\ k\in \mathcal{K} \label{eq:cs_compl}\\
&\pi_k\geq 0 & \ k\in \mathcal{K}\\
&\tau \in\mathcal{T},f\geq 0
\end{align}
\end{subequations}

\section{Bounds and size reduction}
\label{sec:comput_matters}
The formulations introduced in Section~\ref{sec:slr} can be tightened with bounds that
exploit structural properties of the toll-setting problem. When the data permit, 
they can also be reduced in size. This section establishes three distinct families of bounds, 
and one model size reduction.

The first family of bounds (Section \ref{sec:toll_bounds}) is concerned with bounding the toll 
set $\mathcal{T}$. 
In Beck et al.~\cite{BLS_2024}, explicit upper bounds on toll values are assumed,
which may be interpreted as regulatory restrictions but are otherwise arbitrary.
We show that, under the common assumption that every commodity admits
a toll-free alternative path (e.g., \cite{BLMS_2001,LMS_1998,BCN_2024}) --- interpretable
as a route offered by a competitor of the leader --- finite upper bounds on
toll values can be derived endogenously, without any arbitrary choice.
Moreover, as our computational experiments will demonstrate,
the magnitude of these bounds can have a significant impact on computational
performance.

The second (Section \ref{sec:var_bounds}) 
concerns default bounds on the variables appearing in the formulations of
Section~\ref{sec:slr}. These bounds are either valid inequalities, satisfied by every
feasible solution, or constraints that can be added without loss of optimality, in the
sense that at least one optimal solution satisfies them. In both cases they may improve
the quality of continuous relaxations and thereby accelerate the solution process. They 
are added to the new formulations (i.e. the ones of Sections \ref{sec:sd_formulations} and 
\ref{sec:cs_formulations})   
and kept active throughout our computational experiments.

The third family  (Section \ref{sec:beckmann_bound}) 
is of a different nature: an upper bound on the \emph{optimal value}
$z^\star$ of~\eqref{eq:BW}, derived from Beckmann's formulation.  It can be  
used to certify optimality gaps but unlike the previous two it is not added to any 
formulation in our computational experiments. 

Lastly, we show (in Section \ref{sec:grouping}) 
how to reduce the model size for formulations involving
per-commodity arc-flow variables $\left(f_{a,k}\right)$. When the arc costs are
independent of the commodity, commodities sharing a common destination can be merged
without affecting the set of optimal solutions. This grouping reduces the number of
per-commodity arc-flow variables and of the associated coupling constraints.

\subsection{Implicit bounds induced by toll-free alternatives}
\label{sec:toll_bounds}

We now show that 
the toll-free path assumption introduced above 
can be used to derive 
upper bounds on toll values.
Let $A^{\tt toll}\subseteq A$ denote the set of toll arcs, i.e.,
arcs $a\in A$ for which $\tau_a^{\tt up}=\max_{\tau\in\mathcal{T}}\tau_a>0$,
and let $A\setminus A^{\tt toll}$ denote the set of toll-free arcs.
For each commodity $k\in\mathcal K$, let
$\mathcal P_k^{0}\subseteq \mathcal P_k$
denote the set of $(o_k,s_k)$-paths containing no toll arc,
and assume that $\mathcal P_k^{0}\neq\emptyset$ for all $k\in\mathcal K$.

For each commodity $k$, define

\[
\pi_k^\infty
=
\min_{p\in\mathcal P_k^{0}}
\sum_{a\in p}
\left(
\alpha_a\, f_a^{\tt up}
+
\beta_a
\right),
\]
where $f_a^{\tt up}$ is the arc-flow upper bound defined in~\eqref{eq:fa_bounds}.

Finally, let

\[
\tau^\infty
=
\max_{k\in\mathcal K}
\pi_k^\infty.
\]

\begin{proposition}
\label{prop:toll_upper_bound}
Consider problem~\eqref{eq:BW} with the profit-maximization objective
$F(f,\tau)=\sum_{a\in A}f_a\tau_a$.
Under the assumption that $\mathcal P_k^{0}\neq\emptyset$ for every
commodity $k\in\mathcal K$,
problem~\eqref{eq:BW} admits an optimal solution $(\tau,f)$ whose toll vector
satisfies
\[
\tau_a
\le
\tau_a^\infty:=\max(0,\tau^\infty-\beta_a), 
\qquad a\in A^{\tt toll}.
\]
In particular, the toll set $\mathcal{T}$ can be assumed to be bounded
without loss of optimality.
\end{proposition}
\begin{proof}
There exists an optimal solution $(\tau^\star,f^\star)$
of~\eqref{eq:BW} satisfying
\[
f_a^\star
\le
f_a^{\tt up},
\qquad a\in A.
\] 

Indeed, if $f^\star$ carries positive flow along a cycle of zero
effective cost, then, since $\alpha_a,\beta_a,f_a^\star,\tau_a^\star\ge0$
for every arc $a$ on the cycle, the effective cost
$\alpha_a f_a^\star+\beta_a+\tau_a^\star=0$ forces $\tau_a^\star=0$ on
every such arc. Removing the flow along the cycle therefore preserves
both the Wardrop equilibrium conditions~\eqref{eq:wardrop_syst1} and the
leader's objective value. 
Repeating this operation for every such cycle yields an optimal solution whose
per-commodity flows decompose into elementary $(o_k,s_k)$-path flows only, and
thus satisfy the zero-flow conditions~\eqref{eq:flow_bounds}. Summing over
commodities, each routing exactly $d_k$ units, gives $f_a^\star\le f_a^{\tt up}$
as defined in~\eqref{eq:fa_bounds}. 
We henceforth fix such an optimal solution $(\tau^\star,f^\star)$.

Hence, for every non-toll arc $a$,
\[
c_a(f_a^\star)
=
\alpha_a f_a^\star+\beta_a
\le
\alpha_a\, f_a^{\tt up}
+\beta_a.
\]
It follows that, for every commodity $k\in\mathcal K$, the cheapest
toll-free $(o_k,s_k)$-path has effective cost at most $\pi_k^\infty$
under $\tau^\star$, and hence the shortest-path cost satisfies
$\pi_k^\star\le \pi_k^\infty\le\tau^\infty$.

Now suppose that some toll arc $a\in A^{\tt toll}$ satisfies
$\tau_a^\star>\max(0,\,\tau^\infty-\beta_a)$.
Since arc costs are nonnegative, the effective cost of any
$(o_k,s_k)$-path $p$ containing arc $a$ satisfies
\[
\widehat{c}_{p,k}(f^\star,\tau^\star)
\ge
\beta_a + \tau_a^\star
> \max(\beta_a,\,\tau^\infty)
\ge
\tau^\infty
\ge
\pi_k^\star
\]
for every commodity $k$.
Hence every path containing arc $a$ is strictly more expensive than the
shortest-path cost, so the Wardrop conditions imply $f_a^\star=0$.

Define $\tau^{\rm new}$ by setting
$\tau_a^{\rm new}=\max(0,\,\tau^\infty-\beta_a)$
and $\tau_{a'}^{\rm new}=\tau_{a'}^\star$ for all $a'\neq a$.
Since $0\le\tau_a^{\rm new}\le\tau_a^\star$ and $\mathcal{T}$ is
downward closed, $\tau^{\rm new}\in\mathcal{T}$.
We show that $(\tau^{\rm new},f^\star)$ is feasible for~\eqref{eq:BW}
and achieves the same objective value as $(\tau^\star,f^\star)$.

Since only $\tau_a$ was decreased, the cost of every path not containing
$a$ is unchanged, while the cost of every path containing $a$ can only
decrease. In particular, every path $p\in\mathcal{P}_k$ with
$f_{p,k}^\star>0$ avoids arc $a$ (since $f_a^\star=0$) and thus retains
its cost $\pi_k^\star$. Moreover, every path containing $a$ has cost at
least $\beta_a+\tau_a^{\rm new}=\max(\beta_a,\,\tau^\infty)\ge\pi_k^\star$
under $\tau^{\rm new}$. Hence $\pi_k^\star$ remains the shortest-path
cost under $\tau^{\rm new}$, and $f^\star$ still satisfies the Wardrop
conditions~\eqref{eq:wardrop_syst1}, so that
$f^\star\in\mathcal{W}(\tau^{\rm new})$.

Since $f_a^\star=0$, the leader's objective satisfies
$F(f^\star,\tau^{\rm new})
=\sum_{a'\in A}f_{a'}^\star\tau_{a'}^{\rm new}
=\sum_{a'\in A}f_{a'}^\star\tau_{a'}^\star
=F(f^\star,\tau^\star)$.
Therefore $(\tau^{\rm new},f^\star)$ is a feasible solution
of~\eqref{eq:BW} achieving the optimal value, and is thus itself optimal.

Repeating this operation for every toll arc violating the bound yields an
optimal solution satisfying
\[
\tau_a
\le
\max\left(0,\,\tau^\infty-\beta_a\right),
\qquad a\in A^{\tt toll}.
\]
This proves the result.
\end{proof}
\begin{remark}
The argument above does not use the optimality of $(\tau^\star,f^\star)$: applied to
any feasible solution $(\tau,f)$ of~\eqref{eq:BW}, it yields a feasible solution with
the same objective value whose tolls satisfy $\tau_a\le\tau_a^\infty$ for all
$a\in A^{\tt toll}$. 
Hence, under the assumption that $\mathcal P_k^{0}\neq\emptyset$ for every
$k\in\mathcal K$, the compactness assumption on $\mathcal{T}$ made in the
introduction can be relaxed: when $\mathcal{T}$ is closed, convex and downward
closed but not necessarily bounded, the supremum of~\eqref{eq:BW} over
$\mathcal{T}$ equals its supremum over the compact set
$\mathcal{T}\cap\{\tau\colon \tau_a\le\tau_a^\infty,\ a\in A^{\tt toll}\}$.
\end{remark}
\subsection{Default variable bounds}
\label{sec:var_bounds}

Throughout this subsection and Section~\ref{sec:beckmann_bound}, we assume that the toll
set $\mathcal{T}$ is bounded, so that $\tau_a^{\tt up}:=\max_{\tau\in\mathcal{T}}\tau_a$ is
finite for every $a\in A$. This entails no loss of optimality: by
Proposition~\ref{prop:toll_upper_bound}, an unbounded toll set can be replaced, without
affecting the optimal value, by its intersection with the bounds
$\tau_a\le\max(0,\tau^\infty-\beta_a)$. We also use the arc-flow upper bound
$f_a^{\tt up}$ defined in~\eqref{eq:fa_bounds}, which holds without loss of optimality
(see the proof of Proposition~\ref{prop:toll_upper_bound}).

All formulations of Section~\ref{sec:slr} that involve commodity-specific
arc-flow variables are equipped, by default, with the native bounds of the
formulation~\eqref{eq:form_bls} of~\cite{BLS_2024}. 
First, the zero-flow 
conditions~\eqref{eq:flow_bounds} are imposed for every commodity. Second, for
every arc $a\in A$ and commodity $k\in\mathcal K$,
\begin{equation}
\label{eq:fak_ub}
0\le f_{a,k}\le d_k .
\end{equation}
Together
with~\eqref{eq:flow_bounds}, these bounds yield the arc-flow bound
$0\le f_a\le f_a^{\tt up}$ of~\eqref{eq:fa_bounds}. 

\begin{remark}
Tighter upper bounds on the arc flows could in principle be computed by
maximizing $f_a$ over the feasible set of~\eqref{eq:form_bls} for each arc
$a\in A$. Since the objective $\sum_{a\in A}f_a\tau_a$ is then replaced by the
linear objective $f_a$, each such subproblem is a mixed-integer \emph{linear}
program. Preliminary experiments, however, showed that the computational cost
of solving these subproblems outweighs the benefit of the resulting tighter
bounds, especially on the larger instances. We therefore do not pursue this
approach in our computational study.
\end{remark}

The shortest-path variables $\pi_k$ are 
bounded by
\begin{equation}
\label{eq:bounds_pik}
0
\le
\pi_k
\le
\min_{p\in\mathcal P_k}
\sum_{a\in p}
\left(
\alpha_a f_a^{\tt up}
+
\beta_a
+
\tau_a^{\tt up}
\right),
\qquad
k\in\mathcal K,
\end{equation}
the right-hand side upper-bounding the effective cost of the shortest
$(o_k,s_k)$-path obtained by replacing each arc flow and toll by $f_a^{\tt up}$
and $\tau_a^{\tt up}$, respectively. 
Inequality~\eqref{eq:bounds_pik} plays, for the shortest-path value
$\pi_k$, the role that the node-potential bounds~\eqref{eq:bounds_tik} on the
variables $t_{i,k}$ play in (BLS). Like the flow bounds above, these bounds
are imposed in every 
formulation carrying the corresponding variables 
for all our computational experiments.

\subsection{A Beckmann-based upper bound on the optimal value}
\label{sec:beckmann_bound}

The convex structure of Beckmann's formulation yields 
an upper bound on the optimal value $z^\star$
of~\eqref{eq:BW}. 
Unlike the bounds of Section~\ref{sec:var_bounds}, this bound is not
added to the formulations in our computational experiments; it is used only 
to certify optimality gaps.

Let $\Phi(\tau)$ denote the optimal objective value of
problem~\eqref{eq:wardrop_syst_aff} for a fixed toll vector $\tau$.
Since the objective $B(f,\tau)$ is nondecreasing in $\tau$ for every fixed
$f\ge0$, the value function $\Phi$ is nondecreasing as well:
\[
\tau^1 \le \tau^2
\quad\Longrightarrow\quad
\Phi(\tau^1)\le \Phi(\tau^2).
\]

In any feasible solution $(\tau,f)$ of~\eqref{eq:BW}, the flow $f$ is a
Wardrop equilibrium and therefore minimizes $B(\cdot,\tau)$, so that
$B(f,\tau)=\Phi(\tau)$. Combined with the monotonicity of $\Phi$ and the
upper bound $\tau\le\tau^{\tt up}:=(\tau_a^{\tt up})_{a\in A}$, this yields
the valid inequality
\begin{equation}
\label{eq:Beckmann_ub}
B(f,\tau)
\;=\;
\Phi(\tau)
\;\le\;
\Phi(\tau^{\tt up}).
\end{equation}

For formulation~\eqref{eq:form_sd_arc_small},
inequality~\eqref{eq:Beckmann_ub} can be rewritten as
\begin{equation}
\label{eq:Beckmann_ub_sd}
t
+
\sum_{a\in A}
\left(
\frac{\alpha_a}{2}f_a^2
+
\beta_a f_a
\right)
\le
\Phi(\tau^{\tt up}).
\end{equation}

Moreover, since $f$ is feasible for problem~\eqref{eq:wardrop_syst_aff}
with zero tolls, we have
\[
\Phi(0)
\le
B(f,0)
=
\sum_{a\in A}
\left(
\frac{\alpha_a}{2}f_a^2
+
\beta_a f_a
\right),
\]
and we deduce the following upper bound on the leader's optimal objective value $z^\star$:
\begin{equation}
\label{eq:upper_bound}
z^\star\;\le\;\overline{z}_{\tt B}\;:=\;\Phi(\tau^{\tt up})-\Phi(0).
\end{equation}

\begin{remark}
\label{rem:beckmann_vs_product}
The bound~\eqref{eq:upper_bound} refines the elementary product bound induced
by~\eqref{eq:fa_bounds} alone. Indeed, since $\tau_a\le\tau_a^{\tt up}$ and
$f_a\le f_a^{\tt up}$ at every feasible point of~\eqref{eq:BW}, the leader's
revenue satisfies
\[
\sum_{a\in A} f_a\tau_a
\;\le\;
\sum_{a\in A} f_a^{\tt up}\,\tau_a^{\tt up}.
\]
The Beckmann-based bound~\eqref{eq:upper_bound} is 
always at least as tight as this, and is
typically much tighter. To see this, let $f^0$ be a Wardrop equilibrium at zero
tolls. Being a pointwise minimum over $f$ of the maps $\tau\mapsto B(f,\tau)$,
each affine in $\tau$, the value function $\Phi$ is concave, and any such $f^0$ is
a supergradient of $\Phi$ at $\tau=0$ (since $\nabla_\tau B(f,\tau)=f$). The
resulting first-order inequality $\Phi(\tau^{\tt up})\le\Phi(0)+\sum_{a\in A}
f_a^0\,\tau_a^{\tt up}$ gives
\[
\Phi(\tau^{\tt up})-\Phi(0)
\;\le\;
\sum_{a\in A} f_a^0\,\tau_a^{\tt up}
\;\le\;
\sum_{a\in A} f_a^{\tt up}\,\tau_a^{\tt up},
\]
the last inequality following from $f_a^0\le f_a^{\tt up}$. The product bound
treats flow and toll as independently maximizable on each arc and thus ignores
the equilibrium response: it grows linearly in the magnitude of the toll bounds,
whereas the value of~\eqref{eq:upper_bound} remains bounded, since raising the
tolls only diverts flow onto the toll-free paths guaranteed by
$\mathcal P_k^{0}\neq\emptyset$. 
\end{remark}

\subsection{Grouping commodities by destination}
\label{sec:grouping}

The arc-based formulations of Section~\ref{sec:slr} carry one flow variable
$f_{a,k}$ per arc $a$ and commodity $k$. When the arc costs do not depend on
the commodity---as is the case in \cite{BLS_2024} and our computational experiments 
---commodities sharing a common destination can be
merged without affecting the set of optimal solutions, which reduces the number
of flow variables and of the associated coupling constraints. We therefore group
commodities by destination: a grouped commodity $k$ has a single destination
$s_k$ and a set of origins $O_k$, its demand vector $b^{k}$ having entries
$(b^{k})_{v}=d_{k,v}$ for $v\in V\setminus\left\{s_k\right\}$ and $(b^{k})_{s_k}=-d_k$, 
writing $d_{k,v}$ for the demand of commodity 
$k$ issued from node $v$ (with $d_{k,v}=0$ when $v\notin O_k$ and $d_{k,v}>0$ otherwise), and 
$d_k=\sum_{o\in O_k} d_{k,o}$ the total demand. 

The default variable bounds of Section~\ref{sec:var_bounds} can be adjusted to this grouping setting. 
The destination-side condition of~\eqref{eq:flow_bounds} (no flow leaves $s_k$)
remains valid for every grouped commodity. 
The origin-side condition (no flow enters the origin) can no
longer be imposed verbatim once a commodity has several origins, since flow may
enter one origin to serve a demand of the group issued from another node. It is
replaced by the following bound: 
every arc
$a=(u,v)$ satisfies
\begin{equation}
\label{eq:fak_grouped}
f_{a,k}\ \le\
\begin{cases}
0, & u=s_k,\\[2pt]
d_k-d_{k,v}, & u\ne s_k,
\end{cases}
\qquad k\in\mathcal K.
\end{equation}
Indeed, by the argument in the proof of Proposition~\ref{prop:toll_upper_bound},
problem~\eqref{eq:BW} admits an optimal solution whose flow is a Wardrop
equilibrium whose per-commodity flows decompose into elementary
$(o_k,s_k)$-path flows. 
For a single-origin commodity, \eqref{eq:fak_grouped}
recovers the origin-side condition of~\eqref{eq:flow_bounds} and the bound
$f_{a,k}\le d_k$; for a multi-origin commodity it tightens the latter to
$d_k-d_{k,v}$ whenever $v$ is one of its origins. 
Summing~\eqref{eq:fak_grouped} over commodities gives the arc-flow
bound~\eqref{eq:fa_up_grouped}, the multi-origin counterpart
of~\eqref{eq:fa_bounds}; the derived bounds~\eqref{eq:bounds_tik}
and~\eqref{eq:bounds_pik} are then evaluated from this expression of
$f_a^{\tt up}$. 
\begin{equation}
\label{eq:fa_up_grouped}
f_a\ \le\ f_a^{\tt up}\ =\ \sum_{k\in\mathcal K\,:\ s_k\ne u}\bigl(d_k-d_{k,v}\bigr),
\qquad a=(u,v).
\end{equation} 
\section{Computational experiments}
\label{sec:computational_experiments}
\subsection{Experimental setup}
\label{sec:exp_setup}

All experiments were carried out on a machine equipped with a
13th~Gen Intel  
Core\texttrademark{} i7-1355U processor
(1.70~GHz) and 32~GB of RAM, running a 64-bit operating system.
The formulations were solved with Gurobi~13.0.2 \cite{gurobi}, exploiting its ability to
handle convex quadratic constraints as well as nonconvex bilinear terms;
the latter is enabled by setting the parameter \texttt{NonConvex} to~2.
Gurobi solves the resulting nonconvex quadratically constrained problems
using a spatial branch-and-bound framework combined with outer
approximation techniques. 
All other parameters are left at their default values. Because the
i7-1355U has a hybrid architecture (two performance and eight efficiency
cores), we fix the number of threads to~4 (\texttt{Threads}~$=4$) and pin
the solver to the four logical processors of the two performance cores, so
that every run executes on identical cores. 
Wardrop equilibria are computed by a deterministic single-threaded barrier method. 
They arise in the computation of the bounds and, for the constraint-generation formulations, 
at each iteration, where the equilibrium under the current toll vector is recomputed to turn the 
relaxation flow into a bilevel-feasible incumbent. All these computations are included in the 
reported solution times; each is solved in at most a few seconds, even on the largest (Eastern 
Massachusetts) instances. 
To further improve the
reproducibility of the reported solve times, Turbo Boost was disabled,
fixing the performance cores at their nominal base frequency of 1.70~GHz. 
The relative optimality gap tolerance is left at its default
value of $0.01\%$, so that a run is considered solved to optimality once this
tolerance is reached. As a consequence, objective values reported for
different formulations on the same instance may differ slightly.
Unless stated otherwise, a time limit of 1800~seconds is imposed on each run.

\paragraph{Instances.}
Our benchmark is built from three road networks of increasing size, all drawn from
the \texttt{TransportationNetworks} data repository \cite{TransportationNetworks}. 
Throughout
Section~\ref{sec:computational_experiments} we essentially work with a subnetwork of the Sioux
Falls network~\cite{LMP_1975}, composed of $12$ nodes and $36$ arcs and already
considered in~\cite{BLS_2024}; we turn to the two larger networks --- the
\emph{complete} Sioux Falls network ($24$ nodes, $76$ arcs) and the Eastern
Massachusetts network ($74$ nodes, $258$ arcs) --- only to assess how the
formulations scale with the topology. 
For each network we vary the number of
toll arcs in $|A^{\rm toll}|\in\{2,4,6,8,10\}$ and the number of commodities along a
nested demand ladder, where the level $|\mathcal{K}|=k$ uses the first $k$ demands in a
fixed per-network ordering, so that smaller demand levels are contained in larger ones. Further 
details are given in Appendices \ref{app:instance_Sioux}, \ref{sec:app_bignet} and \ref{sec:EAM_net}.

\paragraph{Grouping commodities by destination.}
Following Section~\ref{sec:grouping}, commodities are grouped by destination in
all the considered formulations; this is possible here because the arc costs of our
instances do not depend on the commodity. Since origins and destinations play
symmetric roles once arcs and OD pairs are reversed, we assume without loss of
generality that the number of distinct destinations does not exceed the number
of distinct origins, reversing arcs and OD pairs whenever this is not the case. 
The reduction is applied uniformly to all formulations, ensuring a fair
comparison; in particular, the formulation~\eqref{eq:form_bls} of Beck et al. is
used in its grouped-by-destination form. This form is equivalent to the original
one in terms of optimal solutions, while offering a computational advantage over
the ungrouped version that grows with instance size, so that our comparison is
not biased against~\cite{BLS_2024}. 

\paragraph{Formulations under comparison.}
We compare four formulations, which fall into two groups according to the
guarantees they provide:
\begin{itemize}
\item Two \emph{exact} formulations, which solve problem~\eqref{eq:BW} to
global optimality: the formulation of Beck et al.~\eqref{eq:form_bls} and the
arc-based strong-duality formulation~\eqref{eq:form_sd_arc_small}.
\item Two \emph{path-based} formulations, namely the path-flow
formulation~\eqref{eq:form_sd_paths_small} and the complementarity-based
formulation~\eqref{eq:form_cs}. Both involve a number of path variables and
associated constraints that grows exponentially with the network size, and
are therefore solved by column-and-constraint generation. Starting from a
restricted set of represented paths, one iteratively solves the current
restriction and searches for violated path constraints; whenever such a
constraint is found, both the constraint and the path variable associated
with the corresponding path are added to the formulation. In our
implementation, the set of paths represented in the constraints coincides
with the set of paths represented by the variables.  
The procedure terminates when no violated path inequality remains. Since the
restricted models are nonconvex, this termination certifies feasibility but not
optimality: it does not rule out the existence of a path whose addition, as a
variable, would improve the objective. The returned solution is therefore not
guaranteed to be optimal for~\eqref{eq:BW}, and these two formulations are used
as heuristics (see Section~\ref{sec:exp_heuristic}). 
\end{itemize}
 The bounds from Sections \ref{sec:toll_bounds} and 
\ref{sec:var_bounds} are always included in the formulations whenever applicable (i.e. if the bounded 
variables appear in the formulation)  for our experiments. 

\paragraph{Cost structure for toll arcs.} 
Unless stated otherwise, 
the cost structure on toll arcs
follows~\cite{BLS_2024} (see also Appendix \ref{app:instance_Sioux}), and the toll upper bound on each toll arc
$a\in A^{\tt toll}$ is set to
\[
\tau_a^{\tt up}=\min\!\left(c_a^{\tt fft},\,\tau_a^\infty\right),
\]
the smaller of the free-flow arc cost $c_a^{\tt fft}$ used as toll bound
in~\cite{BLS_2024} and the implicit bound $\tau_a^\infty$ of
Proposition~\ref{prop:toll_upper_bound}. (Here $c_a^{\tt fft}$ is the original
free-flow cost serving as the regulatory factor, whereas the intercept
$\beta_a$ entering $\tau_a^\infty$ is that of the toll arc's cost function,
equal to $c_a^{\tt fft}/2$ unless stated otherwise.) 
By Proposition~\ref{prop:toll_upper_bound}, taking this minimum is 
optimality-preserving. 
The same
bound is imposed on every formulation. 
We refer to this as the \emph{reference configuration}.
Other choices for the definition of bounds and costs on toll arcs will be investigated in 
Sections \ref{sec:exp_bounds} and \ref{sec:exp_costs}, respectively.

\subsection{Comparison of the exact formulations}
\label{sec:exp_exact}

We first compare the two exact formulations, namely the formulation
(BLS)~\eqref{eq:form_bls} of Beck et al.\ and our arc-based strong-duality
formulation (SD-A)~\eqref{eq:form_sd_arc_small}, on the reference
configuration. 
\subsubsection*{Solving (SD-A) by constraint generation}

We solve the arc-based formulation (SD-A)~\eqref{eq:form_sd_arc_small} by
generating the path inequalities~\eqref{eq:sd_arc_path_ineq} dynamically, since
their number grows exponentially with the network size. We start from a
restriction of~\eqref{eq:form_sd_arc_small} in which only a subset of the
inequalities~\eqref{eq:sd_arc_path_ineq} is present, namely those associated,
for each commodity, with a shortest path computed with respect to the effective 
arc costs $\alpha_a f_a+\beta_a+\tau_a$ of the Wardrop equilibrium at the two
toll vectors $\tau_a=0$ ($a\in A$) and $\tau_a=\tau_a^{\tt up}$ ($a\in A$). This
restriction is then solved to optimality, yielding a solution
$(\tau^\star,f^\star,\pi^\star,t^\star)$; we write
$\lambda_a^\star=\alpha_a f_a^\star+\beta_a+\tau_a^\star$ for the associated
effective arc costs, $a\in A$.

Since the inequalities~\eqref{eq:sd_arc_path_ineq} read
$\pi_k\le\sum_{a\in p}\lambda_a$, the most violated one for commodity $k$ is
found by a single shortest-path computation with respect to the costs
$\lambda_a^\star$. 
Two cases arise.

\emph{Case 1: no inequality is violated}. 
Then $(\tau^\star,f^\star,\pi^\star,t^\star)$ satisfies all
inequalities~\eqref{eq:sd_arc_path_ineq} and is thus feasible
for~\eqref{eq:form_sd_arc_small}; since the restriction is a relaxation
of~\eqref{eq:form_sd_arc_small}, it is optimal for it, so that $f^\star$ is a Wardrop
equilibrium for $\tau^\star$ and $(\tau^\star,f^\star)$ is optimal for~\eqref{eq:BW}.
The procedure stops.

\emph{Case 2: at least one inequality is violated} 
for some commodity $k$. For each such $k$ we recover a
minimizing path and add the corresponding inequality~\eqref{eq:sd_arc_path_ineq}
to the restriction, which is solved anew, and the procedure iterates.

The results are reported in Table~\ref{tab:rq1}, with the following information.
The fields $\left|\mathcal{K}\right|$, $(\#o,\#d)$, and
$\left|A^{\tt toll}\right|$ describe the instance:
\begin{itemize}
\item $\left|\mathcal{K}\right|$: number of OD pairs;
\item $(\#o,\#d)$: number of distinct origins, followed by the number of
distinct destinations, in the instance as given, i.e.\ before the possible
reversal of arcs and OD pairs described in Section~\ref{sec:exp_setup}; 
\item $\left|A^{\tt toll}\right|$: the number of toll arcs, i.e., arcs $a\in A$
with $\tau_a^{\tt{up}}>0$.
\end{itemize}
The remaining fields report, for each formulation:
\begin{itemize}
\item objective: a valid lower bound $\underline{z}$ on the optimal value
$z^\star$ of~\eqref{eq:BW}. When the instance is solved to optimality,
$\underline{z}=z^\star$. When the time limit is reached, $\underline{z}$ is a
feasible objective value for~\eqref{eq:BW}, obtained as follows. For (BLS),
which is solved as a single model, the best feasible solution returned by the
solver is already feasible for~\eqref{eq:BW}, and its objective is reported. For
(SD-A), which is solved by constraint generation, the flow of the current
restriction need not be a Wardrop equilibrium; we therefore compute a Wardrop
equilibrium for the tolls of the best feasible solution of the restriction
(by convex quadratic programming) 
and report the corresponding leader's
revenue, which is feasible for~\eqref{eq:BW}. A dash (--) indicates that no
feasible solution was found.
\item time: the total solution time in seconds; ``TL'' indicates that the time
limit was reached.
\item nodes: the number of branch-and-bound nodes explored, as reported by the
solver. For (SD-A), this is the cumulative number over all constraint-generation
iterations.
\item gap$_{\tt own}$: the relative optimality gap certified by the formulation
\emph{alone}, defined as
$\left(\overline{z}_{\tt dual}-\underline{z}\right)/\underline{z}$, where
$\overline{z}_{\tt dual}$ is the dual bound returned by the solver at termination
for that formulation. For (BLS), solved as a single model,
$\overline{z}_{\tt dual}$ is the dual bound of the full model. For (SD-A), it is
the dual bound of the last restriction solved; since that restriction is
obtained from~\eqref{eq:form_sd_arc_small} by omitting a subset of the path
inequalities~\eqref{eq:sd_arc_path_ineq}, it is a relaxation
of~\eqref{eq:form_sd_arc_small}, and its dual bound is therefore a valid upper
bound on $z^\star$. Both quantities are thus valid upper bounds on $z^\star$, and
gap$_{\tt own}$ is a valid certificate for the corresponding formulation taken in
isolation.
\item gap: a certified relative optimality gap, defined as
\[
gap=\frac{\overline{z}-\underline{z}}{\underline{z}},
\qquad 
\overline{z}=\min\!\left(\overline{z}^{\,\tt BLS}_{\tt dual},\
\overline{z}^{\,\tt SD\text{-}A}_{\tt dual},\ \overline{z}_{\tt B}\right),
\]
where $\overline{z}^{\,\tt BLS}_{\tt dual}$ and $\overline{z}^{\,\tt SD\text{-}A}_{\tt dual}$
are the dual bounds returned by the solver at termination for the exact formulations
(BLS) and (SD-A), and $\overline{z}_{\tt B}$ is the Beckmann-based upper
bound~\eqref{eq:upper_bound}. Each is a valid upper bound on $z^\star$ for the given
instance, hence so is their minimum. 

\end{itemize}
For (SD-A), we additionally report:
\begin{itemize}
\item $it$: the number of constraint-generation iterations;
\item $\#p_i$, $\#p_g$: the number of path inequalities~\eqref{eq:sd_arc_path_ineq}
present after initialization, and the number generated subsequently by
separation, respectively.
\end{itemize}

\medskip
\noindent\textbf{Main observations.}
The results highlight a marked difference in scalability between the two
formulations. Within the time limit, (BLS) solves 33 of the 55 instances and
(SD-A) solves 52. The formulation (BLS) solves every instance with up to 7
commodities and still solves four of the five instances for
$\left|\mathcal{K}\right|\in\{8,10,15\}$, but only one of five for $\left|\mathcal{K}\right|=20$ and
none for $\left|\mathcal{K}\right|\in\{25,30,132\}$; on eight instances ---
$\left|\mathcal{K}\right|=30$ with 2, 4 and 8 toll arcs, and all five instances with
$\left|\mathcal{K}\right|=132$ --- it returns no feasible solution at all within the time
budget. In contrast, (SD-A) returns a feasible solution on every instance and
solves all five instances with $\left|\mathcal{K}\right|=132$, in between 5.4 and 36.0
seconds. On the 32 instances solved by both formulations, (SD-A) is faster on 20, with a
median solution time of 7.4~s against 18.4~s for (BLS) and a median speed-up of
$1.5\times$. This advantage is not uniform: on the instance $\left|\mathcal{K}\right|=8$
with 6 toll arcs, (SD-A) is markedly slower, taking 1268~s against 874~s for
(BLS), and this single instance accounts for most of its cumulative time
(1817~s, against 3790~s for (BLS)).

The two gap columns reveal a second, more qualitative difference. On every
instance solved to optimality, each formulation certifies its own solution:
gap$_{\tt own}$ is then at most $0.01\%$, i.e.\ the solver tolerance. On the
time-limited runs, however, the two formulations behave in opposite ways. For
(BLS), gap$_{\tt own}$ ranges from $0.051\%$ to $111.4\%$, with a median of
$21.3\%$, whereas the pooled gap never exceeds $0.881\%$ and is at most $0.01\%$
on 12 of the 14 such runs: (BLS) frequently reaches an optimal, or near-optimal,
solution, but its own dual bound remains too loose to certify it. Its
difficulty on these instances is therefore one of certification rather than of
primal search. For (SD-A), on the contrary, gap$_{\tt own}$ coincides with the
pooled gap on two of its three time-limited runs --- $\left|\mathcal{K}\right|=8$ and
$\left|\mathcal{K}\right|=30$, both with 10 toll arcs, at $0.985\%$ and $0.281\%$ --- and
pooling improves it only on the third, $\left|\mathcal{K}\right|=7$ with 10 toll arcs, from
$0.023\%$ to $0.010\%$. The strong-duality formulation is thus essentially
self-certifying: pooling the dual bounds improves its gaps only marginally, so
that the gaps reported for (SD-A) are not an artefact of borrowing bounds from
(BLS).

The constraint-generation procedure terminates in a handful of iterations. The
shortest-path initialization provides a single path per commodity for each of
the two reference toll vectors $\tau=0$ and $\tau=\tau^{\tt up}$; being
deliberately sparse, it leaves at least one violated inequality on every
instance of the benchmark, so that separation is triggered throughout. The exact
separation nonetheless terminates the procedure quickly: at most 11 path
inequalities are generated on any instance, in at most 5 iterations. As each
separation round amounts to one reverse-Dijkstra computation per distinct
destination, the cost of (SD-A) is dominated by the solution of the successive
restrictions --- each a nonconvex quadratically constrained program --- rather
than by the separation itself.

\medskip
We emphasize that (SD-A) does not dominate (BLS) uniformly. On $\left|\mathcal{K}\right|=7$
with 10 toll arcs, (BLS) is solved within the time limit (in $1739.1$~s) while
(SD-A) reaches it; on $\left|\mathcal{K}\right|=8$ with 6 toll arcs, both formulations solve
the instance to optimality, but (BLS) is the faster of the two ($874.3$~s against
$1268.4$~s); and on $\left|\mathcal{K}\right|=8$ with 10 toll arcs, where both reach the
time limit, (BLS) returns the better feasible solution ($38\,825.91$ against
$38\,785.82$, an advantage of about $0.1\%$). Conversely, (SD-A) is the stronger
method on the  
$\left|\mathcal{K}\right|=30$ instances: with 8 toll arcs it solves the
instance to optimality (in $439$~s) whereas (BLS) finds no feasible solution,
and with 10 toll arcs, where both time out, it returns the better feasible
solution ($107\,221.60$ against $107\,075.08$). Moreover, the three instances on
which (SD-A) reaches the time limit (highlighted in bold in Table \ref{tab:rq1}) 
are not the largest ones: (SD-A) solves
$\left|\mathcal{K}\right|=132$ with 10 toll arcs in $27.7$~s, exploring only $225$
branch-and-bound nodes, whereas $\left|\mathcal{K}\right|=8$ with 10 toll arcs requires more
than $3.7$ million nodes without closing the gap. The relative performance of the
two formulations, and the difficulty of an instance for (SD-A), therefore depend
on the structure of the instance rather than on its size alone.

\begin{sidewaystable}
\caption{\label{tab:rq1} Comparison of the exact formulations (BLS) and (SD-A)
(time limit TL${}=1800$\,s). The ``objective'' field reports a valid lower bound
$\underline{z}$ on $z^\star$; ``gap$_{\tt own}$'' the gap certified by the
formulation alone, and ``gap'' the certified gap
$(\overline{z}-\underline{z})/\underline{z}$ obtained with the pooled
instance-level bound $\overline{z}$; see the field descriptions above for
details. A dash (--) indicates that no feasible solution was found.}
\tiny
\centering
\begin{tabular}{|c|c|c||c|c|c|c|c||c|c|c|c|c|c|c|c|}
\hline
\multirow{2}{*}{$\left|\mathcal{K}\right|$} &
\multirow{2}{*}{$(\#o,\#d)$} &
\multirow{2}{*}{$\left|A^{\tt toll}\right|$} &
\multicolumn{5}{|c||}{${\tt{(BLS)}}$} &
\multicolumn{8}{|c|}{${\tt{(SD\text{-}A)}}$}
\\
\cline{4-16}
& & & objective & time & nodes & gap$_{\tt own}$ (\%) & gap (\%)
& objective & time & $\#p_i$ & $\#p_g$ & nodes & it & gap$_{\tt own}$ (\%) & gap (\%)\\
\hline
\multirow{5}{*}{4} & \multirow{5}{*}{(4,3)}
& 2  & 11\,787.08 & 2.7 & 5\,007 & 0.010 & 0.000 & 11\,787.08 & 3.6 & 4 & 2 & 3 & 3 & 0.000 & 0.000 \\
& & 4  & 16\,287.08 & 2.7 & 3\,944 & 0.008 & 0.000 & 16\,287.08 & 4.7 & 4 & 2 & 3 & 3 & 0.000 & 0.000 \\
& & 6  & 18\,987.08 & 1.8 & 982 & 0.000 & 0.000 & 18\,987.08 & 4.7 & 4 & 2 & 3 & 3 & 0.010 & 0.000 \\
& & 8  & 23\,870.26 & 6.8 & 1\,780 & 0.004 & 0.004 & 23\,870.26 & 6.7 & 5 & 1 & 2\,921 & 2 & 0.009 & 0.004 \\
& & 10  & 26\,212.70 & 24.1 & 34\,751 & 0.007 & 0.003 & 26\,212.69 & 3.7 & 5 & 4 & 961 & 3 & 0.003 & 0.003 \\
\hline
\multirow{5}{*}{5} & \multirow{5}{*}{(5,4)}
& 2  & 11\,787.08 & 1.2 & 1\,083 & 0.000 & 0.000 & 11\,787.08 & 5.6 & 5 & 2 & 3 & 3 & 0.000 & 0.000 \\
& & 4  & 16\,287.08 & 1.2 & 15 & 0.000 & 0.000 & 16\,287.08 & 5.1 & 5 & 2 & 3 & 3 & 0.000 & 0.000 \\
& & 6  & 18\,987.08 & 9.7 & 14\,056 & 0.000 & 0.000 & 18\,987.08 & 4.7 & 5 & 2 & 3 & 3 & 0.008 & 0.000 \\
& & 8  & 23\,870.26 & 8.6 & 2\,177 & 0.001 & 0.001 & 23\,870.26 & 8.1 & 6 & 1 & 4\,069 & 2 & 0.009 & 0.001 \\
& & 10  & 26\,603.80 & 18.4 & 22\,183 & 0.007 & 0.000 & 26\,603.80 & 5.2 & 7 & 4 & 1\,314 & 3 & 0.000 & 0.000 \\
\hline
\multirow{5}{*}{6} & \multirow{5}{*}{(6,5)}
& 2  & 11\,743.58 & 3.5 & 1\,988 & 0.003 & 0.003 & 11\,743.58 & 5.8 & 6 & 5 & 4 & 4 & 0.009 & 0.003 \\
& & 4  & 19\,199.33 & 5.9 & 4\,099 & 0.005 & 0.005 & 19\,199.33 & 20.3 & 6 & 6 & 8\,151 & 5 & 0.008 & 0.005 \\
& & 6  & 22\,799.33 & 6.4 & 1\,509 & 0.007 & 0.005 & 22\,799.33 & 22.9 & 6 & 6 & 2\,045 & 5 & 0.005 & 0.005 \\
& & 8  & 30\,343.02 & 59.8 & 53\,522 & 0.006 & 0.006 & 30\,343.01 & 26.3 & 8 & 4 & 24\,142 & 3 & 0.010 & 0.006 \\
& & 10  & 39\,619.54 & 118.3 & 143\,176 & 0.009 & 0.009 & 39\,619.54 & 178.1 & 8 & 4 & 137\,951 & 3 & 0.010 & 0.009 \\
\hline
\multirow{5}{*}{7} & \multirow{5}{*}{(7,6)}
& 2  & 8\,593.37 & 4.4 & 1\,744 & 0.004 & 0.004 & 8\,593.37 & 6.3 & 8 & 3 & 252 & 3 & 0.009 & 0.004 \\
& & 4  & 19\,477.39 & 15.2 & 12\,037 & 0.004 & 0.004 & 19\,477.39 & 17.5 & 8 & 5 & 790 & 5 & 0.005 & 0.004 \\
& & 6  & 22\,177.40 & 26.5 & 10\,505 & 0.010 & 0.007 & 22\,177.39 & 17.8 & 8 & 5 & 616 & 5 & 0.007 & 0.007 \\
& & 8  & 28\,931.57 & 332.9 & 385\,594 & 0.010 & 0.010 & 28\,931.56 & 26.9 & 9 & 4 & 10\,191 & 4 & 0.010 & 0.010 \\
& & 10  & \textbf{31\,005.14} & \textbf{1739.1} & \textbf{1\,391\,779} & \textbf{0.010} & \textbf{0.010} &\textbf{ 31\,005.08} &\textbf{ TL} &\textbf{ 9} &\textbf{ 3} & \textbf{820\,390} &\textbf{ 3} &\textbf{ 0.023} &\textbf{ 0.010} \\
\hline
\multirow{5}{*}{8} & \multirow{5}{*}{(8,7)}
& 2  & 8\,612.07 & 7.6 & 1\,903 & 0.008 & 0.005 & 8\,612.07 & 5.0 & 9 & 4 & 3 & 3 & 0.005 & 0.005 \\
& & 4  & 18\,540.90 & 27.7 & 14\,589 & 0.007 & 0.007 & 18\,540.89 & 8.2 & 11 & 2 & 1\,255 & 2 & 0.009 & 0.007 \\
& & 6  & 26\,508.83 & 874.3 & 656\,144 & 0.009 & 0.009 & 26\,508.83 & 1268.4 & 9 & 5 & 2\,714\,674 & 3 & 0.010 & 0.009 \\
& & 8  & 29\,508.83 & 359.3 & 292\,940 & 0.010 & 0.010 & 29\,508.83 & 41.3 & 9 & 5 & 22\,025 & 3 & 0.010 & 0.010 \\
& & 10  & \textbf{38\,825.91} & \textbf{TL} & \textbf{579\,276} & \textbf{6.695} & \textbf{0.881} & \textbf{38\,785.82} & \textbf{TL} & \textbf{10} & \textbf{3}& \textbf{3\,723\,867 }& \textbf{3} & \textbf{0.985} & \textbf{0.985} \\
\hline\hline
\multirow{5}{*}{10} & \multirow{5}{*}{(10,9)}
& 2  & 11\,055.09 & 11.2 & 3\,929 & 0.008 & 0.003 & 11\,055.09 & 8.7 & 12 & 3 & 39 & 3 & 0.003 & 0.003 \\
& & 4  & 17\,863.25 & 136.1 & 214\,345 & 0.010 & 0.006 & 17\,863.25 & 10.6 & 12 & 3 & 2\,487 & 3 & 0.006 & 0.006 \\
& & 6  & 30\,245.72 & 18.3 & 1\,813 & 0.004 & 0.004 & 30\,245.72 & 12.1 & 13 & 2 & 2\,073 & 2 & 0.009 & 0.004 \\
& & 8  & 37\,612.33 & TL & 784\,915 & 0.535 & 0.009 & 37\,612.33 & 37.1 & 13 & 5 & 2\,088 & 3 & 0.009 & 0.009 \\
& & 10  & 40\,612.33 & 366.7 & 159\,702 & 0.010 & 0.005 & 40\,612.33 & 47.9 & 13 & 5 & 2\,789 & 3 & 0.005 & 0.005 \\
\hline
\multirow{5}{*}{15} & \multirow{5}{*}{(12,10)}
& 2  & 14\,000.00 & 72.7 & 29\,500 & 0.009 & 0.008 & 14\,000.00 & 4.9 & 15 & 5 & 3 & 3 & 0.008 & 0.008 \\
& & 4  & 28\,362.90 & 88.5 & 12\,455 & 0.010 & 0.007 & 28\,362.90 & 6.2 & 15 & 5 & 3 & 3 & 0.007 & 0.007 \\
& & 6  & 39\,606.76 & 139.0 & 20\,978 & 0.000 & 0.000 & 39\,606.76 & 6.2 & 17 & 4 & 3 & 3 & 0.010 & 0.000 \\
& & 8  & 49\,379.93 & 847.3 & 414\,651 & 0.010 & 0.007 & 49\,379.91 & 16.6 & 17 & 4 & 968 & 3 & 0.007 & 0.007 \\
& & 10  & 61\,791.00 & TL & 474\,673 & 0.445 & 0.003 & 61\,791.00 & 20.7 & 18 & 6 & 268 & 4 & 0.003 & 0.003 \\
\hline
\multirow{5}{*}{20} & \multirow{5}{*}{(12,10)}
& 2  & 29\,695.88 & TL & 614\,968 & 0.051 & 0.003 & 29\,695.88 & 7.7 & 20 & 6 & 3 & 3 & 0.003 & 0.003 \\
& & 4  & 45\,359.25 & TL & 239\,884 & 1.035 & 0.009 & 45\,359.25 & 6.5 & 20 & 6 & 3 & 3 & 0.009 & 0.009 \\
& & 6  & 54\,452.42 & 191.1 & 21\,242 & 0.005 & 0.001 & 54\,452.41 & 6.7 & 20 & 7 & 3 & 3 & 0.001 & 0.001 \\
& & 8  & 65\,252.42 & TL & 168\,891 & 20.494 & 0.000 & 65\,252.42 & 21.1 & 20 & 7 & 17 & 3 & 0.000 & 0.000 \\
& & 10  & 78\,057.43 & TL & 115\,784 & 44.854 & 0.001 & 78\,057.41 & 11.4 & 22 & 6 & 11 & 3 & 0.001 & 0.001 \\
\hline
\multirow{5}{*}{25} & \multirow{5}{*}{(12,11)}
& 2  & 27\,009.92 & TL & 138\,625 & 12.603 & 0.009 & 27\,010.26 & 10.4 & 27 & 5 & 8 & 3 & 0.008 & 0.008 \\
& & 4  & 54\,577.76 & TL & 345\,202 & 22.027 & 0.010 & 54\,577.74 & 9.7 & 26 & 7 & 3 & 3 & 0.010 & 0.010 \\
& & 6  & 66\,477.74 & TL & 161\,183 & 29.256 & 0.010 & 66\,477.74 & 13.0 & 26 & 8 & 3 & 3 & 0.010 & 0.010 \\
& & 8  & 83\,277.74 & TL & 184\,512 & 36.067 & 0.009 & 83\,277.74 & 10.3 & 26 & 9 & 9 & 3 & 0.009 & 0.009 \\
& & 10  & 98\,674.12 & TL & 113\,948 & 47.880 & 0.009 & 98\,674.11 & 26.5 & 27 & 4 & 2\,534 & 2 & 0.009 & 0.009 \\
\hline
\multirow{5}{*}{30} & \multirow{5}{*}{(12,12)}
& 2  & -- & TL & 112\,931 & -- & -- & 26\,666.18 & 11.0 & 33 & 11 & 4 & 4 & 0.010 & 0.010 \\
& & 4  & -- & TL & 106\,703 & -- & -- & 39\,102.74 & 47.8 & 36 & 5 & 19\,428 & 2 & 0.008 & 0.008 \\
& & 6  & 56\,712.73 & TL & 133\,978 & 111.374 & 0.010 & 56\,712.73 & 80.9 & 37 & 5 & 30\,523 & 2 & 0.010 & 0.010 \\
& & 8  & -- & TL & 107\,598 & -- & -- & 85\,373.21 & 439.0 & 37 & 6 & 76\,447 & 2 & 0.010 & 0.010 \\
& & 10  & \textbf{107\,075.08} & \textbf{TL} & \textbf{124\,470} & \textbf{87.170} & \textbf{0.418} & \textbf{107\,221.60} &\textbf{ TL} & \textbf{36} & \textbf{6} & \textbf{2\,292\,726} &\textbf{ 2} &\textbf{ 0.281} & \textbf{0.281} \\
\hline
\multirow{5}{*}{132} & \multirow{5}{*}{(12,12)}
& 2  & -- & TL & 40\,796 & -- & -- & 97\,020.71 & 6.7 & 138 & 4 & 2 & 2 & 0.009 & 0.009 \\
& & 4  & -- & TL & 41\,086 & -- & -- & 134\,696.98 & 5.4 & 138 & 4 & 2 & 2 & 0.006 & 0.006 \\
& & 6  & -- & TL & 48\,997 & -- & -- & 206\,690.32 & 19.8 & 138 & 6 & 55 & 3 & 0.000 & 0.000 \\
& & 8  & -- & TL & 42\,353 & -- & -- & 279\,290.32 & 36.0 & 137 & 6 & 149 & 3 & 0.002 & 0.002 \\
& & 10  & -- & TL & 51\,818 & -- & -- & 335\,336.97 & 27.7 & 141 & 2 & 225 & 2 & 0.000 & 0.000 \\
\hline
\end{tabular}
\end{sidewaystable}
\subsection{Impact of the toll bounds}
\label{sec:exp_bounds}

In the reference configuration of Section~\ref{sec:exp_setup}, the toll upper bound is
$\tau_a^{\tt up}=\min(c_a^{\tt fft},\tau_a^\infty)$, the smaller of the regulatory cap
$c_a^{\tt fft}$ used in~\cite{BLS_2024} and the implicit bound $\tau_a^\infty$ of
Proposition~\ref{prop:toll_upper_bound}. We now assess how the magnitude of this bound
affects tractability, by comparing it with a tighter and a looser choice:
\begin{equation}
\label{eq:toll_ub_types}
\tau_a^{\tt up}\in\Bigl\{\ \tfrac{1}{2}\min\!\left(c_a^{\tt fft},\tau_a^\infty\right),\quad
\min\!\left(c_a^{\tt fft},\tau_a^\infty\right),\quad \tau_a^\infty\ \Bigr\},
\qquad a\in A^{\tt toll}.
\end{equation}
The intermediate setting is the reference configuration, used in
Section~\ref{sec:exp_exact} and in the remainder of the paper; it keeps the toll set
$\mathcal{T}$ bounded. The loosest one retains only the endogenous bound
$\tau_a^\infty$. The tightest one halves the reference bound.

Taking the minimum with $\tau_a^\infty$ is never weaker than the regulatory cap alone,
and is strictly tighter whenever $\tau_a^\infty<c_a^{\tt fft}$; on the instances
considered here, however, this never occurs, so that the reference bound coincides with
the cap $c_a^{\tt fft}$ of~\cite{BLS_2024} on all 
toll arcs. The three settings
of~\eqref{eq:toll_ub_types} may therefore be read, on this benchmark, as half the
regulatory cap, the regulatory cap, and no regulatory cap.

On the $23$ instances solved to optimality under all three settings, the optimal value
$z^\star$ increases by a median factor of $2.34$ --- and by up to $6$ --- from the
tightest to the loosest bound.

\begin{table}[h!]
\caption{\label{tab:rq_tollsummary} Effect of the toll upper bound~\eqref{eq:toll_ub_types}
on the tractability of the exact formulations (time limit TL${}=1800$\,s, $55$ instances).
``solved'' counts instances solved to proven optimality; ``no sol.'' counts instances for
which no feasible solution was found; ``time'' is the shifted geometric mean of the
solution time (shift $10$\,s, unsolved runs counted at the time limit). Per-instance
results are reported in Tables~\ref{tab:rq1}, \ref{tab:rq_toll0} and~\ref{tab:rq_toll2}.}
\centering
\small
\begin{tabular}{|c||c|c|c||c|c|c|}
\hline
\multirow{2}{*}{$\tau_a^{\tt up}$}
& \multicolumn{3}{c||}{${\tt (BLS)}$} & \multicolumn{3}{c|}{${\tt (SD\text{-}A)}$}\\
\cline{2-7}
& solved & no sol. & time (s) & solved & no sol. & time (s)\\
\hline\hline
$\tfrac{1}{2}\min(c_a^{\tt fft},\tau_a^\infty)$ & $41/55$ & $7$  & $60.5$   & $55/55$ & $0$ & $5.1$   \\
\hline
$\min(c_a^{\tt fft},\tau_a^\infty)$             & $33/55$ & $8$  & $206.1$  & $52/55$          & $0$ & $24.5$           \\
\hline
$\tau_a^\infty$                                 & $8/55$  & $14$ & $1149.2$ & $23/55$          & $0$ & $403.9$          \\
\hline
\end{tabular}
\end{table}

\noindent\textbf{Tractability.}
Table~\ref{tab:rq_tollsummary} shows a clear, monotone effect: the looser the toll bound,
the harder the instances, for both formulations. As the bound widens from
$\tfrac{1}{2}\min(c_a^{\tt fft},\tau_a^\infty)$ to $\tau_a^\infty$, the number of instances
solved by (SD-A) falls from all $55$ to $23$, and by (BLS) from $41$ to $8$, while the
shifted geometric mean of the solution time grows by more than an order of magnitude.
This is consistent with the spatial branch-and-bound applied to the bilinear terms
$f_a\tau_a$: a wider toll range enlarges the box over which the McCormick
relaxation~\cite{MCC_1976} is built, so that more branching is required to close the gap.
Across all three settings, (SD-A) dominates (BLS), solving from $1.3$ times as many
instances under the tightest bound to close to three times as many under the loosest, and
remaining between $3$ and $12$ times faster in shifted geometric mean; (SD-A) always
returns a feasible solution, whereas (BLS) fails to do so on up to $14$ instances under the
loosest bound. Under the tightest bound, (SD-A) solves every instance, all but two of them
with at most $100$ branch-and-bound nodes (see Table~\ref{tab:rq_toll0}).

\noindent\textbf{Gaps at the time limit.}
The residual gaps confirm the reading of Section~\ref{sec:exp_exact} and extend it across
regimes. On the instances that one formulation solves but the other does not, the failing
formulation stays within a few hundredths of a percent of optimality --- at most $0.054\%$
for (BLS), with a mean of $0.007\%$ over $31$ such instances, and at most $0.010\%$ for
(SD-A) over $2$ --- so that these failures are again of certification rather than of primal
search. On the instances left unsolved by \emph{both} formulations, the picture depends on
the regime. Under the tightest bound there are none; under the reference bound there are two
--- $\left|\mathcal{K}\right|=8$ and $\left|\mathcal{K}\right|=30$, both with $10$ toll arcs --- with pooled gaps
of at most $0.985\%$. 
Under $\tau_a^\infty$, by contrast, $31$ instances are left unsolved by both; on the
$21$ of them on which (BLS) returns a feasible solution, the pooled gaps average about
$20\%$ for (BLS) and $38\%$ for (SD-A), and exceed $100\%$ on the hardest instances.
The genuinely hard core of the problem thus emerges only once the regulatory cap is
removed; there, on these $21$ instances, the feasible solutions returned by (BLS) are
markedly better than those of (SD-A) (a mean certified gap of $20.4\%$ against
$37.6\%$), even though (SD-A) solves far more instances overall and is the only one to
return a feasible solution on the remaining $10$.

\subsection{Impact of the cost structure on toll arcs}
\label{sec:exp_costs}

In the reference configuration of Section~\ref{sec:exp_setup}, the cost function of a
toll arc $a\in A^{\tt toll}$ follows~\cite{BLS_2024}: both coefficients read from the
instance file are halved, so that its slope is $\alpha_a/2$ and its intercept is
$\beta_a=c_a^{\tt fft}/2$. We now assess the effect of this choice by retaining instead
the original coefficients, i.e.\ slope $\alpha_a$ and intercept
$\beta_a=c_a^{\tt fft}$. All other settings, including the toll upper bound
$\tau_a^{\tt up}=\min(c_a^{\tt fft},\tau_a^\infty)$, are those of the reference
configuration; only the cost of the toll arcs changes.

Retaining the original cost coefficients lowers the optimal revenue --- the median ratio
is $0.74$ over the $44$ instances solved under both settings, with extreme values $0.56$
and $0.96$ --- while making the instances harder for both formulations. The revenue drop
is expected: since every commodity retains a toll-free alternative, increasing the
un-tolled cost of the toll arcs reduces the margin available to the leader. Note that the
toll upper bound itself is unchanged here: on every toll arc of the benchmark
$\tau_a^\infty>c_a^{\tt fft}$, so that $\tau_a^{\tt up}=c_a^{\tt fft}$ under both cost
structures. 
\begin{table}[h!]
\caption{\label{tab:rq_costsummary} Effect of the cost structure on toll arcs on the
tractability of the exact formulations (time limit TL${}=1800$\,s, $55$ instances).
Columns are as in Table~\ref{tab:rq_tollsummary}: ``solved'' counts instances solved to
proven optimality, ``no sol.'' those for which no feasible solution was found, and
``time'' is the shifted geometric mean of the solution time (shift $10$\,s, unsolved runs
counted at the time limit). 
Per-instance results are reported in Table~\ref{tab:rq1} and Appendix~\ref{sec:app_impact_tollcosttype}.}
\centering
\small
\begin{tabular}{|c||c|c|c||c|c|c|}
\hline
\multirow{2}{*}{$(\text{slope},\ \text{intercept})$ on $A^{\tt toll}$}
& \multicolumn{3}{c||}{${\tt (BLS)}$} & \multicolumn{3}{c|}{${\tt (SD\text{-}A)}$}\\
\cline{2-7}
& solved & no sol. & time (s) & solved & no sol. & time (s)\\
\hline\hline
$\left(\alpha_a/2,\ c_a^{\tt fft}/2\right)$ (reference) & $33/55$ & $8$ & $206.1$ & $52/55$ & $0$ & $24.5$ \\
\hline
$\left(\alpha_a,\ c_a^{\tt fft}\right)$                 & $22/55$ & $10$ & $471.7$ & $45/55$          & $0$ & $56.0$          \\ \hline
\end{tabular}
\end{table}

\noindent\textbf{Tractability.}
Table~\ref{tab:rq_costsummary} shows that retaining the original coefficients penalizes
both formulations: the number of instances solved to optimality drops from $33$ to $22$
for (BLS) and from $52$ to $45$ for (SD-A), while the shifted geometric mean of the
solution time roughly doubles for both. The ranking of the two formulations is
unaffected, and the advantage of (SD-A) actually grows: on the instances solved by both, its
median speed-up over (BLS) increases from $1.5\times$ to $2.4\times$, and it remains the
only formulation that returns a feasible solution on every instance, whereas (BLS) fails
to do so on $10$ of them. The degradation is not uniform in the number of commodities:
(SD-A) still solves all five instances with $\left|\mathcal{K}\right|=132$ and all five with
$\left|\mathcal{K}\right|=15$, but only two of five for $\left|\mathcal{K}\right|=10$ --- another indication
that the difficulty of an instance is not governed by its size alone.
Finally, the constraint-generation procedure remains equally effective under the original
coefficients. Separation is triggered on essentially all instances in both cost structures,
but converges in at most five iterations, generating at most $11$ path inequalities per
instance under the reference coefficients and at most $16$ under the original ones. The
initial seed is slightly larger under the original coefficients --- on average $0.85$ more
inequalities per instance, and more on $39$ of the $55$ --- and the total number of
inequalities generated over the benchmark rises from $248$ to $301$, consistent with a
steeper congestion term spreading the equilibrium flow over a larger number of paths.

\noindent\textbf{Gaps at the time limit.}
The picture of Sections~\ref{sec:exp_exact} and~\ref{sec:exp_bounds} carries over. On the
instances that (SD-A) solves but (BLS) does not, (BLS) is again close to optimality ---
on average $0.008\%$ and at most $0.033\%$ over the $14$ such instances --- so that these
failures remain ones of certification rather than of primal search; symmetrically, every
instance solved by (BLS) is also solved by (SD-A). On the ten instances left unsolved by
\emph{both} formulations, the residual pooled gaps remain small: on the nine where (BLS)
also returns a feasible solution, they average $0.840\%$ for (BLS) against $1.046\%$ for
(SD-A), so that (BLS) returns slightly better feasible solutions on this hard core.

\subsection{Path-based formulations as heuristics}
\label{sec:exp_heuristic}

We now evaluate the two path-based formulations, (SD-P)~\eqref{eq:form_sd_paths_small}
and (CC)~\eqref{eq:form_cs}, solved by column-and-constraint generation as
described in Section~\ref{sec:exp_setup}. As explained there, the termination of
this procedure does not certify global optimality for~\eqref{eq:BW}; these
formulations are therefore used as heuristics, and we assess them in terms of
the quality of the solutions they produce and of the associated computational
effort.

\paragraph{Evaluation by bounding.}
For each instance, the heuristic returns a valid lower bound $\underline{z}$ on
the optimal value $z^\star$ of~\eqref{eq:BW}, obtained as follows. If the
procedure terminates within the time limit (no violated constraint is found),
the flow of the current restriction is a Wardrop equilibrium for the
corresponding tolls, so the restriction's optimal value is feasible
for~\eqref{eq:BW} and is reported as $\underline{z}$. If the time limit is
reached, we take the tolls of a feasible solution of the current restriction,
compute a Wardrop equilibrium for these tolls (by solving~\eqref{eq:wardrop_syst_aff}), 
and report the resulting leader's revenue, which is feasible
for~\eqref{eq:BW}.

The certified relative
gap
\[
\frac{\overline{z}-\underline{z}}{\underline{z}}
\]
is then an upper bound on the true gap between the heuristic solution and the
optimum.

\paragraph{Results.}
The results are reported in Table~\ref{tab:rq4}, for both (SD-P) and (CC).

\emph{Solution quality.} Both heuristics prove remarkably accurate. On the $53$
instances for which one of the exact formulations of Section~\ref{sec:exp_exact}
proved optimality, (SD-P) and (CC) both return a solution whose value coincides
with $z^\star$, up to the solver tolerance, on \emph{all} $53$ of them. On the two
remaining instances --- $\left|\mathcal{K}\right|=8$ with $10$ toll arcs and $\left|\mathcal{K}\right|=30$
with $10$ toll arcs --- for which no exact formulation closed the gap within the time
limit, both heuristics match the best feasible solution found by any method. 
The certified gaps of Table~\ref{tab:rq4} are correspondingly small: at most $0.01\%$
on $52$ of the $55$ instances for (SD-P) and on $53$ for (CC), with a median of
$0.005\%$. The larger values, up to $0.881\%$, occur on the two instances that no exact
formulation closed --- where the heuristic solution is the best one known and the
residual gap may reflect the looseness of $\overline{z}$ --- and, for (SD-P), on
$\left|\mathcal{K}\right|=6$ with $10$ toll arcs ($0.011\%$), where the returned value
lies within the solver tolerance of the optimum.

\emph{Comparison of the two heuristics.} The two formulations return essentially the
same objective values 
--- the largest relative difference between their values over the whole
benchmark is $2.2\cdot 10^{-5}$ --- but they differ sharply in computational effort.
(CC) is faster on $54$ of the $55$ instances, by a median factor of $13.6$ and by up
to nearly three orders of magnitude; over the whole benchmark it requires $20.5$
seconds and $25\,069$ branch-and-bound nodes in total, against $1026$ seconds and
$2.3$ million nodes for (SD-P). Both terminate within the time limit on every instance.
The contrast is sharpest on the instance with $\left|\mathcal{K}\right|=8$ and $10$ toll arcs,
where (SD-P) needs $475$ seconds and more than $1.5$ million nodes while (CC)
terminates in $0.8$ seconds after $1\,176$ nodes, both returning the same objective value. 
Since the two formulations share the same column-and-constraint generation scheme, the
same initialization and, as observed above, the same objective values, this gap in
performance is attributable to the way the equilibrium conditions are expressed: the
complementarity conditions handled through SOS1 constraints in (CC) are processed far
more efficiently than the strong-duality inequality of (SD-P), whose quadratic term
couples all arc flows.

\emph{Comparison with the exact formulations.} The heuristics are most useful precisely
where the exact formulations struggle. On the three instances on which (SD-A) reaches
the time limit, (CC) returns in at most $2.9$ seconds a solution at least as good as the
incumbent (SD-A) reaches after $1800$ seconds; on $\left|\mathcal{K}\right|=30$ with $8$ toll arcs,
where (BLS) finds no feasible solution at all, it returns the best solution known in
under two seconds. The constraint-generation procedure behaves as it does for (SD-A):
the sparse shortest-path initialization leaves at least one violated inequality on every
instance, so that separation is triggered throughout, but it converges in at most five
iterations, generating at most $13$ path inequalities on any instance.

\paragraph{Sensitivity to the initial path set.}
Initializing with the zero-toll equilibrium alone preserves the accuracy of both
formulations: each still returns an optimal solution on every instance whose optimal
value is known. The burden merely shifts to the separation, which generates about $60\%$
more path inequalities over the benchmark ($402$ against $250$). Initializing with the
equilibrium at $\tau=\tau^{\tt up}$ alone, by contrast, degrades the results markedly,
and does so slightly more for (CC) than for (SD-P): (SD-P) returns an optimal solution on
$46$ of the $53$ instances, with a worst-case shortfall of $10.2\%$, and (CC) on 
 $42$, with a worst case of $11.2\%$.

The explanation is structural. At $\tau=\tau^{\tt up}$ the tolled arcs are expensive, so
the corresponding equilibrium largely avoids them, and the resulting path set offers
little capacity to route flow through the toll arcs --- precisely the flow on which
revenue is raised. Although the separation is still triggered on every instance, it is
driven by the relaxation flow and does not, on its own, recover the missing
revenue-carrying routes: the column-and-constraint generation then converges to an
impoverished restricted model, leaving the returned solution suboptimal on $7$ of the
$53$ instances for (SD-P) and on $11$ for (CC). Consistently with this, the two
formulations, which return essentially identical solutions under the reference
initialization, no longer do so under this one.

These experiments show that the quality of the solutions returned by the path-based
formulations is not an automatic consequence of the column-and-constraint generation
scheme: it relies on the initial path set containing the routes that users take when
tolls do not deter them. They also indicate that the two equilibria used for
initialization play complementary roles, the zero-toll one being essential to
solution quality and the one at $\tau=\tau^{\tt up}$ 
mainly reducing the number of path inequalities that must be generated. 

\begin{table}[h!]
\caption{\label{tab:rq5_init} Sensitivity of the path-based formulations to the
initial path set, in the reference configuration (time limit TL${}=1800$\,s, $55$
instances). ``optimal'' counts the instances, among the $53$ whose optimal value is
known from an exact formulation, on which the returned solution is optimal;
``$\sum\#p_g$'' is the total number of path inequalities generated, and ``separated''
the number of instances on which at least one was generated. 
Per-instance results are reported in Tables~\ref{tab:rq5_pathmin} and~\ref{tab:rq5_pathmax}
(Appendix~\ref{sec:app_heur}).}
\centering
\small
\begin{tabular}{|l|l||c|c|c|c|c|}
\hline
& initial path set & optimal & $\sum\#p_g$ & separated & time (s) & worst shortfall\\
\hline\hline
\multirow{3}{*}{(SD-P)}
& $\tau=0$ and $\tau=\tau^{\tt up}$ & $\mathbf{53/53}$ & $250$ & $55/55$ & $1026.1$ & $0.002\%$ \\
& $\tau=0$ only                     & $\mathbf{53/53}$ & $402$ & $55/55$ & $2722.2$ & $0.002\%$ \\
& $\tau=\tau^{\tt up}$ only         & $46/53$          & $326$ & $55/55$ & $516.7$  & $10.230\%$ \\
\hline
\multirow{3}{*}{(CC)}
& $\tau=0$ and $\tau=\tau^{\tt up}$ & $\mathbf{53/53}$ & $247$ & $55/55$ & $20.5$  & $0.001\%$ \\
& $\tau=0$ only                     & $\mathbf{53/53}$ & $402$ & $55/55$ & $26.2$  & $0.001\%$ \\
& $\tau=\tau^{\tt up}$ only         & $42/53$          & $317$ & $55/55$ & $16.0$  & $11.246\%$ \\
\hline
\end{tabular}
\end{table}

\subsection{The full Sioux Falls and Eastern Massachusetts networks}
\label{sec:FullSiouxFalls_EMA}

\begin{table}[h!]
\centering
\small
\setlength{\tabcolsep}{2.4pt}
\renewcommand{\arraystretch}{1.1}
\caption{\label{tab:rq_larger} Summary of the results on the three networks (reference
configuration, time limit TL${}=1800$\,s), for the exact formulations (top) and the
path-based heuristics (bottom), reported with the same fields. ``Subnetwork'', ``Full SF'' and ``East.\ Mass.'' denote the $12$-node Sioux Falls subnetwork, the complete Sioux Falls network and the Eastern Massachusetts network, respectively. ``\#\,inst.'' is the
number of instances and ``\#\,known'' the number of them whose optimal value is known,
i.e.\ solved to proven optimality by an exact formulation. For each method, ``term.''
counts the runs that terminate within the time limit --- for an exact formulation
this certifies optimality, whereas for a heuristic it only means that no violated path
inequality remains; ``opt'' counts the instances, among the ``\#\,known'' ones, on which
the returned solution is optimal up to the solver tolerance ($0.01\%$), whether or not
the run terminates; ``best'' counts the instances on which the returned solution
matches, within the same tolerance, the best solution returned by any of the four
methods; ``no feas.'' counts the instances on which no feasible solution is returned;
``sgm'' is the shifted geometric mean of the solution time (in seconds, $10$\,s shift,
time-limited runs counted at the time limit).}
\begin{tabular}{|l|c|c||c|c|c|c|r||c|c|c|c|r|}
\hline
\multicolumn{13}{|c|}{\emph{Exact formulations}}\\
\hline
\multirow{2}{*}{Network} & \multirow{2}{*}{\shortstack{\#\\inst.}} & \multirow{2}{*}{\shortstack{\#\\known}}
 & \multicolumn{5}{c||}{(BLS)} & \multicolumn{5}{c|}{(SD-A)} \\
\cline{4-13}
 & & & term. & opt & best & \begin{tabular}[b]{@{}c@{}}no\\feas.\end{tabular} & sgm & term. & opt & best & \begin{tabular}[b]{@{}c@{}}no\\feas.\end{tabular} & sgm \\
\hline\hline
Subnetwork       & $55$ & $53$ & $33$ & $45$ & $46$ & $8$  & $206.1$  & $52$ & $53$ & $53$ & $0$ & $24.5$  \\
Full SF          & $55$ & $36$ & $7$  & $14$ & $19$ & $31$ & $1305.6$ & $35$ & $35$ & $39$ & $0$ & $232.7$ \\
East.\ Mass.    & $60$ & $26$ & $0$  & $11$ & $12$ & $45$ & $1800.0$ & $26$ & $26$ & $29$ & $0$ & $392.5$ \\
\hline\hline
\multicolumn{13}{|c|}{\emph{Path-based heuristics}}\\
\hline
\multirow{2}{*}{Network} & \multirow{2}{*}{\shortstack{\#\\inst.}} & \multirow{2}{*}{\shortstack{\#\\known}}
 & \multicolumn{5}{c||}{(SD-P)} & \multicolumn{5}{c|}{(CC)} \\
\cline{4-13}
 & & & term. & opt & best & \begin{tabular}[b]{@{}c@{}}no\\feas.\end{tabular} & sgm & term. & opt & best & \begin{tabular}[b]{@{}c@{}}no\\feas.\end{tabular} & sgm \\
\hline\hline
Subnetwork       & $55$ & $53$ & $55$ & $53$ & $55$ & $0$ & $6.7$   & $55$ & $53$ & $55$ & $0$ & $0.4$  \\
Full SF          & $55$ & $36$ & $45$ & $32$ & $43$ & $0$ & $48.4$  & $53$ & $34$ & $52$ & $0$ & $11.5$ \\
East.\ Mass.    & $60$ & $26$ & $40$ & $26$ & $42$ & $0$ & $112.0$ & $59$ & $26$ & $59$ & $0$ & $8.4$  \\
\hline
\end{tabular}
\end{table}
The results obtained with the reference configuration on the two larger networks are
summarized in Table~\ref{tab:rq_larger}, together with those on the subnetwork, with
the per-instance results in Appendices~\ref{sec:app_bignet} and~\ref{sec:EAM_net}. The
exact formulations (upper part) and the heuristics (lower part) are reported with the
same fields, so that the two parts can be read side by side --- with the caveat that
termination within the time limit (``term.'') certifies optimality only for the exact
formulations.

As the network grows from $12$ to $24$ to $74$ nodes, (BLS) degrades rapidly: it solves
$33$ of the $55$ subnetwork instances, $7$ of the $55$ full-Sioux-Falls instances, and
\emph{none} of the $60$ Eastern Massachusetts instances, on $45$ of which it returns no
feasible solution at all. (SD-A), by contrast, remains effective and always feasible ---
solving $52$, $35$ and $26$ instances respectively. 
The ``opt'' column shows that part of the
difficulty of (BLS) is one of certification: it returns an optimal solution on $45$,
$14$ and $11$ of the instances whose optimum is known, while proving optimality on only
$33$, $7$ and $0$ of them. This is not a matter of instance size alone: on all three
networks, some small-demand instances remain unsolved while some large ones are solved.
It is a matter of formulation. The strong-duality reformulation 
remains usable on
networks on which the big-$M$ linearization of (BLS) no longer even locates a feasible
point.

The two path-based heuristics extend this reach further. Wherever an exact formulation
certifies an optimum, both recover it: on all $53$ such instances on the subnetwork, on
$32$ (SD-P) and $34$ (CC) of the $36$ on the full Sioux Falls, and on all $26$ on
Eastern Massachusetts, with worst-case shortfalls below $0.2\%$; and where the exact
formulations run out of time or feasibility, the heuristics still return a solution on
every instance. Between the two, (CC) is consistently the faster --- with a shifted
geometric mean of $0.4$, $11.5$ and $8.4$ seconds, against $6.7$, $48.4$ and $112.0$ for
(SD-P) --- and, on the larger networks, the more robust: it terminates within the time
limit on $53$ of the $55$ full-Sioux-Falls instances and on $59$ of the $60$ Eastern
Massachusetts ones, against $45$ and $40$ for (SD-P). It is also the method that most
often returns the best solution known: on $52$ and $59$ instances of the two larger
networks, against $39$ and $29$ for (SD-A) and $43$ and $42$ for (SD-P). In particular,
on the instances on which (SD-A) itself reaches the time limit, (CC) returns a solution
at least as good as the (SD-A) incumbent on all $3$ such instances of the subnetwork,
on $19$ of the $20$ of the full Sioux Falls network, and on all $34$ of the Eastern
Massachusetts network.
\section{Conclusion}
\label{sec:conclusion}
 
In this paper, we studied the toll-setting problem~\eqref{eq:BW} under Wardrop
equilibrium constraints, in the case of separable affine arc cost functions and with
the objective of maximizing the leader's revenue. Starting from Beckmann's convex
formulation of Wardrop equilibria and its dual, 
we derived three new single-level
reformulations of the problem: an arc-based formulation (SD-A), which we solve exactly, and
two path-based formulations, (SD-P) and (CC), solved by column-and-constraint
generation. We showed that the common assumption that every commodity admits a
toll-free path induces finite implicit bounds on the tolls
(Proposition~\ref{prop:toll_upper_bound}), which makes it possible to treat the case
of unbounded tolls without resorting to arbitrary bounds, and we derived from
Beckmann's potential an upper bound on the optimal revenue that applies uniformly to
all the formulations considered and is used to certify optimality gaps.
 
On the computational side, we conducted an extensive study on instances derived from
the Sioux Falls and Eastern Massachusetts networks, comparing our formulations with
the recent formulation of Beck et al.~\cite{BLS_2024}. The arc-based formulation
(SD-A) proved markedly more robust: it solved to proven optimality substantially more
instances than (BLS) and, unlike (BLS), returned a feasible solution on every instance
of every network. The difference widened with network size --- on the largest network,
(BLS) solved none of the instances and returned no feasible solution on most of them,
whereas (SD-A) still solved a sizeable fraction --- and it is one of formulation rather
than of instance size, since on all three networks some small-demand instances resisted
while some larger ones were solved. Pooling the dual bounds of the formulations, rather
than relying on each solver's own bound, also allowed us to localize one 
difficulty: on 
many instances left unsolved by (BLS) with an 
incumbent, the latter  turned out to be optimal, or nearly
so, the obstacle being the certification of optimality rather than the search for a good
solution.

 Used as heuristics, the two path-based formulations proved remarkably accurate:
they returned an optimal solution on every instance whose optimal value is known on the
subnetwork and on the Eastern Massachusetts network, and on all but a few on the full
Sioux Falls, always within a fraction of a percent of the optimum; and they returned a
feasible solution on every instance on which the exact formulations found none, often
in a few seconds where those formulations reached the time limit. The
complementarity-based formulation (CC) proved consistently and substantially cheaper
than (SD-P), and, on the larger networks, the more robust of the two; the two return
essentially identical solutions on the subnetwork, but no longer coincide as the network
grows. This accuracy rests on the initial set of paths: an ablation of the
initialization shows that the shortest paths at the zero-toll equilibrium are important 
to it.

Finally,
and beyond the specific figures, our experiments highlight that the magnitude of the
toll bounds is a significant factor affecting computational performance. 
 
Several directions for future research emerge from this work. 
First, the path-based formulations (SD-P) and (CC) are used here as heuristics:
termination of the column-and-constraint generation guarantees feasibility, but
not optimality. Given their accuracy in our experiments, devising a way to turn
them into exact methods is a natural next step. 
Second, the paths
and feasible solutions generated by the path-based heuristics could be used to
initialize or warm-start the exact formulation (SD-A); our observation that the
bottleneck often lies in certification rather than in primal search suggests,
however, that the benefit of a better incumbent alone may be limited, and that such
an initialization would have to be combined with techniques improving the dual
bound. Third, a more systematic study of the influence of the toll-arc cost
structure, only briefly touched upon here, would be of interest. 

Finally, another promising direction for future research is to develop tailored solution 
approaches that exploit the specific structure of the reformulations, such as the fact that 
the nonconvexity of (SD-A) is concentrated in a single bilinear constraint.

\section*{Acknowledgments}
This research benefited from the support of the FMJH Program Gaspard Monge for
optimization and operations research and their interactions with data science.

\bibliographystyle{plain}
\bibliography{references}

\appendix

\section{Instances based on the Sioux Falls network}
\label{app:instance_Sioux}
\subsection{Description of the network and demands}
In the computations reported in Sections~\ref{sec:exp_exact}--\ref{sec:exp_heuristic},
we use a subnetwork of the Sioux Falls network~\cite{LMP_1975,TransportationNetworks},
already considered in~\cite{BLS_2024}. It is induced by the $12$ nodes
$\{7,8,9,10,15,16,17,18,19,20,21,22\}$ of the original network and has $36$ arcs; see
Figure~\ref{fig:siouxfalls}.

\begin{figure}[htbp]
\centering
\caption{The Sioux Falls network ($24$ nodes, $76$ arcs), drawn from the node
coordinates and links of the benchmark data~\cite{TransportationNetworks}, with the
original node identifiers. The shaded subnetwork induced by the $12$ nodes
$\{7,8,9,10,15,16,17,18,19,20,21,22\}$ ($36$ arcs), used in the main experiments, is
drawn with thicker edges. Each drawn edge represents a pair of opposite arcs.}
\label{fig:siouxfalls}
\begin{tikzpicture}[
    sub/.style={circle,draw=black,line width=1pt,fill=black!15,inner sep=0pt,minimum size=6.5mm,font=\small},
    reg/.style={circle,draw=black,line width=0.6pt,fill=white,inner sep=0pt,minimum size=6.5mm,font=\small}]
 
  \fill[black!7,rounded corners=10pt] (4.49,3.89) -- (4.81,2.2) -- (4.93,1.24) -- (7.03,1.99) -- (9.0,6.81) -- (8.79,9.08) -- (6.98,9.04) -- (4.69,7.66) -- (4.42,6.44) -- cycle;
 
  \draw[line width=0.55pt] (1.91,14.0) -- (6.82,13.2);
  \draw[line width=0.55pt] (1.91,14.0) -- (1.58,9.43);
  \draw[line width=0.55pt] (6.82,13.2) -- (6.79,11.11);
  \draw[line width=0.55pt] (1.58,9.43) -- (3.84,8.36);
  \draw[line width=0.55pt] (1.58,9.43) -- (1.1,6.1);
  \draw[line width=0.55pt] (3.84,8.36) -- (5.14,8.41);
  \draw[line width=0.55pt] (3.84,8.36) -- (3.87,6.12);
  \draw[line width=0.55pt] (5.14,8.41) -- (6.79,11.11);
  \draw[line width=0.55pt] (5.14,8.41) -- (5.16,6.95);
  \draw[line width=0.55pt] (6.79,11.11) -- (6.81,8.21);
  \draw[line width=0.55pt] (5.15,6.0) -- (3.87,6.12);
  \draw[line width=0.55pt] (3.87,6.12) -- (1.1,6.1);
  \draw[line width=0.55pt] (3.87,6.12) -- (3.52,4.43);
  \draw[line width=0.55pt] (1.1,6.1) -- (0.0,0.0);
  \draw[line width=0.55pt] (0.0,0.0) -- (3.67,1.43);
  \draw[line width=0.55pt] (3.52,4.43) -- (5.14,4.44);
  \draw[line width=0.55pt] (3.52,4.43) -- (3.53,2.77);
  \draw[line width=0.55pt] (5.18,2.05) -- (3.67,1.43);
  \draw[line width=0.55pt] (5.16,2.98) -- (3.53,2.77);
  \draw[line width=0.55pt] (3.53,2.77) -- (3.67,1.43);
  \draw[line width=1.8pt] (8.31,8.38) -- (6.81,8.21);
  \draw[line width=1.8pt] (8.31,8.38) -- (8.25,6.42);
  \draw[line width=1.8pt] (6.81,8.21) -- (5.16,6.95);
  \draw[line width=1.8pt] (6.81,8.21) -- (6.81,6.6);
  \draw[line width=1.8pt] (5.16,6.95) -- (5.15,6.0);
  \draw[line width=1.8pt] (5.15,6.0) -- (5.14,4.44);
  \draw[line width=1.8pt] (5.15,6.0) -- (6.81,6.6);
  \draw[line width=1.8pt] (5.15,6.0) -- (6.81,5.55);
  \draw[line width=1.8pt] (5.14,4.44) -- (6.82,4.46);
  \draw[line width=1.8pt] (5.14,4.44) -- (5.16,2.98);
  \draw[line width=1.8pt] (6.81,6.6) -- (6.81,5.55);
  \draw[line width=1.8pt] (6.81,6.6) -- (8.25,6.42);
  \draw[line width=1.8pt] (6.81,5.55) -- (6.82,4.46);
  \draw[line width=1.8pt] (8.25,6.42) -- (6.83,2.82);
  \draw[line width=1.8pt] (6.82,4.46) -- (6.83,2.82);
  \draw[line width=1.8pt] (6.83,2.82) -- (5.18,2.05);
  \draw[line width=1.8pt] (6.83,2.82) -- (5.16,2.98);
  \draw[line width=1.8pt] (5.18,2.05) -- (5.16,2.98);
  \node[reg] at (1.91,14.0) {1};
  \node[reg] at (6.82,13.2) {2};
  \node[reg] at (1.58,9.43) {3};
  \node[reg] at (3.84,8.36) {4};
  \node[reg] at (5.14,8.41) {5};
  \node[reg] at (6.79,11.11) {6};
  \node[sub] at (8.31,8.38) {7};
  \node[sub] at (6.81,8.21) {8};
  \node[sub] at (5.16,6.95) {9};
  \node[sub] at (5.15,6.0) {10};
  \node[reg] at (3.87,6.12) {11};
  \node[reg] at (1.1,6.1) {12};
  \node[reg] at (0.0,0.0) {13};
  \node[reg] at (3.52,4.43) {14};
  \node[sub] at (5.14,4.44) {15};
  \node[sub] at (6.81,6.6) {16};
  \node[sub] at (6.81,5.55) {17};
  \node[sub] at (8.25,6.42) {18};
  \node[sub] at (6.82,4.46) {19};
  \node[sub] at (6.83,2.82) {20};
  \node[sub] at (5.18,2.05) {21};
  \node[sub] at (5.16,2.98) {22};
  \node[reg] at (3.53,2.77) {23};
  \node[reg] at (3.67,1.43) {24};
\end{tikzpicture}
\end{figure}
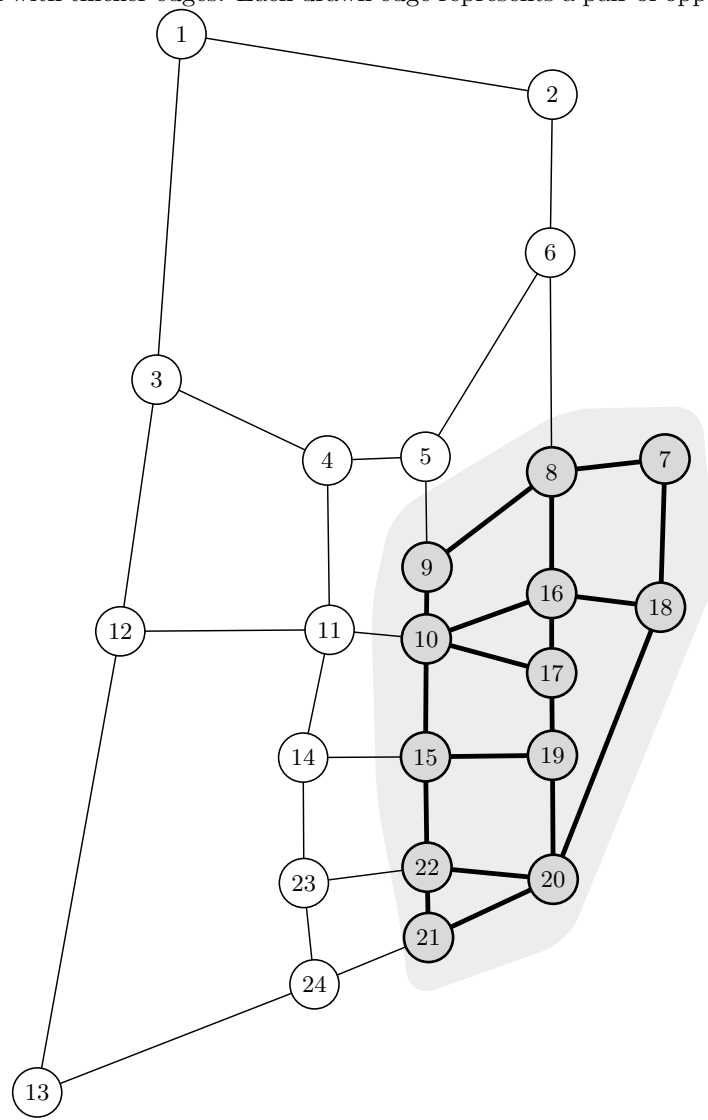

For the arc cost functions, we use the ones proposed in~\cite{BLS_2024}:
\begin{equation}
c_a\left(f_a\right)=c_a^{\tt{fft}}\left(1+1.5\cdot\frac{f_a}{u_a}\right),\qquad a\in A,
\end{equation}
where $c_a^{\tt{fft}}$ is the so-called \emph{free flow time} given in the
description of the network and $u_a$ is the capacity associated with arc
$a\in A$; in the notation of Section~\ref{sec:wardrop}, $\alpha_a=1.5\,c_a^{\tt fft}/u_a$
and $\beta_a=c_a^{\tt fft}$. Following~\cite{BLS_2024}, once the toll arcs have been
selected (Appendix~\ref{app:toll_arcs}), both coefficients of the cost function of each
toll arc are halved in the reference configuration, which yields the slope $\alpha_a/2$
and the intercept $c_a^{\tt fft}/2$ for $a\in A^{\tt toll}$; the alternative in which the
original coefficients are retained is studied in Section~\ref{sec:exp_costs}. 
The same cost functions, built from the free-flow times and capacities of the network
data~\cite{TransportationNetworks}, are used for the complete Sioux Falls and Eastern
Massachusetts networks (Appendices~\ref{sec:app_bignet} and~\ref{sec:EAM_net}).

For the demands, our \emph{complete} set $\mathcal{K}$ refers to the
restriction, to the subgraph under consideration, of the demands of the
original instance having both endpoints in this subgraph; there are 132 such
demands. When considering a smaller number $k$ of commodities, with
$k\in\{4,5,6,7,8,10,15,20,25,30\}$, we mean the first $k$ demands as ordered
in Table~\ref{tab:demands_eval}. The case $k=132$ corresponds to the complete
set $\mathcal{K}$. 

\begin{table}[h!]
\centering
\caption{OD pairs considered in our instances based on a subnetwork of the
Sioux Falls network, with the demand values and identifiers. The ordering of
the demands (column \#) determines, for each $k\le 30$, the first $k$
commodities used in the corresponding instances.}
\label{tab:demands_eval}
 \begin{tabular}{|c|c|c||c|c|c||c|c|c|}
 \hline
\#& $(o_k,s_k)$& demand &\# & $(o_k,s_k)$& demand &\# & $(o_k,s_k)$& demand  \\
 \hline\hline
1&(9,21)& 300.0&11&(10,21)&1200.0&21&(19,21)&400.0\\
2&(8,20)& 900.0&12&(18,22)& 300.0&22&(22,16)&1200.0\\
3&(16,21)& 600.0&13&(9,17)& 900.0&23&(21,7)&200.0\\
4&(17,22)& 1700.0&14&(8,19)& 700.0&24&(20,9)&600.0\\
5&(21,18)& 100.0&15&(7,22)& 500.0&25&(9,22)&700.0\\
6&(20,10)& 2500.0&16&(10,7)& 1900.0&26&(8,21)&400.0\\
7&(15,7)& 500.0&17&(16,20)&1600.0&27&(7,19)&400.0\\
8&(19,9)& 400.0&18&(18,9)&200.0&28 &(10,19)&1800.0\\
9&(22,8)& 500.0&19&(17,18)&600.0&29&(16,9)&1400.0\\
10&(7,17) & 1000.0&20&(15,20)&1100.0&30&(18,15)&200.0\\
 \hline
 \end{tabular}
\end{table}

\subsection{Setting the toll arcs}
\label{app:toll_arcs}
The following procedure is used for the three networks considered. Let $t$ denote the
desired number of toll arcs; in our computations we consider $t\in\{2,4,6,8,10\}$.
First, we compute a Wardrop equilibrium assuming zero tolls on all arcs, and we let
$f_a^{\tt WE}$ denote the equilibrium flow on arc $a\in A$. Then, for each commodity
$k\in\mathcal{K}$, we compute a shortest path with respect to the arc costs
$c_a(f_a^{\tt WE})$, and we order the arcs by decreasing number of shortest paths in
which they appear, ties being broken arbitrarily. We then process the arcs in this
order. For each candidate arc, we set it as a toll arc provided that a toll-free path
still exists for every commodity afterwards; otherwise, the arc is left toll-free and
the next candidate is processed similarly. The procedure stops once $t$ toll arcs have
been set, or once all arcs have been processed. Being greedy, it offers no guarantee of
reaching $t$ toll arcs while preserving a toll-free path for every commodity, even when
such a set of toll arcs exists; on all the instances considered in our experiments,
however, the desired number $t$ was reached. Finally, since the candidate arcs are
processed in a fixed order that does not depend on $t$, and since the decision taken
for a candidate arc depends only on the arcs already selected, the procedure yields a
nested family of toll-arc sets: for $t_1<t_2$, the set of toll arcs of the instance with
$t_1$ toll arcs is contained in that of the instance with $t_2$ toll arcs.

\section{Detailed results for the impact of the toll bounds}
\label{sec:app_impact_toll_bounds}
Tables~\ref{tab:rq_toll0} and~\ref{tab:rq_toll2} give the per-instance results
underlying the summary of Table~\ref{tab:rq_tollsummary} (Section~\ref{sec:exp_bounds}),
under the tightest bound $\tfrac{1}{2}\min(c_a^{\tt fft},\tau_a^\infty)$ and the loosest
bound $\tau_a^\infty$ respectively; the reference bound
$\min(c_a^{\tt fft},\tau_a^\infty)$ is reported in Table~\ref{tab:rq1}.

\begin{sidewaystable}
\caption{\label{tab:rq_toll0} Comparison of (BLS) and (SD-A) under the tightest toll
upper bound $\tau_a^{\tt up}=\tfrac{1}{2}\min\!\left(c_a^{\tt fft},\tau_a^\infty\right)$
of~\eqref{eq:toll_ub_types} (time limit TL${}=1800$\,s). Columns are as in
Table~\ref{tab:rq1}. A dash (--) indicates that no feasible solution was found.}
\tiny
\centering
\begin{tabular}{|c|c|c||c|c|c|c|c||c|c|c|c|c|c|c|c|}
\hline
\multirow{2}{*}{$\left|\mathcal{K}\right|$} &
\multirow{2}{*}{$(\#o,\#d)$} &
\multirow{2}{*}{$\left|A^{\tt{toll}}\right|$} &
\multicolumn{5}{|c||}{${\tt{(BLS)}}$} &
\multicolumn{8}{|c|}{${\tt{(SD\text{-}A)}}$}
\\
\cline{4-16}
& & & objective & time & nodes & gap$_{\tt own}$ (\%) & gap (\%)
& objective & time & $\#p_i$ & $\#p_g$ & nodes & it & gap$_{\tt own}$ (\%) & gap (\%)\\
\hline
\multirow{5}{*}{4} & \multirow{5}{*}{(4,3)}
& 2  & 5\,980.51 & 0.4 & 25 & 0.000 & 0.000 & 5\,980.51 & 1.6 & 4 & 1 & 2 & 2 & 0.009 & 0.000 \\
& & 4  & 8\,230.51 & 0.6 & 215 & 0.000 & 0.000 & 8\,230.51 & 1.7 & 4 & 1 & 2 & 2 & 0.009 & 0.000 \\
& & 6  & 9\,580.51 & 1.0 & 814 & 0.000 & 0.000 & 9\,580.51 & 1.8 & 4 & 1 & 2 & 2 & 0.007 & 0.000 \\
& & 8  & 13\,006.05 & 0.9 & 7 & 0.000 & 0.000 & 13\,006.05 & 1.8 & 4 & 1 & 2 & 2 & 0.002 & 0.000 \\
& & 10  & 14\,850.00 & 1.0 & 242 & 0.000 & 0.000 & 14\,850.00 & 0.1 & 4 & 0 & 1 & 1 & 0.005 & 0.000 \\
\hline
\multirow{5}{*}{5} & \multirow{5}{*}{(5,4)}
& 2  & 5\,980.51 & 0.4 & 1 & 0.000 & 0.000 & 5\,980.51 & 2.0 & 5 & 1 & 2 & 2 & 0.008 & 0.000 \\
& & 4  & 8\,230.51 & 0.5 & 8 & 0.000 & 0.000 & 8\,230.51 & 1.4 & 5 & 1 & 2 & 2 & 0.009 & 0.000 \\
& & 6  & 9\,580.51 & 1.0 & 1 & 0.000 & 0.000 & 9\,580.51 & 2.1 & 5 & 1 & 2 & 2 & 0.007 & 0.000 \\
& & 8  & 13\,006.05 & 4.4 & 6\,952 & 0.009 & 0.002 & 13\,006.05 & 2.0 & 5 & 1 & 2 & 2 & 0.002 & 0.002 \\
& & 10  & 15\,150.00 & 1.7 & 919 & 0.000 & 0.000 & 15\,150.00 & 1.0 & 5 & 0 & 1 & 1 & 0.009 & 0.000 \\
\hline
\multirow{5}{*}{6} & \multirow{5}{*}{(6,5)}
& 2  & 5\,994.89 & 1.2 & 562 & 0.000 & 0.000 & 5\,994.89 & 2.8 & 6 & 3 & 3 & 3 & 0.008 & 0.000 \\
& & 4  & 11\,282.72 & 5.7 & 4\,967 & 0.000 & 0.000 & 11\,282.72 & 4.3 & 6 & 3 & 3 & 3 & 0.009 & 0.000 \\
& & 6  & 13\,082.72 & 7.6 & 7\,129 & 0.000 & 0.000 & 13\,082.72 & 4.0 & 6 & 3 & 3 & 3 & 0.009 & 0.000 \\
& & 8  & 19\,780.51 & 1.5 & 87 & 0.001 & 0.001 & 19\,780.51 & 3.2 & 7 & 1 & 2 & 2 & 0.010 & 0.001 \\
& & 10  & 26\,201.54 & 3.0 & 1\,341 & 0.000 & 0.000 & 26\,201.54 & 1.5 & 7 & 0 & 1 & 1 & 0.007 & 0.000 \\
\hline
\multirow{5}{*}{7} & \multirow{5}{*}{(7,6)}
& 2  & 4\,650.00 & 5.4 & 3\,235 & 0.000 & 0.000 & 4\,650.00 & 3.6 & 7 & 3 & 3 & 3 & 0.001 & 0.000 \\
& & 4  & 12\,016.58 & 8.6 & 7\,235 & 0.000 & 0.000 & 12\,016.58 & 5.9 & 7 & 3 & 3 & 3 & 0.009 & 0.000 \\
& & 6  & 13\,366.58 & 5.2 & 1\,834 & 0.005 & 0.005 & 13\,366.58 & 7.2 & 7 & 3 & 3 & 3 & 0.009 & 0.005 \\
& & 8  & 19\,875.28 & 7.4 & 5\,977 & 0.000 & 0.000 & 19\,873.82 & 5.6 & 7 & 1 & 2 & 2 & 0.007 & 0.007 \\
& & 10  & 20\,875.28 & 9.4 & 5\,213 & 0.007 & 0.007 & 20\,875.28 & 3.5 & 7 & 1 & 3 & 2 & 0.007 & 0.007 \\
\hline
\multirow{5}{*}{8} & \multirow{5}{*}{(8,7)}
& 2  & 4\,650.00 & 7.2 & 2\,331 & 0.000 & 0.000 & 4\,650.00 & 3.8 & 8 & 3 & 3 & 3 & 0.002 & 0.000 \\
& & 4  & 11\,457.33 & 9.3 & 4\,646 & 0.009 & 0.008 & 11\,457.33 & 3.6 & 9 & 1 & 2 & 2 & 0.008 & 0.008 \\
& & 6  & 16\,221.00 & 5.8 & 1\,963 & 0.000 & 0.000 & 16\,221.00 & 3.1 & 8 & 2 & 2 & 2 & 0.007 & 0.000 \\
& & 8  & 17\,721.00 & 6.0 & 1\,697 & 0.006 & 0.006 & 17\,721.00 & 3.7 & 8 & 2 & 2 & 2 & 0.006 & 0.006 \\
& & 10  & 24\,106.60 & 11.2 & 4\,122 & 0.000 & 0.000 & 24\,106.60 & 5.5 & 8 & 1 & 9 & 2 & 0.008 & 0.000 \\
\hline
\multirow{5}{*}{10} & \multirow{5}{*}{(10,9)}
& 2  & 7\,315.88 & 9.6 & 1\,615 & 0.000 & 0.000 & 7\,315.87 & 4.1 & 11 & 2 & 3 & 3 & 0.009 & 0.000 \\
& & 4  & 11\,315.88 & 10.0 & 1\,447 & 0.002 & 0.002 & 11\,315.87 & 6.9 & 11 & 2 & 3 & 3 & 0.007 & 0.002 \\
& & 6  & 19\,393.13 & 7.4 & 1\,397 & 0.000 & 0.000 & 19\,393.13 & 9.7 & 11 & 2 & 3 & 3 & 0.009 & 0.000 \\
& & 8  & 23\,208.88 & 8.2 & 1\,437 & 0.006 & 0.003 & 23\,208.88 & 5.5 & 11 & 1 & 2 & 2 & 0.003 & 0.003 \\
& & 10  & 24\,708.88 & 10.9 & 1\,386 & 0.001 & 0.001 & 24\,708.88 & 6.5 & 11 & 1 & 7 & 2 & 0.006 & 0.001 \\
\hline
\multirow{5}{*}{15} & \multirow{5}{*}{(12,10)}
& 2  & 7\,000.00 & 20.0 & 1\,262 & 0.000 & 0.000 & 7\,000.00 & 4.1 & 15 & 5 & 3 & 3 & 0.009 & 0.000 \\
& & 4  & 14\,400.00 & 40.0 & 10\,427 & 0.002 & 0.000 & 14\,400.00 & 6.3 & 15 & 4 & 3 & 3 & 0.009 & 0.000 \\
& & 6  & 21\,671.55 & 22.7 & 1\,368 & 0.009 & 0.005 & 21\,671.55 & 5.1 & 16 & 1 & 2 & 2 & 0.005 & 0.005 \\
& & 8  & 27\,171.55 & 26.0 & 1\,358 & 0.001 & 0.001 & 27\,171.55 & 5.8 & 16 & 1 & 2 & 2 & 0.009 & 0.001 \\
& & 10  & 32\,815.20 & 136.9 & 21\,344 & 0.000 & 0.000 & 32\,815.20 & 5.9 & 17 & 1 & 7 & 2 & 0.002 & 0.000 \\
\hline
\multirow{5}{*}{20} & \multirow{5}{*}{(12,10)}
& 2  & 15\,600.00 & 60.8 & 10\,771 & 0.000 & 0.000 & 15\,600.00 & 6.5 & 20 & 4 & 3 & 3 & 0.006 & 0.000 \\
& & 4  & 22\,800.00 & 131.8 & 20\,747 & 0.010 & 0.004 & 22\,799.99 & 5.7 & 20 & 4 & 3 & 3 & 0.004 & 0.004 \\
& & 6  & 27\,500.01 & 163.5 & 21\,212 & 0.004 & 0.004 & 27\,500.00 & 6.8 & 20 & 4 & 3 & 3 & 0.006 & 0.004 \\
& & 8  & 32\,900.00 & 223.4 & 31\,647 & 0.010 & 0.004 & 32\,900.00 & 6.5 & 20 & 4 & 3 & 3 & 0.004 & 0.004 \\
& & 10  & 40\,459.59 & TL & 242\,818 & 17.450 & 0.009 & 40\,459.59 & 5.5 & 22 & 2 & 2 & 2 & 0.009 & 0.009 \\
\hline
\multirow{5}{*}{25} & \multirow{5}{*}{(12,11)}
& 2  & 14\,250.00 & 167.4 & 20\,899 & 0.000 & 0.000 & 14\,250.00 & 4.2 & 26 & 2 & 2 & 2 & 0.000 & 0.000 \\
& & 4  & 28\,549.69 & TL & 668\,243 & 4.886 & 0.009 & 28\,549.69 & 5.2 & 26 & 3 & 2 & 2 & 0.009 & 0.009 \\
& & 6  & 34\,499.69 & 52.9 & 2\,413 & 0.006 & 0.006 & 34\,499.68 & 4.9 & 26 & 3 & 2 & 2 & 0.009 & 0.006 \\
& & 8  & 42\,899.69 & TL & 210\,119 & 16.578 & 0.009 & 42\,899.69 & 6.1 & 26 & 3 & 2 & 2 & 0.009 & 0.009 \\
& & 10  & 54\,407.86 & TL & 275\,610 & 8.081 & 0.000 & 54\,407.86 & 14.5 & 26 & 3 & 2 & 2 & 0.000 & 0.000 \\
\hline
\multirow{5}{*}{30} & \multirow{5}{*}{(12,12)}
& 2  & -- & TL & 114\,863 & -- & -- & 14\,679.70 & 8.0 & 31 & 10 & 3 & 3 & 0.000 & 0.000 \\
& & 4  & -- & TL & 111\,642 & -- & -- & 24\,991.81 & 17.9 & 31 & 10 & 231 & 3 & 0.000 & 0.000 \\
& & 6  & 33\,411.50 & TL & 135\,349 & 43.103 & 0.000 & 33\,411.50 & 17.2 & 34 & 7 & 17 & 3 & 0.000 & 0.000 \\
& & 8  & 49\,312.24 & TL & 139\,449 & 40.700 & 0.003 & 49\,312.24 & 13.8 & 34 & 7 & 1\,140 & 3 & 0.003 & 0.003 \\
& & 10  & 68\,283.57 & TL & 165\,142 & 18.510 & 0.003 & 68\,283.57 & 5.8 & 31 & 3 & 2 & 2 & 0.003 & 0.003 \\
\hline
\multirow{5}{*}{132} & \multirow{5}{*}{(12,12)}
& 2  & -- & TL & 48\,311 & -- & -- & 52\,898.13 & 5.5 & 138 & 4 & 2 & 2 & 0.009 & 0.009 \\
& & 4  & -- & TL & 45\,825 & -- & -- & 74\,595.31 & 5.2 & 134 & 8 & 2 & 2 & 0.009 & 0.009 \\
& & 6  & -- & TL & 51\,579 & -- & -- & 110\,595.31 & 4.7 & 134 & 8 & 2 & 2 & 0.010 & 0.010 \\
& & 8  & -- & TL & 45\,186 & -- & -- & 146\,895.31 & 10.1 & 133 & 9 & 2 & 2 & 0.001 & 0.001 \\
& & 10  & -- & TL & 49\,562 & -- & -- & 175\,026.07 & 11.8 & 137 & 5 & 2 & 2 & 0.008 & 0.008 \\
\hline
\end{tabular}
\end{sidewaystable}

\begin{sidewaystable}
\caption{\label{tab:rq_toll2} Comparison of (BLS) and (SD-A) under the loosest toll
upper bound $\tau_a^{\tt up}=\tau_a^\infty$ of~\eqref{eq:toll_ub_types} 
(time limit TL${}=1800$\,s). Columns are as in
Table~\ref{tab:rq1}. A dash (--) indicates that no feasible solution was found.}
\tiny
\centering
\begin{tabular}{|c|c|c||c|c|c|c|c||c|c|c|c|c|c|c|c|}
\hline
\multirow{2}{*}{$\left|\mathcal{K}\right|$} &
\multirow{2}{*}{$(\#o,\#d)$} &
\multirow{2}{*}{$\left|A^{\tt{toll}}\right|$} &
\multicolumn{5}{|c||}{${\tt{(BLS)}}$} &
\multicolumn{8}{|c|}{${\tt{(SD\text{-}A)}}$}
\\
\cline{4-16}
& & & objective & time & nodes & gap$_{\tt own}$ (\%) & gap (\%)
& objective & time & $\#p_i$ & $\#p_g$ & nodes & it & gap$_{\tt own}$ (\%) & gap (\%)\\
\hline
\multirow{5}{*}{4} & \multirow{5}{*}{(4,3)}
& 2  & 16\,266.08 & 5.3 & 3\,673 & 0.000 & 0.000 & 16\,266.08 & 2.9 & 8 & 0 & 1\,395 & 1 & 0.007 & 0.000 \\
& & 4  & 19\,207.03 & 207.6 & 524\,310 & 0.010 & 0.009 & 19\,207.03 & 7.0 & 8 & 0 & 8\,667 & 1 & 0.009 & 0.009 \\
& & 6  & 24\,621.96 & TL & 3\,584\,499 & 0.012 & 0.010 & 24\,621.95 & 0.3 & 8 & 2 & 2 & 2 & 0.010 & 0.010 \\
& & 8  & 27\,388.84 & TL & 1\,689\,498 & 92.341 & 0.145 & 15\,574.24 & TL & 8 & 0 & 177\,095 & 1 & 76.115 & 76.115 \\
& & 10  & 30\,193.95 & TL & 1\,462\,295 & 149.338 & 0.000 & 30\,193.94 & 0.6 & 8 & 10 & 5 & 5 & 0.000 & 0.000 \\
\hline
\multirow{5}{*}{5} & \multirow{5}{*}{(5,4)}
& 2  & 16\,266.08 & 6.9 & 6\,278 & 0.000 & 0.000 & 16\,266.08 & 4.4 & 9 & 0 & 2\,096 & 1 & 0.003 & 0.000 \\
& & 4  & 19\,207.03 & 50.1 & 84\,738 & 0.007 & 0.007 & 19\,207.03 & 6.1 & 9 & 0 & 6\,351 & 1 & 0.010 & 0.007 \\
& & 6  & 24\,621.96 & TL & 3\,213\,388 & 0.036 & 0.000 & 24\,621.95 & 0.3 & 9 & 2 & 2 & 2 & 0.000 & 0.000 \\
& & 8  & 27\,388.85 & TL & 1\,054\,364 & 99.651 & 0.160 & 15\,574.24 & TL & 9 & 0 & 181\,684 & 1 & 76.142 & 76.142 \\
& & 10  & 30\,585.05 & TL & 887\,490 & 169.830 & 0.000 & 30\,585.05 & 0.6 & 10 & 9 & 4 & 4 & 0.000 & 0.000 \\
\hline
\multirow{5}{*}{6} & \multirow{5}{*}{(6,5)}
& 2  & 16\,035.71 & 28.2 & 35\,245 & 0.001 & 0.001 & 16\,035.71 & 11.9 & 10 & 1 & 3\,882 & 2 & 0.005 & 0.001 \\
& & 4  & 23\,369.55 & 475.8 & 525\,186 & 0.010 & 0.010 & 23\,369.48 & TL & 12 & 1 & 467\,867 & 2 & 0.020 & 0.010 \\
& & 6  & 29\,506.91 & TL & 2\,418\,455 & 0.304 & 0.009 & 29\,506.90 & 8.4 & 12 & 1 & 1\,531 & 2 & 0.009 & 0.009 \\
& & 8  & 35\,745.20 & TL & 720\,625 & 118.200 & 0.010 & 35\,745.19 & 36.1 & 12 & 5 & 25\,574 & 2 & 0.010 & 0.010 \\
& & 10  & 43\,860.17 & TL & 483\,858 & 440.529 & 3.400 & 32\,261.96 & TL & 12 & 5 & 301\,060 & 2 & 40.572 & 40.572 \\
\hline
\multirow{5}{*}{7} & \multirow{5}{*}{(7,6)}
& 2  & 10\,977.31 & 418.6 & 1\,045\,321 & 0.010 & 0.009 & 10\,977.31 & 16.4 & 11 & 2 & 11\,095 & 2 & 0.009 & 0.009 \\
& & 4  & 24\,977.49 & TL & 840\,402 & 8.325 & 0.006 & 24\,977.49 & 700.1 & 14 & 3 & 272\,899 & 4 & 0.006 & 0.006 \\
& & 6  & 30\,657.79 & TL & 450\,381 & 104.985 & 1.303 & 23\,467.76 & TL & 14 & 0 & 224\,419 & 1 & 32.340 & 32.340 \\
& & 8  & 38\,101.82 & TL & 393\,418 & 341.561 & 2.129 & 26\,079.09 & TL & 14 & 0 & 220\,864 & 1 & 49.211 & 49.211 \\
& & 10  & 39\,176.17 & TL & 333\,357 & 675.696 & 1.861 & 32\,894.70 & TL & 14 & 6 & 250\,042 & 3 & 21.312 & 21.312 \\
\hline
\multirow{5}{*}{8} & \multirow{5}{*}{(8,7)}
& 2  & 10\,863.47 & TL & 3\,036\,057 & 0.013 & 0.006 & 10\,863.44 & 59.2 & 12 & 4 & 33\,171 & 3 & 0.006 & 0.006 \\
& & 4  & 21\,588.57 & TL & 430\,937 & 44.821 & 0.010 & 21\,588.57 & 1659.2 & 15 & 1 & 151\,807 & 2 & 0.010 & 0.010 \\
& & 6  & 34\,739.09 & TL & 257\,845 & 250.466 & 0.054 & 34\,754.58 & 1109.2 & 16 & 4 & 277\,563 & 3 & 0.010 & 0.010 \\
& & 8  & 38\,584.87 & TL & 197\,140 & 585.678 & 5.230 & 33\,315.28 & TL & 16 & 2 & 452\,464 & 2 & 21.875 & 21.875 \\
& & 10  & 45\,913.64 & TL & 226\,358 & 1076.156 & 7.894 & 30\,484.38 & TL & 16 & 2 & 164\,287 & 2 & 62.503 & 62.503 \\
\hline
\multirow{5}{*}{10} & \multirow{5}{*}{(10,9)}
& 2  & 13\,891.58 & 88.0 & 160\,200 & 0.010 & 0.009 & 13\,891.51 & 65.3 & 16 & 4 & 39\,561 & 3 & 0.009 & 0.009 \\
& & 4  & 23\,126.45 & TL & 328\,511 & 79.613 & 0.173 & 23\,112.92 & TL & 18 & 3 & 283\,005 & 2 & 0.232 & 0.232 \\
& & 6  & 34\,898.31 & TL & 251\,003 & 334.018 & 7.357 & 33\,869.16 & TL & 19 & 0 & 101\,760 & 1 & 10.619 & 10.619 \\
& & 8  & 41\,119.09 & TL & 176\,266 & 1117.696 & 18.635 & 41\,471.60 & TL & 20 & 3 & 214\,955 & 2 & 17.627 & 17.627 \\
& & 10  & 47\,326.90 & TL & 207\,058 & 1299.108 & 14.663 & 32\,801.81 & TL & 20 & 0 & 81\,384 & 1 & 65.438 & 65.438 \\
\hline
\multirow{5}{*}{15} & \multirow{5}{*}{(12,10)}
& 2  & 23\,192.78 & TL & 715\,873 & 5.151 & 0.008 & 23\,192.77 & 20.6 & 21 & 5 & 4\,204 & 3 & 0.008 & 0.008 \\
& & 4  & 87\,051.45 & TL & 133\,023 & 193.187 & 0.000 & 87\,051.44 & 567.8 & 26 & 4 & 66\,235 & 2 & 0.000 & 0.000 \\
& & 6  & 111\,287.54 & TL & 143\,717 & 448.926 & 12.212 & 112\,266.45 & TL & 29 & 3 & 202\,200 & 2 & 11.234 & 11.234 \\
& & 8  & 147\,793.20 & TL & 197\,768 & 591.225 & 6.473 & 125\,154.28 & TL & 29 & 0 & 86\,703 & 1 & 25.733 & 25.733 \\
& & 10  & 164\,377.85 & TL & 113\,495 & 1007.194 & 9.292 & 134\,768.09 & TL & 30 & 9 & 88\,641 & 2 & 33.304 & 33.304 \\
\hline
\multirow{5}{*}{20} & \multirow{5}{*}{(12,10)}
& 2  & -- & TL & 158\,918 & -- & -- & 35\,956.20 & 21.4 & 28 & 5 & 5\,008 & 3 & 0.003 & 0.003 \\
& & 4  & 43\,085.09 & TL & 148\,677 & 1139.367 & 78.612 & 74\,902.02 & TL & 32 & 0 & 85\,269 & 1 & 2.741 & 2.741 \\
& & 6  & 45\,408.54 & TL & 103\,281 & 2930.608 & 115.754 & 91\,299.90 & TL & 34 & 0 & 106\,968 & 1 & 7.307 & 7.307 \\
& & 8  & 100\,334.04 & TL & 96\,306 & 2571.770 & 18.245 & 105\,305.21 & TL & 36 & 0 & 70\,316 & 1 & 12.663 & 12.663 \\
& & 10  & 101\,787.40 & TL & 83\,423 & 4728.296 & 50.131 & 101\,713.41 & TL & 38 & 0 & 138\,684 & 1 & 50.240 & 50.240 \\
\hline
\multirow{5}{*}{25} & \multirow{5}{*}{(12,11)}
& 2  & -- & TL & 120\,121 & -- & -- & 55\,042.72 & 260.6 & 35 & 2 & 178\,785 & 2 & 0.010 & 0.010 \\
& & 4  & -- & TL & 98\,450 & -- & -- & 92\,812.63 & TL & 43 & 0 & 242\,238 & 1 & 0.181 & 0.181 \\
& & 6  & -- & TL & 87\,379 & -- & -- & 116\,111.61 & TL & 44 & 0 & 101\,123 & 1 & 8.930 & 8.930 \\
& & 8  & -- & TL & 65\,762 & -- & -- & 134\,804.36 & TL & 46 & 0 & 264\,217 & 1 & 15.552 & 15.552 \\
& & 10  & -- & TL & 77\,738 & -- & -- & 145\,375.13 & TL & 48 & 0 & 118\,474 & 1 & 34.039 & 34.039 \\
\hline
\multirow{5}{*}{30} & \multirow{5}{*}{(12,12)}
& 2  & -- & TL & 84\,763 & -- & -- & 41\,512.14 & 61.2 & 43 & 13 & 22\,399 & 4 & 0.007 & 0.007 \\
& & 4  & -- & TL & 57\,403 & -- & -- & 50\,661.67 & TL & 47 & 0 & 196\,770 & 1 & 23.922 & 23.922 \\
& & 6  & -- & TL & 67\,650 & -- & -- & 184\,050.69 & TL & 54 & 0 & 347\,606 & 1 & 11.407 & 11.407 \\
& & 8  & 237\,407.03 & TL & 53\,460 & 2475.392 & 28.233 & 233\,314.05 & TL & 57 & 0 & 320\,610 & 1 & 30.483 & 30.483 \\
& & 10  & 288\,699.19 & TL & 78\,123 & 4485.315 & 47.186 & 175\,152.70 & TL & 58 & 0 & 327\,153 & 1 & 142.603 & 142.603 \\
\hline
\multirow{5}{*}{132} & \multirow{5}{*}{(12,12)}
& 2  & -- & TL & 44\,033 & -- & -- & 279\,723.36 & 207.7 & 178 & 14 & 9\,327 & 2 & 0.000 & 0.000 \\
& & 4  & -- & TL & 36\,154 & -- & -- & 432\,303.88 & TL & 186 & 0 & 340\,605 & 1 & 10.141 & 10.141 \\
& & 6  & -- & TL & 36\,594 & -- & -- & 633\,058.67 & TL & 206 & 0 & 272\,405 & 1 & 27.975 & 27.975 \\
& & 8  & -- & TL & 41\,203 & -- & -- & 1\,192\,205.05 & TL & 217 & 0 & 277\,444 & 1 & 79.274 & 79.274 \\
& & 10  & -- & TL & 45\,657 & -- & -- & 1\,860\,962.73 & TL & 227 & 0 & 285\,897 & 1 & 45.614 & 45.614 \\
\hline
\end{tabular}
\end{sidewaystable}

\section{Detailed results for the impact of the cost structure on toll arcs}
\label{sec:app_impact_tollcosttype}
Table~\ref{tab:rq_cost1} gives the per-instance results for the original toll-arc cost
coefficients $(\alpha_a,c_a^{\tt fft})$, underlying the summary of
Table~\ref{tab:rq_costsummary} (Section~\ref{sec:exp_costs}); the results with the reference coefficients
$(\alpha_a/2,c_a^{\tt fft}/2)$ are reported in Table~\ref{tab:rq1}.
\begin{sidewaystable}
\caption{\label{tab:rq_cost1} Comparison of (BLS) and (SD-A) when the original
cost coefficients of the toll arcs are retained, i.e.\ without halving $\alpha_a$
and $c_a^{\tt fft}$ (time limit TL${}=1800$\,s). The toll upper bound is the
reference one, $\tau_a^{\tt up}=\min(c_a^{\tt fft},\tau_a^\infty)$. Columns are as
in Table~\ref{tab:rq1}. A dash (--) indicates that no feasible solution was found.}
\tiny
\centering
\begin{tabular}{|c|c|c||c|c|c|c|c||c|c|c|c|c|c|c|c|}
\hline
\multirow{2}{*}{$\left|\mathcal{K}\right|$} &
\multirow{2}{*}{$(\#o,\#d)$} &
\multirow{2}{*}{$\left|A^{\tt{toll}}\right|$} &
\multicolumn{5}{|c||}{${\tt{(BLS)}}$} &
\multicolumn{8}{|c|}{${\tt{(SD\text{-}A)}}$}
\\
\cline{4-16}
& & & objective & time & nodes & gap$_{\tt own}$ (\%) & gap (\%)
& objective & time & $\#p_i$ & $\#p_g$ & nodes & it & gap$_{\tt own}$ (\%) & gap (\%)\\
\hline
\multirow{5}{*}{4} & \multirow{5}{*}{(4,3)}
& 2  & 8\,954.63 & 1.1 & 377 & 0.000 & 0.000 & 8\,954.63 & 2.2 & 5 & 1 & 2 & 2 & 0.000 & 0.000 \\
& & 4  & 10\,677.56 & 33.7 & 88\,288 & 0.010 & 0.006 & 10\,677.56 & 9.6 & 6 & 1 & 2\,477 & 2 & 0.006 & 0.006 \\
& & 6  & 13\,377.56 & 34.9 & 55\,637 & 0.009 & 0.009 & 13\,377.56 & 12.5 & 6 & 2 & 1\,823 & 2 & 0.010 & 0.009 \\
& & 8  & 14\,122.90 & 1009.2 & 2\,364\,954 & 0.010 & 0.008 & 14\,122.89 & 7.4 & 7 & 0 & 4\,315 & 1 & 0.008 & 0.008 \\
& & 10  & 15\,259.98 & TL & 2\,713\,898 & 0.238 & 0.000 & 15\,259.98 & 0.3 & 6 & 3 & 2 & 2 & 0.000 & 0.000 \\
\hline
\multirow{5}{*}{5} & \multirow{5}{*}{(5,4)}
& 2  & 8\,954.63 & 2.1 & 2\,250 & 0.000 & 0.000 & 8\,954.63 & 1.4 & 6 & 1 & 2 & 2 & 0.001 & 0.000 \\
& & 4  & 10\,677.56 & 14.4 & 25\,062 & 0.008 & 0.007 & 10\,677.56 & 8.1 & 7 & 1 & 1\,891 & 2 & 0.007 & 0.007 \\
& & 6  & 13\,377.56 & 14.3 & 22\,351 & 0.009 & 0.004 & 13\,377.56 & 11.5 & 7 & 2 & 1\,580 & 2 & 0.004 & 0.004 \\
& & 8  & 14\,122.90 & TL & 3\,682\,431 & 0.013 & 0.010 & 14\,122.89 & 14.5 & 8 & 0 & 8\,339 & 1 & 0.010 & 0.010 \\
& & 10  & 15\,342.20 & TL & 2\,124\,802 & 0.360 & 0.000 & 15\,342.19 & 2.9 & 9 & 1 & 2 & 2 & 0.000 & 0.000 \\
\hline
\multirow{5}{*}{6} & \multirow{5}{*}{(6,5)}
& 2  & 8\,824.98 & 3.9 & 1\,744 & 0.000 & 0.000 & 8\,824.98 & 2.7 & 7 & 2 & 2 & 2 & 0.000 & 0.000 \\
& & 4  & 13\,475.12 & 21.5 & 32\,018 & 0.010 & 0.003 & 13\,475.12 & 10.6 & 8 & 4 & 916 & 3 & 0.003 & 0.003 \\
& & 6  & 14\,387.88 & 88.5 & 149\,888 & 0.010 & 0.007 & 14\,387.88 & 12.7 & 9 & 3 & 817 & 3 & 0.007 & 0.007 \\
& & 8  & 16\,920.57 & 113.4 & 132\,312 & 0.010 & 0.005 & 16\,920.56 & 13.8 & 9 & 5 & 2\,351 & 4 & 0.005 & 0.005 \\
& & 10  & 19\,784.83 & TL & 1\,296\,282 & 0.438 & 0.066 & 19\,784.80 & TL & 8 & 6 & 417\,902 & 3 & 0.066 & 0.066 \\
\hline
\multirow{5}{*}{7} & \multirow{5}{*}{(7,6)}
& 2  & 6\,605.61 & 8.9 & 14\,012 & 0.001 & 0.001 & 6\,605.61 & 10.4 & 8 & 3 & 3\,112 & 2 & 0.009 & 0.001 \\
& & 4  & 13\,792.25 & 30.5 & 23\,302 & 0.004 & 0.000 & 13\,792.25 & 29.2 & 9 & 3 & 4\,779 & 3 & 0.000 & 0.000 \\
& & 6  & 16\,492.25 & 52.0 & 45\,407 & 0.009 & 0.008 & 16\,492.25 & 25.9 & 9 & 4 & 6\,783 & 3 & 0.008 & 0.008 \\
& & 8  & 18\,451.77 & 224.5 & 141\,886 & 0.008 & 0.007 & 18\,451.77 & 16.6 & 9 & 6 & 5\,900 & 5 & 0.007 & 0.007 \\
& & 10  & 18\,786.83 & TL & 893\,595 & 1.812 & 1.113 & 18\,533.94 & TL & 10 & 4 & 383\,356 & 3 & 2.493 & 2.493 \\
\hline
\multirow{5}{*}{8} & \multirow{5}{*}{(8,7)}
& 2  & 6\,553.21 & 25.5 & 24\,294 & 0.006 & 0.001 & 6\,553.21 & 14.8 & 9 & 5 & 3\,502 & 3 & 0.001 & 0.001 \\
& & 4  & 10\,715.73 & 235.7 & 166\,023 & 0.009 & 0.009 & 10\,715.73 & 178.0 & 11 & 2 & 148\,773 & 2 & 0.010 & 0.009 \\
& & 6  & 16\,570.17 & TL & 912\,082 & 0.754 & 0.002 & 16\,570.17 & 23.7 & 11 & 4 & 3\,865 & 3 & 0.002 & 0.002 \\
& & 8  & 19\,570.17 & TL & 835\,973 & 7.763 & 0.002 & 19\,570.17 & 9.0 & 11 & 4 & 2\,676 & 3 & 0.002 & 0.002 \\
& & 10  & 24\,828.86 & TL & 625\,603 & 34.511 & 0.021 & 24\,828.80 & TL & 12 & 3 & 801\,136 & 3 & 0.021 & 0.021 \\
\hline
\multirow{5}{*}{10} & \multirow{5}{*}{(10,9)}
& 2  & 7\,356.59 & 15.3 & 10\,441 & 0.000 & 0.000 & 7\,356.59 & 5.2 & 12 & 4 & 3 & 3 & 0.009 & 0.000 \\
& & 4  & 13\,051.88 & 67.2 & 33\,358 & 0.008 & 0.002 & 13\,051.88 & 17.3 & 12 & 5 & 2\,191 & 3 & 0.002 & 0.002 \\
& & 6  & 17\,585.21 & TL & 598\,405 & 11.199 & 0.423 & 17\,561.89 & TL & 14 & 4 & 212\,167 & 2 & 0.556 & 0.556 \\
& & 8  & 20\,953.78 & TL & 476\,752 & 26.516 & 3.059 & 20\,894.69 & TL & 15 & 3 & 259\,937 & 2 & 3.350 & 3.350 \\
& & 10  & 23\,953.78 & TL & 548\,774 & 9.481 & 2.597 & 23\,894.69 & TL & 15 & 3 & 290\,736 & 2 & 2.851 & 2.851 \\
\hline
\multirow{5}{*}{15} & \multirow{5}{*}{(12,10)}
& 2  & 12\,512.77 & 136.0 & 62\,437 & 0.001 & 0.000 & 12\,512.77 & 9.5 & 15 & 7 & 9 & 3 & 0.000 & 0.000 \\
& & 4  & 26\,595.75 & 260.4 & 83\,474 & 0.006 & 0.006 & 26\,595.75 & 4.1 & 16 & 3 & 2 & 2 & 0.006 & 0.006 \\
& & 6  & 34\,071.18 & 1758.0 & 575\,647 & 0.010 & 0.008 & 34\,071.18 & 20.9 & 18 & 5 & 1\,196 & 3 & 0.008 & 0.008 \\
& & 8  & 43\,994.02 & TL & 321\,998 & 6.404 & 0.005 & 43\,994.02 & 15.3 & 18 & 7 & 144 & 3 & 0.005 & 0.005 \\
& & 10  & 54\,397.08 & TL & 246\,295 & 17.549 & 0.006 & 54\,397.08 & 12.3 & 20 & 6 & 770 & 3 & 0.006 & 0.006 \\
\hline
\multirow{5}{*}{20} & \multirow{5}{*}{(12,10)}
& 2  & 18\,238.34 & TL & 741\,509 & 2.095 & 0.010 & 18\,238.34 & 513.9 & 22 & 6 & 680\,089 & 3 & 0.010 & 0.010 \\
& & 4  & 34\,194.12 & TL & 203\,762 & 26.730 & 0.022 & 34\,194.12 & TL & 22 & 7 & 1\,854\,752 & 4 & 0.022 & 0.022 \\
& & 6  & 40\,197.45 & TL & 185\,476 & 36.122 & 0.010 & 40\,197.44 & 50.4 & 22 & 6 & 1\,675 & 3 & 0.010 & 0.010 \\
& & 8  & 50\,997.45 & TL & 209\,590 & 63.332 & 0.009 & 50\,997.45 & 49.4 & 22 & 6 & 1\,748 & 3 & 0.009 & 0.009 \\
& & 10  & 59\,088.53 & TL & 179\,680 & 75.196 & 0.244 & 59\,211.15 & TL & 24 & 6 & 331\,036 & 3 & 0.037 & 0.037 \\
\hline
\multirow{5}{*}{25} & \multirow{5}{*}{(12,11)}
& 2  & 24\,613.92 & TL & 143\,329 & 10.970 & 0.010 & 24\,613.92 & 4.9 & 27 & 4 & 2 & 2 & 0.010 & 0.010 \\
& & 4  & 40\,223.94 & TL & 259\,501 & 84.732 & 0.016 & 40\,223.94 & TL & 26 & 12 & 2\,799\,759 & 3 & 0.016 & 0.016 \\
& & 6  & 47\,699.32 & TL & 163\,439 & 67.300 & 0.010 & 47\,699.43 & 314.6 & 30 & 7 & 280\,774 & 3 & 0.010 & 0.010 \\
& & 8  & 64\,484.54 & TL & 153\,498 & 63.699 & 0.033 & 64\,499.43 & 75.4 & 29 & 9 & 52\,810 & 2 & 0.010 & 0.010 \\
& & 10  & 73\,663.86 & TL & 137\,479 & 89.264 & 0.010 & 73\,663.86 & 89.2 & 28 & 11 & 34\,200 & 3 & 0.010 & 0.010 \\
\hline
\multirow{5}{*}{30} & \multirow{5}{*}{(12,12)}
& 2  & -- & TL & 85\,993 & -- & -- & 23\,377.16 & 7.0 & 30 & 11 & 3 & 3 & 0.007 & 0.007 \\
& & 4  & -- & TL & 104\,473 & -- & -- & 28\,937.38 & 86.6 & 34 & 10 & 44\,059 & 3 & 0.009 & 0.009 \\
& & 6  & -- & TL & 88\,880 & -- & -- & 46\,596.39 & 59.0 & 33 & 11 & 6\,361 & 3 & 0.010 & 0.010 \\
& & 8  & -- & TL & 82\,077 & -- & -- & 68\,708.84 & TL & 38 & 7 & 237\,423 & 2 & 0.057 & 0.057 \\
& & 10  & -- & TL & 81\,841 & -- & -- & 81\,950.63 & 152.0 & 38 & 7 & 27\,902 & 3 & 0.007 & 0.007 \\
\hline
\multirow{5}{*}{132} & \multirow{5}{*}{(12,12)}
& 2  & -- & TL & 42\,061 & -- & -- & 84\,338.86 & 7.5 & 132 & 14 & 3 & 3 & 0.000 & 0.000 \\
& & 4  & -- & TL & 42\,050 & -- & -- & 117\,720.93 & 12.1 & 136 & 12 & 13 & 3 & 0.000 & 0.000 \\
& & 6  & -- & TL & 42\,917 & -- & -- & 197\,414.38 & 33.6 & 140 & 16 & 16 & 4 & 0.000 & 0.000 \\
& & 8  & -- & TL & 44\,923 & -- & -- & 260\,490.63 & 32.3 & 140 & 12 & 332 & 3 & 0.000 & 0.000 \\
& & 10  & -- & TL & 51\,018 & -- & -- & 316\,419.91 & 25.2 & 136 & 10 & 115 & 2 & 0.001 & 0.001 \\
\hline
\end{tabular}
\end{sidewaystable}

\section{Detailed results for the heuristics}
\label{sec:app_heur}

This appendix reports the full per-instance results for the two path-based heuristics
of Section~\ref{sec:exp_heuristic}: Table~\ref{tab:rq4} for the reference configuration,
and Tables~\ref{tab:rq5_pathmin} and~\ref{tab:rq5_pathmax} for the initialization 
ablation summarized in Table~\ref{tab:rq5_init}.
\begin{sidewaystable}
\caption{\label{tab:rq4} Path-based heuristics (SD-P) and (CC) (time limit
TL${}=1800$\,s). The ``obj'' field reports the valid lower bound $\underline{z}$ on
$z^\star$ produced by the heuristic, and ``gap'' the certified gap
$(\overline{z}-\underline{z})/\underline{z}$ obtained with the pooled instance-level
bound $\overline{z}$ of Section~\ref{sec:exp_exact}; see
Section~\ref{sec:exp_heuristic} for details. The remaining fields are as described
for Table~\ref{tab:rq1}. 
}
\tiny
\centering
\begin{tabular}{|c|c|c||c|c|c|c|c|c|c||c|c|c|c|c|c|c|}
\hline
\multirow{2}{*}{$\left|\mathcal{K}\right|$} &
\multirow{2}{*}{$(\#o,\#d)$} &
\multirow{2}{*}{$\left|A^{\tt{toll}}\right|$} &
\multicolumn{7}{|c||}{(SD-P)} &
\multicolumn{7}{|c|}{(CC)}
\\
\cline{4-17}
& & &
obj & time & $\#p_i$ & $\#p_g$ & nodes & it & gap &
obj & time & $\#p_i$ & $\#p_g$ & nodes & it & gap \\
\hline
\multirow{5}{*}{4} & \multirow{5}{*}{(4,3)}
& 2  & 11\,787.07 & 1.6 & 4 & 2 & 2 & 3 & 0.000 & 11\,787.08 & 0.1 & 4 & 2 & 2 & 3 & 0.000 \\
& & 4  & 16\,287.08 & 1.5 & 4 & 2 & 2 & 3 & 0.000 & 16\,287.08 & 0.1 & 4 & 2 & 2 & 3 & 0.000 \\
& & 6  & 18\,987.08 & 1.3 & 4 & 2 & 2 & 3 & 0.000 & 18\,987.08 & 0.1 & 4 & 2 & 2 & 3 & 0.000 \\
& & 8  & 23\,870.26 & 1.0 & 5 & 1 & 2 & 2 & 0.004 & 23\,870.26 & 0.1 & 5 & 1 & 6 & 2 & 0.004 \\
& & 10  & 26\,212.69 & 2.0 & 5 & 4 & 4 & 4 & 0.003 & 26\,212.69 & 0.2 & 5 & 3 & 37 & 3 & 0.003 \\
\hline
\multirow{5}{*}{5} & \multirow{5}{*}{(5,4)}
& 2  & 11\,787.07 & 1.5 & 5 & 2 & 2 & 3 & 0.000 & 11\,787.08 & 0.1 & 5 & 2 & 2 & 3 & 0.000 \\
& & 4  & 16\,287.08 & 1.5 & 5 & 2 & 2 & 3 & 0.000 & 16\,287.08 & 0.1 & 5 & 2 & 2 & 3 & 0.000 \\
& & 6  & 18\,987.08 & 1.4 & 5 & 2 & 2 & 3 & 0.000 & 18\,987.08 & 0.1 & 5 & 2 & 2 & 3 & 0.000 \\
& & 8  & 23\,870.26 & 1.0 & 6 & 1 & 2 & 2 & 0.001 & 23\,870.26 & 0.1 & 6 & 1 & 6 & 2 & 0.001 \\
& & 10  & 26\,603.80 & 3.0 & 7 & 5 & 4 & 4 & 0.000 & 26\,603.80 & 0.2 & 7 & 3 & 37 & 3 & 0.000 \\
\hline
\multirow{5}{*}{6} & \multirow{5}{*}{(6,5)}
& 2  & 11\,743.58 & 2.8 & 6 & 5 & 3 & 4 & 0.003 & 11\,743.58 & 0.2 & 6 & 5 & 12 & 4 & 0.003 \\
& & 4  & 19\,199.33 & 9.8 & 6 & 6 & 8\,806 & 5 & 0.005 & 19\,199.33 & 0.3 & 6 & 6 & 128 & 5 & 0.005 \\
& & 6  & 22\,799.33 & 9.8 & 6 & 6 & 8\,737 & 5 & 0.005 & 22\,799.33 & 0.3 & 6 & 6 & 128 & 5 & 0.005 \\
& & 8  & 30\,343.01 & 21.4 & 8 & 4 & 69\,395 & 3 & 0.006 & 30\,343.01 & 0.5 & 8 & 4 & 1\,271 & 3 & 0.006 \\
& & 10  & 39\,618.66 & 2.9 & 8 & 4 & 3 & 3 & 0.011 & 39\,619.54 & 0.6 & 8 & 4 & 2\,168 & 3 & 0.009 \\
\hline
\multirow{5}{*}{7} & \multirow{5}{*}{(7,6)}
& 2  & 8\,593.37 & 2.0 & 8 & 3 & 37 & 3 & 0.004 & 8\,593.37 & 0.1 & 8 & 3 & 32 & 3 & 0.004 \\
& & 4  & 19\,477.39 & 3.7 & 8 & 5 & 5 & 5 & 0.004 & 19\,477.39 & 0.4 & 8 & 5 & 141 & 5 & 0.004 \\
& & 6  & 22\,177.39 & 3.6 & 8 & 5 & 5 & 5 & 0.007 & 22\,177.39 & 0.4 & 8 & 5 & 141 & 5 & 0.007 \\
& & 8  & 28\,931.56 & 7.3 & 9 & 3 & 7\,738 & 3 & 0.010 & 28\,931.56 & 0.5 & 9 & 3 & 568 & 3 & 0.010 \\
& & 10  & 31\,005.14 & 1.8 & 9 & 2 & 2 & 2 & 0.010 & 31\,005.14 & 0.3 & 9 & 2 & 150 & 2 & 0.010 \\
\hline
\multirow{5}{*}{8} & \multirow{5}{*}{(8,7)}
& 2  & 8\,612.07 & 1.9 & 9 & 3 & 3 & 3 & 0.005 & 8\,612.07 & 0.1 & 9 & 3 & 7 & 3 & 0.005 \\
& & 4  & 18\,540.89 & 2.0 & 11 & 2 & 2 & 2 & 0.007 & 18\,540.89 & 0.2 & 11 & 2 & 19 & 2 & 0.007 \\
& & 6  & 26\,508.83 & 21.5 & 9 & 6 & 41\,095 & 4 & 0.009 & 26\,508.83 & 0.4 & 9 & 6 & 488 & 4 & 0.009 \\
& & 8  & 29\,508.83 & 17.1 & 9 & 6 & 28\,783 & 4 & 0.010 & 29\,508.71 & 0.4 & 9 & 6 & 710 & 4 & 0.010 \\
& & 10  & 38\,825.90 & 475.4 & 10 & 5 & 1\,523\,376 & 4 & 0.881 & 38\,825.90 & 0.8 & 10 & 5 & 1\,176 & 4 & 0.881 \\
\hline
\multirow{5}{*}{10} & \multirow{5}{*}{(10,9)}
& 2  & 11\,055.08 & 5.0 & 12 & 3 & 618 & 3 & 0.003 & 11\,055.09 & 0.2 & 12 & 3 & 29 & 3 & 0.003 \\
& & 4  & 17\,863.25 & 7.4 & 12 & 3 & 4\,845 & 3 & 0.006 & 17\,863.25 & 0.3 & 12 & 3 & 309 & 3 & 0.006 \\
& & 6  & 30\,245.72 & 2.8 & 13 & 2 & 2 & 2 & 0.004 & 30\,245.72 & 0.3 & 13 & 2 & 266 & 2 & 0.004 \\
& & 8  & 37\,612.33 & 87.1 & 13 & 4 & 98\,017 & 3 & 0.009 & 37\,612.33 & 0.6 & 13 & 4 & 305 & 3 & 0.009 \\
& & 10  & 40\,612.33 & 79.5 & 13 & 4 & 88\,389 & 3 & 0.005 & 40\,612.33 & 0.5 & 13 & 4 & 281 & 3 & 0.005 \\
\hline
\multirow{5}{*}{15} & \multirow{5}{*}{(12,10)}
& 2  & 14\,000.00 & 0.1 & 15 & 5 & 0 & 3 & 0.008 & 14\,000.00 & 0.1 & 15 & 5 & 1 & 3 & 0.008 \\
& & 4  & 28\,362.90 & 1.2 & 15 & 5 & 2 & 3 & 0.007 & 28\,362.90 & 0.1 & 15 & 5 & 2 & 3 & 0.007 \\
& & 6  & 39\,606.76 & 1.9 & 17 & 4 & 3 & 3 & 0.000 & 39\,606.76 & 0.2 & 17 & 4 & 3 & 3 & 0.000 \\
& & 8  & 49\,379.91 & 5.7 & 17 & 4 & 4\,012 & 3 & 0.007 & 49\,379.91 & 0.3 & 17 & 4 & 236 & 3 & 0.007 \\
& & 10  & 61\,791.00 & 5.9 & 18 & 6 & 897 & 4 & 0.003 & 61\,791.00 & 0.5 & 18 & 6 & 246 & 4 & 0.003 \\
\hline
\multirow{5}{*}{20} & \multirow{5}{*}{(12,10)}
& 2  & 29\,695.88 & 1.8 & 20 & 6 & 2 & 3 & 0.003 & 29\,695.88 & 0.1 & 20 & 6 & 2 & 3 & 0.003 \\
& & 4  & 45\,359.25 & 1.5 & 20 & 6 & 2 & 3 & 0.009 & 45\,359.25 & 0.1 & 20 & 6 & 2 & 3 & 0.009 \\
& & 6  & 54\,452.42 & 1.6 & 20 & 7 & 2 & 3 & 0.001 & 54\,452.42 & 0.1 & 20 & 7 & 2 & 3 & 0.001 \\
& & 8  & 65\,252.42 & 1.5 & 20 & 7 & 2 & 3 & 0.000 & 65\,252.42 & 0.1 & 20 & 7 & 2 & 3 & 0.000 \\
& & 10  & 78\,057.41 & 3.0 & 22 & 6 & 3 & 3 & 0.001 & 78\,057.41 & 0.2 & 22 & 6 & 3 & 3 & 0.001 \\
\hline
\multirow{5}{*}{25} & \multirow{5}{*}{(12,11)}
& 2  & 27\,009.85 & 2.3 & 27 & 5 & 3 & 3 & 0.009 & 27\,009.92 & 0.1 & 27 & 5 & 3 & 3 & 0.009 \\
& & 4  & 54\,577.74 & 2.3 & 26 & 7 & 3 & 3 & 0.010 & 54\,577.74 & 0.2 & 26 & 7 & 3 & 3 & 0.010 \\
& & 6  & 66\,477.74 & 2.1 & 26 & 7 & 3 & 3 & 0.010 & 66\,477.74 & 0.2 & 26 & 7 & 3 & 3 & 0.010 \\
& & 8  & 83\,277.74 & 1.9 & 26 & 8 & 3 & 3 & 0.009 & 83\,277.74 & 0.2 & 26 & 8 & 3 & 3 & 0.009 \\
& & 10  & 98\,674.11 & 3.1 & 27 & 4 & 1\,066 & 2 & 0.009 & 98\,674.11 & 0.2 & 27 & 4 & 63 & 2 & 0.009 \\
\hline
\multirow{5}{*}{30} & \multirow{5}{*}{(12,12)}
& 2  & 26\,666.18 & 3.5 & 33 & 13 & 4 & 4 & 0.010 & 26\,666.18 & 0.6 & 33 & 13 & 112 & 4 & 0.010 \\
& & 4  & 39\,102.74 & 17.9 & 36 & 6 & 24\,316 & 2 & 0.008 & 39\,102.74 & 0.7 & 36 & 6 & 703 & 2 & 0.008 \\
& & 6  & 56\,712.73 & 16.7 & 37 & 5 & 18\,675 & 2 & 0.010 & 56\,712.73 & 1.0 & 37 & 5 & 1\,388 & 2 & 0.010 \\
& & 8  & 85\,373.21 & 20.1 & 37 & 6 & 6\,635 & 2 & 0.010 & 85\,373.22 & 1.4 & 37 & 6 & 4\,010 & 2 & 0.010 \\
& & 10  & 107\,438.16 & 141.2 & 36 & 7 & 346\,748 & 2 & 0.079 & 107\,438.16 & 2.9 & 36 & 7 & 9\,705 & 2 & 0.079 \\
\hline
\multirow{5}{*}{132} & \multirow{5}{*}{(12,12)}
& 2  & 97\,020.71 & 1.8 & 138 & 4 & 2 & 2 & 0.009 & 97\,020.71 & 0.4 & 138 & 4 & 34 & 2 & 0.009 \\
& & 4  & 134\,696.98 & 1.8 & 138 & 4 & 2 & 2 & 0.006 & 134\,696.98 & 0.4 & 138 & 4 & 22 & 2 & 0.006 \\
& & 6  & 206\,690.32 & 2.4 & 138 & 6 & 3 & 3 & 0.000 & 206\,690.32 & 0.7 & 138 & 6 & 43 & 3 & 0.000 \\
& & 8  & 279\,290.32 & 2.7 & 137 & 6 & 3 & 3 & 0.002 & 279\,290.32 & 0.7 & 137 & 6 & 22 & 3 & 0.002 \\
& & 10  & 335\,336.96 & 1.7 & 141 & 2 & 2 & 2 & 0.000 & 335\,336.96 & 0.6 & 141 & 2 & 29 & 2 & 0.000 \\
\hline
\end{tabular}
\end{sidewaystable}
\begin{sidewaystable}
\caption{\label{tab:rq5_pathmin} Path-based formulations (SD-P) and (CC) initialized
with the paths of the zero-toll equilibrium only, i.e.\ $\tau_a=0$ for all $a\in A$
(reference configuration, time limit TL${}=1800$\,s). Columns are as in
Table~\ref{tab:rq4}.}
\tiny
\centering
\begin{tabular}{|c|c|c||c|c|c|c|c|c|c||c|c|c|c|c|c|c|}
\hline
\multirow{2}{*}{$\left|\mathcal{K}\right|$} &
\multirow{2}{*}{$(\#o,\#d)$} &
\multirow{2}{*}{$\left|A^{\tt{toll}}\right|$} &
\multicolumn{7}{|c||}{(SD-P)} &
\multicolumn{7}{|c|}{(CC)}
\\
\cline{4-17}
& & &
obj & time & $\#p_i$ & $\#p_g$ & nodes & it & gap &
obj & time & $\#p_i$ & $\#p_g$ & nodes & it & gap \\
\hline
\multirow{5}{*}{4} & \multirow{5}{*}{(4,3)}
& 2  & 11\,787.07 & 1.5 & 4 & 2 & 2 & 3 & 0.000 & 11\,787.08 & 0.1 & 4 & 2 & 2 & 3 & 0.000 \\
& & 4  & 16\,287.08 & 1.4 & 4 & 2 & 2 & 3 & 0.000 & 16\,287.08 & 0.1 & 4 & 2 & 2 & 3 & 0.000 \\
& & 6  & 18\,987.08 & 1.3 & 4 & 2 & 2 & 3 & 0.000 & 18\,987.08 & 0.1 & 4 & 2 & 2 & 3 & 0.000 \\
& & 8  & 23\,870.26 & 0.7 & 4 & 2 & 1 & 2 & 0.004 & 23\,870.26 & 0.1 & 4 & 2 & 5 & 2 & 0.004 \\
& & 10  & 26\,212.69 & 2.2 & 4 & 4 & 3 & 4 & 0.003 & 26\,212.69 & 0.2 & 4 & 4 & 35 & 4 & 0.003 \\
\hline
\multirow{5}{*}{5} & \multirow{5}{*}{(5,4)}
& 2  & 11\,787.07 & 1.6 & 5 & 2 & 2 & 3 & 0.000 & 11\,787.08 & 0.1 & 5 & 2 & 2 & 3 & 0.000 \\
& & 4  & 16\,287.08 & 1.5 & 5 & 2 & 2 & 3 & 0.000 & 16\,287.08 & 0.1 & 5 & 2 & 2 & 3 & 0.000 \\
& & 6  & 18\,987.08 & 1.3 & 5 & 2 & 2 & 3 & 0.000 & 18\,987.08 & 0.1 & 5 & 2 & 2 & 3 & 0.000 \\
& & 8  & 23\,870.26 & 0.8 & 5 & 2 & 1 & 2 & 0.001 & 23\,870.26 & 0.1 & 5 & 2 & 5 & 2 & 0.001 \\
& & 10  & 26\,603.80 & 2.3 & 5 & 6 & 3 & 4 & 0.000 & 26\,603.80 & 0.2 & 5 & 5 & 35 & 4 & 0.000 \\
\hline
\multirow{5}{*}{6} & \multirow{5}{*}{(6,5)}
& 2  & 11\,743.58 & 3.0 & 6 & 5 & 3 & 4 & 0.003 & 11\,743.58 & 0.2 & 6 & 5 & 12 & 4 & 0.003 \\
& & 4  & 19\,199.33 & 9.5 & 6 & 6 & 8\,806 & 5 & 0.005 & 19\,199.33 & 0.3 & 6 & 6 & 128 & 5 & 0.005 \\
& & 6  & 22\,799.33 & 10.0 & 6 & 6 & 8\,737 & 5 & 0.005 & 22\,799.33 & 0.3 & 6 & 6 & 128 & 5 & 0.005 \\
& & 8  & 30\,343.01 & 264.2 & 6 & 6 & 547\,696 & 4 & 0.006 & 30\,343.01 & 0.6 & 6 & 6 & 2\,057 & 4 & 0.006 \\
& & 10  & 39\,618.55 & 8.7 & 6 & 6 & 11\,123 & 4 & 0.011 & 39\,619.54 & 0.6 & 6 & 6 & 1\,818 & 4 & 0.009 \\
\hline
\multirow{5}{*}{7} & \multirow{5}{*}{(7,6)}
& 2  & 8\,593.37 & 1.4 & 7 & 4 & 2 & 3 & 0.004 & 8\,593.37 & 0.1 & 7 & 4 & 31 & 3 & 0.004 \\
& & 4  & 19\,477.39 & 4.2 & 7 & 6 & 506 & 5 & 0.004 & 19\,477.39 & 0.4 & 7 & 6 & 167 & 5 & 0.004 \\
& & 6  & 22\,177.39 & 4.3 & 7 & 6 & 506 & 5 & 0.007 & 22\,177.39 & 0.4 & 7 & 6 & 167 & 5 & 0.007 \\
& & 8  & 28\,931.56 & 11.0 & 7 & 5 & 9\,370 & 4 & 0.010 & 28\,931.56 & 0.5 & 7 & 5 & 1\,043 & 4 & 0.010 \\
& & 10  & 31\,005.14 & 12.1 & 7 & 5 & 14\,388 & 3 & 0.010 & 31\,005.14 & 0.4 & 7 & 5 & 630 & 3 & 0.010 \\
\hline
\multirow{5}{*}{8} & \multirow{5}{*}{(8,7)}
& 2  & 8\,612.07 & 1.7 & 8 & 5 & 2 & 3 & 0.005 & 8\,612.07 & 0.2 & 8 & 5 & 24 & 3 & 0.005 \\
& & 4  & 18\,540.89 & 6.9 & 8 & 8 & 5\,176 & 4 & 0.007 & 18\,540.89 & 0.4 & 8 & 8 & 157 & 4 & 0.007 \\
& & 6  & 26\,508.83 & 20.0 & 8 & 7 & 38\,389 & 4 & 0.009 & 26\,508.83 & 0.5 & 8 & 7 & 1\,087 & 4 & 0.009 \\
& & 8  & 29\,508.83 & 17.3 & 8 & 7 & 33\,749 & 4 & 0.010 & 29\,508.83 & 0.4 & 8 & 7 & 573 & 4 & 0.010 \\
& & 10  & 38\,825.90 & 354.4 & 8 & 7 & 1\,354\,519 & 4 & 0.881 & 38\,825.90 & 0.8 & 8 & 7 & 1\,099 & 4 & 0.881 \\
\hline
\multirow{5}{*}{10} & \multirow{5}{*}{(10,9)}
& 2  & 11\,055.09 & 2.6 & 10 & 5 & 42 & 3 & 0.003 & 11\,055.09 & 0.2 & 10 & 5 & 31 & 3 & 0.003 \\
& & 4  & 17\,863.25 & 6.1 & 10 & 5 & 4\,950 & 3 & 0.006 & 17\,863.25 & 0.2 & 10 & 5 & 297 & 3 & 0.006 \\
& & 6  & 30\,245.72 & 4.4 & 10 & 5 & 875 & 3 & 0.004 & 30\,245.72 & 0.3 & 10 & 5 & 266 & 3 & 0.004 \\
& & 8  & 37\,612.33 & 11.8 & 10 & 8 & 22\,131 & 4 & 0.009 & 37\,612.33 & 0.6 & 10 & 8 & 418 & 4 & 0.009 \\
& & 10  & 40\,612.33 & 12.0 & 10 & 8 & 18\,555 & 4 & 0.005 & 40\,612.33 & 0.6 & 10 & 8 & 418 & 4 & 0.005 \\
\hline
\multirow{5}{*}{15} & \multirow{5}{*}{(12,10)}
& 2  & 14\,000.00 & 0.1 & 15 & 5 & 0 & 3 & 0.008 & 14\,000.00 & 0.1 & 15 & 5 & 1 & 3 & 0.008 \\
& & 4  & 28\,362.90 & 1.2 & 15 & 5 & 2 & 3 & 0.007 & 28\,362.90 & 0.1 & 15 & 5 & 2 & 3 & 0.007 \\
& & 6  & 39\,606.76 & 1.4 & 15 & 6 & 2 & 3 & 0.000 & 39\,606.76 & 0.1 & 15 & 6 & 2 & 3 & 0.000 \\
& & 8  & 49\,379.91 & 4.6 & 15 & 6 & 3\,984 & 3 & 0.007 & 49\,379.91 & 0.2 & 15 & 6 & 231 & 3 & 0.007 \\
& & 10  & 61\,791.00 & 6.1 & 15 & 11 & 1\,578 & 4 & 0.003 & 61\,791.00 & 0.5 & 15 & 11 & 95 & 4 & 0.003 \\
\hline
\multirow{5}{*}{20} & \multirow{5}{*}{(12,10)}
& 2  & 29\,695.88 & 1.8 & 20 & 6 & 2 & 3 & 0.003 & 29\,695.88 & 0.1 & 20 & 6 & 2 & 3 & 0.003 \\
& & 4  & 45\,359.25 & 1.6 & 20 & 6 & 2 & 3 & 0.009 & 45\,359.25 & 0.1 & 20 & 6 & 2 & 3 & 0.009 \\
& & 6  & 54\,452.42 & 1.6 & 20 & 7 & 2 & 3 & 0.001 & 54\,452.42 & 0.1 & 20 & 7 & 2 & 3 & 0.001 \\
& & 8  & 65\,252.42 & 1.5 & 20 & 7 & 2 & 3 & 0.000 & 65\,252.42 & 0.1 & 20 & 7 & 2 & 3 & 0.000 \\
& & 10  & 78\,057.41 & 1.3 & 20 & 8 & 2 & 3 & 0.001 & 78\,057.41 & 0.2 & 20 & 8 & 2 & 3 & 0.001 \\
\hline
\multirow{5}{*}{25} & \multirow{5}{*}{(12,11)}
& 2  & 27\,012.54 & 3.1 & 25 & 7 & 2 & 3 & 0.000 & 27\,009.92 & 0.1 & 25 & 7 & 2 & 3 & 0.009 \\
& & 4  & 54\,577.74 & 1.8 & 25 & 8 & 2 & 3 & 0.010 & 54\,577.74 & 0.2 & 25 & 8 & 2 & 3 & 0.010 \\
& & 6  & 66\,477.74 & 1.7 & 25 & 9 & 2 & 3 & 0.010 & 66\,477.74 & 0.2 & 25 & 9 & 2 & 3 & 0.010 \\
& & 8  & 83\,277.74 & 1.6 & 25 & 10 & 2 & 3 & 0.009 & 83\,277.74 & 0.2 & 25 & 10 & 3 & 3 & 0.009 \\
& & 10  & 98\,674.11 & 7.6 & 25 & 10 & 1\,485 & 3 & 0.009 & 98\,674.11 & 0.5 & 25 & 10 & 240 & 3 & 0.009 \\
\hline
\multirow{5}{*}{30} & \multirow{5}{*}{(12,12)}
& 2  & 26\,666.18 & 1.8 & 30 & 13 & 2 & 3 & 0.010 & 26\,666.18 & 0.3 & 30 & 13 & 75 & 3 & 0.010 \\
& & 4  & 39\,102.75 & 16.2 & 30 & 16 & 9\,638 & 4 & 0.008 & 39\,102.74 & 1.2 & 30 & 16 & 1\,565 & 4 & 0.008 \\
& & 6  & 56\,712.73 & 36.5 & 30 & 16 & 48\,655 & 4 & 0.010 & 56\,712.73 & 2.0 & 30 & 16 & 3\,993 & 4 & 0.010 \\
& & 8  & 85\,373.22 & 41.2 & 30 & 17 & 17\,472 & 5 & 0.010 & 85\,373.21 & 2.6 & 30 & 17 & 6\,025 & 5 & 0.010 \\
& & 10  & 106\,524.68 & TL & 30 & 12 & 494\,654 & 3 & 0.937 & 107\,438.16 & 5.1 & 30 & 13 & 17\,560 & 4 & 0.079 \\
\hline
\multirow{5}{*}{132} & \multirow{5}{*}{(12,12)}
& 2  & 97\,020.71 & 1.9 & 132 & 14 & 2 & 3 & 0.009 & 97\,020.71 & 0.6 & 132 & 14 & 51 & 3 & 0.009 \\
& & 4  & 134\,696.98 & 2.1 & 132 & 14 & 2 & 3 & 0.006 & 134\,696.98 & 0.6 & 132 & 14 & 77 & 3 & 0.006 \\
& & 6  & 206\,690.32 & 2.1 & 132 & 14 & 2 & 3 & 0.000 & 206\,690.32 & 0.8 & 132 & 14 & 111 & 3 & 0.000 \\
& & 8  & 279\,290.32 & 2.2 & 132 & 15 & 2 & 3 & 0.002 & 279\,290.32 & 0.7 & 132 & 15 & 64 & 3 & 0.002 \\
& & 10  & 335\,336.96 & 1.2 & 132 & 14 & 1 & 2 & 0.000 & 335\,336.96 & 0.5 & 132 & 14 & 21 & 2 & 0.000 \\
\hline
\end{tabular}
\end{sidewaystable}

\begin{sidewaystable}
\caption{\label{tab:rq5_pathmax} Path-based formulations (SD-P) and (CC) initialized
with the paths of the equilibrium at the maximal tolls only ($\tau_a=\tau_a^{\tt up}$) (reference configuration, time limit TL${}=1800$\,s). Columns are as in Table~\ref{tab:rq4}. 
}
\tiny
\centering
\begin{tabular}{|c|c|c||c|c|c|c|c|c|c||c|c|c|c|c|c|c|}
\hline
\multirow{2}{*}{$\left|\mathcal{K}\right|$} &
\multirow{2}{*}{$(\#o,\#d)$} &
\multirow{2}{*}{$\left|A^{\tt{toll}}\right|$} &
\multicolumn{7}{|c||}{(SD-P)} &
\multicolumn{7}{|c|}{(CC)}
\\
\cline{4-17}
& & &
obj & time & $\#p_i$ & $\#p_g$ & nodes & it & gap &
obj & time & $\#p_i$ & $\#p_g$ & nodes & it & gap \\
\hline
\multirow{5}{*}{4} & \multirow{5}{*}{(4,3)}
& 2  & 11\,787.07 & 1.6 & 4 & 2 & 2 & 3 & 0.000 & 11\,787.08 & 0.1 & 4 & 2 & 2 & 3 & 0.000 \\
& & 4  & 16\,287.08 & 1.4 & 4 & 2 & 2 & 3 & 0.000 & 16\,287.08 & 0.1 & 4 & 2 & 2 & 3 & 0.000 \\
& & 6  & 18\,987.08 & 1.2 & 4 & 2 & 2 & 3 & 0.000 & 18\,987.08 & 0.1 & 4 & 2 & 2 & 3 & 0.000 \\
& & 8  & 23\,870.26 & 0.8 & 4 & 2 & 1 & 2 & 0.004 & 21\,879.96 & 0.0 & 4 & 1 & 1 & 2 & 9.100 \\
& & 10  & 26\,212.69 & 1.2 & 4 & 5 & 3 & 4 & 0.003 & 23\,611.99 & 0.1 & 4 & 2 & 1 & 2 & 11.018 \\
\hline
\multirow{5}{*}{5} & \multirow{5}{*}{(5,4)}
& 2  & 11\,787.07 & 1.6 & 5 & 2 & 2 & 3 & 0.000 & 11\,787.08 & 0.1 & 5 & 2 & 2 & 3 & 0.000 \\
& & 4  & 16\,287.08 & 1.5 & 5 & 2 & 2 & 3 & 0.000 & 16\,287.08 & 0.1 & 5 & 2 & 2 & 3 & 0.000 \\
& & 6  & 18\,987.08 & 1.3 & 5 & 2 & 2 & 3 & 0.000 & 18\,987.08 & 0.1 & 5 & 2 & 2 & 3 & 0.000 \\
& & 8  & 23\,870.26 & 0.8 & 5 & 2 & 1 & 2 & 0.001 & 21\,879.96 & 0.0 & 5 & 1 & 1 & 2 & 9.098 \\
& & 10  & 26\,603.79 & 1.3 & 5 & 5 & 3 & 4 & 0.000 & 23\,611.99 & 0.1 & 5 & 2 & 1 & 2 & 12.671 \\
\hline
\multirow{5}{*}{6} & \multirow{5}{*}{(6,5)}
& 2  & 11\,743.58 & 2.9 & 6 & 5 & 3 & 4 & 0.003 & 11\,743.58 & 0.2 & 6 & 5 & 12 & 4 & 0.003 \\
& & 4  & 19\,199.33 & 10.0 & 6 & 6 & 8\,806 & 5 & 0.005 & 19\,199.33 & 0.3 & 6 & 6 & 128 & 5 & 0.005 \\
& & 6  & 22\,799.33 & 9.9 & 6 & 6 & 8\,737 & 5 & 0.005 & 22\,799.33 & 0.3 & 6 & 6 & 128 & 5 & 0.005 \\
& & 8  & 30\,343.01 & 192.3 & 6 & 6 & 844\,782 & 4 & 0.006 & 30\,343.01 & 0.6 & 6 & 6 & 1\,387 & 4 & 0.006 \\
& & 10  & 39\,619.53 & 8.6 & 6 & 5 & 30\,291 & 4 & 0.009 & 39\,619.54 & 0.4 & 6 & 5 & 316 & 5 & 0.009 \\
\hline
\multirow{5}{*}{7} & \multirow{5}{*}{(7,6)}
& 2  & 8\,593.37 & 1.5 & 7 & 4 & 2 & 3 & 0.004 & 8\,593.37 & 0.2 & 7 & 4 & 31 & 3 & 0.004 \\
& & 4  & 17\,955.10 & 3.2 & 7 & 5 & 4 & 5 & 8.483 & 17\,955.10 & 0.2 & 7 & 5 & 14 & 5 & 8.483 \\
& & 6  & 20\,655.10 & 3.6 & 7 & 5 & 4 & 5 & 7.378 & 20\,655.10 & 0.2 & 7 & 5 & 10 & 5 & 7.378 \\
& & 8  & 27\,546.70 & 9.8 & 7 & 4 & 9\,949 & 4 & 5.038 & 27\,546.70 & 0.2 & 7 & 4 & 317 & 4 & 5.038 \\
& & 10  & 29\,626.13 & 2.3 & 7 & 3 & 2 & 3 & 4.665 & 29\,626.13 & 0.2 & 7 & 3 & 106 & 3 & 4.665 \\
\hline
\multirow{5}{*}{8} & \multirow{5}{*}{(8,7)}
& 2  & 8\,612.07 & 1.7 & 8 & 5 & 2 & 3 & 0.005 & 8\,612.07 & 0.2 & 8 & 5 & 24 & 3 & 0.005 \\
& & 4  & 16\,644.19 & 0.7 & 8 & 4 & 1 & 2 & 11.404 & 16\,644.19 & 0.1 & 8 & 4 & 1 & 2 & 11.404 \\
& & 6  & 24\,245.72 & 14.2 & 8 & 6 & 19\,300 & 4 & 9.344 & 24\,245.72 & 0.3 & 8 & 6 & 664 & 4 & 9.344 \\
& & 8  & 27\,245.72 & 13.8 & 8 & 6 & 18\,508 & 4 & 8.317 & 27\,245.72 & 0.3 & 8 & 6 & 664 & 4 & 8.317 \\
& & 10  & 37\,427.79 & 3.7 & 8 & 6 & 594 & 4 & 4.649 & 37\,427.79 & 0.2 & 8 & 5 & 97 & 4 & 4.649 \\
\hline
\multirow{5}{*}{10} & \multirow{5}{*}{(10,9)}
& 2  & 11\,055.09 & 1.5 & 10 & 4 & 2 & 3 & 0.003 & 11\,055.09 & 0.1 & 10 & 4 & 20 & 3 & 0.003 \\
& & 4  & 17\,863.25 & 6.6 & 10 & 4 & 5\,363 & 3 & 0.006 & 17\,863.25 & 0.3 & 10 & 4 & 2\,639 & 3 & 0.006 \\
& & 6  & 30\,245.72 & 2.2 & 10 & 6 & 2 & 3 & 0.004 & 30\,245.72 & 0.2 & 10 & 6 & 86 & 3 & 0.004 \\
& & 8  & 37\,612.33 & 21.7 & 10 & 8 & 53\,931 & 5 & 0.009 & 37\,612.33 & 0.9 & 10 & 8 & 615 & 5 & 0.009 \\
& & 10  & 40\,612.33 & 18.4 & 10 & 8 & 42\,533 & 5 & 0.005 & 40\,612.33 & 0.7 & 10 & 8 & 576 & 5 & 0.005 \\
\hline
\multirow{5}{*}{15} & \multirow{5}{*}{(12,10)}
& 2  & 14\,000.00 & 0.1 & 15 & 5 & 0 & 3 & 0.008 & 14\,000.00 & 0.1 & 15 & 5 & 1 & 3 & 0.008 \\
& & 4  & 28\,362.90 & 1.2 & 15 & 5 & 2 & 3 & 0.007 & 28\,362.90 & 0.1 & 15 & 5 & 2 & 3 & 0.007 \\
& & 6  & 39\,606.76 & 1.6 & 15 & 5 & 2 & 3 & 0.000 & 39\,606.76 & 0.1 & 15 & 5 & 2 & 3 & 0.000 \\
& & 8  & 49\,379.91 & 4.1 & 15 & 5 & 3\,243 & 3 & 0.007 & 49\,379.91 & 0.2 & 15 & 5 & 81 & 3 & 0.007 \\
& & 10  & 61\,791.00 & 4.8 & 15 & 7 & 601 & 4 & 0.003 & 61\,791.00 & 0.3 & 15 & 7 & 88 & 4 & 0.003 \\
\hline
\multirow{5}{*}{20} & \multirow{5}{*}{(12,10)}
& 2  & 29\,695.88 & 1.8 & 20 & 6 & 2 & 3 & 0.003 & 29\,695.88 & 0.1 & 20 & 6 & 2 & 3 & 0.003 \\
& & 4  & 45\,359.25 & 1.6 & 20 & 6 & 2 & 3 & 0.009 & 45\,359.25 & 0.1 & 20 & 6 & 2 & 3 & 0.009 \\
& & 6  & 54\,452.42 & 1.6 & 20 & 7 & 2 & 3 & 0.001 & 54\,452.42 & 0.1 & 20 & 7 & 2 & 3 & 0.001 \\
& & 8  & 65\,252.42 & 1.5 & 20 & 7 & 2 & 3 & 0.000 & 65\,252.42 & 0.1 & 20 & 7 & 2 & 3 & 0.000 \\
& & 10  & 78\,057.41 & 1.8 & 20 & 6 & 2 & 3 & 0.001 & 78\,057.41 & 0.1 & 20 & 6 & 2 & 3 & 0.001 \\
\hline
\multirow{5}{*}{25} & \multirow{5}{*}{(12,11)}
& 2  & 27\,009.84 & 1.5 & 25 & 7 & 2 & 3 & 0.009 & 27\,009.92 & 0.1 & 25 & 7 & 2 & 3 & 0.009 \\
& & 4  & 54\,577.74 & 1.4 & 25 & 7 & 2 & 3 & 0.010 & 54\,577.74 & 0.1 & 25 & 7 & 2 & 3 & 0.010 \\
& & 6  & 66\,477.74 & 1.5 & 25 & 7 & 2 & 3 & 0.010 & 66\,477.74 & 0.1 & 25 & 7 & 2 & 3 & 0.010 \\
& & 8  & 83\,277.74 & 1.5 & 25 & 8 & 2 & 3 & 0.009 & 83\,277.74 & 0.1 & 25 & 8 & 2 & 3 & 0.009 \\
& & 10  & 98\,674.11 & 1.8 & 25 & 6 & 781 & 2 & 0.009 & 98\,674.11 & 0.2 & 25 & 6 & 103 & 2 & 0.009 \\
\hline
\multirow{5}{*}{30} & \multirow{5}{*}{(12,12)}
& 2  & 26\,666.18 & 3.4 & 30 & 13 & 7 & 4 & 0.010 & 26\,666.18 & 0.5 & 30 & 13 & 50 & 4 & 0.010 \\
& & 4  & 39\,102.74 & 25.2 & 30 & 9 & 34\,886 & 4 & 0.008 & 39\,102.74 & 0.7 & 30 & 9 & 431 & 4 & 0.008 \\
& & 6  & 56\,712.73 & 44.3 & 30 & 10 & 80\,172 & 4 & 0.010 & 56\,712.73 & 0.8 & 30 & 10 & 430 & 4 & 0.010 \\
& & 8  & 85\,373.21 & 23.3 & 30 & 11 & 22\,418 & 3 & 0.010 & 85\,373.21 & 0.6 & 30 & 11 & 669 & 3 & 0.010 \\
& & 10  & 107\,438.16 & 44.7 & 30 & 11 & 126\,906 & 3 & 0.079 & 107\,438.16 & 2.0 & 30 & 11 & 6\,596 & 3 & 0.079 \\
\hline
\multirow{5}{*}{132} & \multirow{5}{*}{(12,12)}
& 2  & 97\,020.71 & 0.7 & 132 & 8 & 1 & 2 & 0.009 & 97\,020.71 & 0.4 & 132 & 8 & 3 & 2 & 0.009 \\
& & 4  & 134\,696.98 & 0.8 & 132 & 8 & 1 & 2 & 0.006 & 134\,696.98 & 0.3 & 132 & 8 & 1 & 2 & 0.006 \\
& & 6  & 206\,690.32 & 1.7 & 132 & 10 & 2 & 3 & 0.000 & 206\,690.32 & 0.6 & 132 & 10 & 23 & 3 & 0.000 \\
& & 8  & 279\,290.32 & 1.7 & 132 & 10 & 2 & 3 & 0.002 & 279\,290.32 & 0.6 & 132 & 10 & 23 & 3 & 0.002 \\
& & 10  & 335\,336.96 & 1.8 & 132 & 10 & 2 & 3 & 0.000 & 335\,336.96 & 0.6 & 132 & 10 & 22 & 3 & 0.000 \\
\hline
\end{tabular}
\end{sidewaystable}

\section{Experiments on the complete Sioux Falls network}
\label{sec:app_bignet}
The experiments of Sections~\ref{sec:exp_exact}--\ref{sec:exp_heuristic} are conducted
on the twelve-node, thirty-six-arc subnetwork of the Sioux Falls network described in
Appendix~\ref{app:instance_Sioux}. 
To assess how the formulations scale with the network, we compare them here, in the
reference configuration of Section~\ref{sec:exp_setup} only, on the \emph{complete}
Sioux Falls network, which has $24$ nodes and $76$ arcs (see Figure~\ref{fig:siouxfalls}).

The instances are generated as on the subnetwork --- the same
Wardrop-equilibrium-based selection of the toll arcs (Appendix~\ref{app:toll_arcs}), the
same reference cost structure, and the same toll upper bound
$\tau_a^{\tt up}=\min(c_a^{\tt fft},\tau_a^\infty)$ --- and follow the same
$55$-instance design: eleven demand levels, ranging from $\left|\mathcal{K}\right|=4$ up
to the full origin--destination demand $\left|\mathcal{K}\right|=528$, each combined with
$|A^{\tt toll}|\in\{2,4,6,8,10\}$ toll arcs. For the smaller demand levels ($k\le 30$),
the first $k$ commodities are taken in the order given in
Table~\ref{tab:demands_eval_SiouxFalls_full}. Table~\ref{tab:rq_bignet} reports the
per-instance results of the exact formulations, with the same fields and the same pooled
instance-level bound $\overline{z}$ as Table~\ref{tab:rq1}.

\begin{table}[h!]
\centering

\caption{OD pairs considered in our instances based on the complete 
Sioux Falls network, with the demand values and identifiers. The ordering of
the demands (column \#) determines, for each $k\le 30$, the first $k$
commodities used in the corresponding instances.}
\label{tab:demands_eval_SiouxFalls_full}
 \begin{tabular}{|c|c|c||c|c|c||c|c|c|}
 \hline
\#& $(o_k,s_k)$& demand &\# & $(o_k,s_k)$& demand &\# & $(o_k,s_k)$& demand  \\
 \hline\hline
1&(1,10)&1300.0 &11&(12,7)&700.0&21&(9,20)&600.0\\
2&(17,22)&1700.0 &12&(13,2)&300.0 &22&(23,17)&600.0\\
3&(20,10)&2500.0 &13&(14,8)&400.0 &23&(10,13)&1900.0\\
4&(11,22)&1100.0 &14&(18,23)&100.0 &24&(15,1)&500.0\\
5&(2,11)&200.0 &15&(21,10)&1200.0 &25&(22,8)&500.0\\
6&(16,4)&800.0&16&(5,22)&200.0 &26&(2,16)&400.0\\
7&(8,1)&800.0&17&(24,16)&300.0&27&(4,23)&500.0\\
8&(3,16)&200.0&18&(7,22)&500.0&28 &(19,7)&400.0\\
9&(6,3)&300.0&19&(6,20)&300.0&29&(10,2)&600.0\\
10&(1,20)& 300.0&20&(8,13)&600.0&30&(13,9)&600.0\\
 \hline
 \end{tabular}
\end{table}

\paragraph{Scalability.}
The larger network sharpens the contrast already visible on the subnetwork. Within
the time limit, (BLS) solves only $7$ of the $55$ instances --- all with
$\left|\mathcal{K}\right|\le 6$ --- and, more tellingly, returns no feasible solution at all on
$31$ of them: on every instance with $\left|\mathcal{K}\right|\ge 10$, and on one further
instance with $\left|\mathcal{K}\right|=7$. Its shifted geometric mean solution time, $1306$
seconds, is close to the time limit itself. In contrast, (SD-A) solves $35$ of the
$55$ instances, returns a feasible solution on every instance, and is more than
five times faster in shifted geometric mean ($233$ seconds). On the six instances
solved by both formulations, (SD-A) is faster on five, with a median speed-up of
$21\times$. A single instance --- $\left|\mathcal{K}\right|=4$ with $6$ toll arcs --- is solved
by (BLS) but not by (SD-A), confirming that neither formulation dominates the other; 
but the overall gap in robustness is wide, (SD-A) solving five times as
many instances. 

\paragraph{Difficulty and network size.}
As on the subnetwork, the difficulty of an instance is not governed by the number of
commodities alone. (SD-A) solves all five full-demand instances ($\left|\mathcal{K}\right|=528$),
in between $149$ and $1355$ seconds, from an initial path set of about $540$ paths; yet
it reaches the time limit on three of the five instances for each of
$\left|\mathcal{K}\right|\in\{20,25,30\}$, and even on two of the five four-commodity instances
($\left|\mathcal{K}\right|=4$ with $6$ and $10$ toll arcs).

\paragraph{Gaps at the time limit.}
The certified gaps show that, here too, the failures of both formulations are
primarily of certification rather than of primal search. 
On the ten instances that (SD-A) solves and on which (BLS) returns a feasible solution
without proving its optimality, that solution is within $0.79\%$ of the optimum, and
$0.093\%$ on average. 
On the nineteen instances that
neither formulation solves, the pooled certified gap of (SD-A) remains moderate --- a
median of $2.2\%$ and a maximum of $13.2\%$ --- so that even where optimality is out of
reach, the strong-duality formulation delivers solutions of controlled quality.
Overall, the complete network confirms the conclusion drawn on the subnetwork:
(SD-A) is the more robust and more reliable of the two exact formulations, and the 
margin in its favour widens as the network grows.
\begin{sidewaystable}
\caption{\label{tab:rq_bignet} Comparison of (BLS) and (SD-A) on the complete Sioux Falls
network ($24$ nodes, $76$ arcs; time limit TL${}=1800$\,s). The first $30$
origin--destination pairs are listed in Table~\ref{tab:demands_eval_SiouxFalls_full} and
the toll arcs are selected as described in Appendix~\ref{app:toll_arcs}. Columns are as
in Table~\ref{tab:rq1}. A dash (--) indicates that no feasible solution was found.}
\tiny
\centering
\begin{tabular}{|c|c|c||c|c|c|c|c||c|c|c|c|c|c|c|c|}
\hline
\multirow{2}{*}{$\left|\mathcal{K}\right|$} &
\multirow{2}{*}{$(\#o,\#d)$} &
\multirow{2}{*}{$\left|A^{\tt{toll}}\right|$} &
\multicolumn{5}{|c||}{${\tt{(BLS)}}$} &
\multicolumn{8}{|c|}{${\tt{(SD\text{-}A)}}$}
\\
\cline{4-16}
& & & objective & time & nodes & gap$_{\tt own}$ (\%) & gap (\%)
& objective & time & $\#p_i$ & $\#p_g$ & nodes & it & gap$_{\tt own}$ (\%) & gap (\%)\\
\hline
\multirow{5}{*}{4} & \multirow{5}{*}{(4,2)}
& 2  & 7\,829.62 & 47.3 & 21\,315 & 0.009 & 0.008 & 7\,829.62 & 4.8 & 5 & 3 & 3 & 3 & 0.008 & 0.008 \\
& & 4  & 12\,030.89 & 58.1 & 29\,131 & 0.010 & 0.010 & 12\,030.89 & 221.3 & 6 & 3 & 256\,227 & 3 & 0.010 & 0.010 \\
& & 6  & 17\,428.60 & 92.8 & 48\,161 & 0.010 & 0.010 & 17\,425.21 & TL & 6 & 3 & 1\,649\,936 & 2 & 0.080 & 0.029 \\
& & 8  & 26\,046.03 & 109.7 & 52\,749 & 0.010 & 0.000 & 26\,046.03 & 2.5 & 6 & 4 & 3 & 3 & 0.000 & 0.000 \\
& & 10  & 36\,277.93 & TL & 1\,829\,351 & 0.958 & 0.958 & 35\,420.36 & TL & 6 & 2 & 1\,350\,462 & 2 & 3.775 & 3.402 \\
\hline
\multirow{5}{*}{5} & \multirow{5}{*}{(5,3)}
& 2  & 7\,829.62 & 116.2 & 21\,370 & 0.003 & 0.003 & 7\,829.62 & 6.0 & 8 & 2 & 3 & 3 & 0.005 & 0.003 \\
& & 4  & 10\,600.37 & 235.8 & 131\,248 & 0.010 & 0.001 & 10\,600.37 & 10.8 & 8 & 3 & 1\,325 & 3 & 0.001 & 0.001 \\
& & 6  & 13\,630.89 & TL & 905\,302 & 18.834 & 0.056 & 13\,628.84 & TL & 7 & 2 & 1\,149\,625 & 3 & 0.071 & 0.071 \\
& & 8  & 16\,362.67 & TL & 656\,199 & 7.469 & 0.208 & 16\,348.47 & TL & 7 & 2 & 1\,016\,000 & 2 & 0.295 & 0.295 \\
& & 10  & 21\,124.65 & TL & 775\,204 & 17.873 & 0.000 & 21\,124.64 & 0.7 & 7 & 4 & 3 & 3 & 0.000 & 0.000 \\
\hline
\multirow{5}{*}{6} & \multirow{5}{*}{(6,4)}
& 2  & 5\,149.26 & 1013.6 & 398\,784 & 0.008 & 0.007 & 5\,149.21 & 6.6 & 6 & 7 & 3 & 3 & 0.009 & 0.009 \\
& & 4  & 12\,978.89 & TL & 182\,912 & 20.713 & 0.008 & 12\,978.89 & 42.0 & 7 & 5 & 7\,931 & 3 & 0.008 & 0.008 \\
& & 6  & 15\,746.97 & TL & 211\,442 & 37.932 & 0.024 & 15\,749.63 & 36.3 & 7 & 7 & 1\,572 & 3 & 0.007 & 0.007 \\
& & 8  & 19\,254.81 & TL & 197\,267 & 39.006 & 0.037 & 19\,254.63 & TL & 8 & 7 & 234\,888 & 3 & 0.038 & 0.038 \\
& & 10  & 24\,519.08 & TL & 146\,078 & 57.526 & 0.229 & 23\,425.50 & TL & 8 & 4 & 920\,978 & 2 & 4.908 & 4.908 \\
\hline
\multirow{5}{*}{7} & \multirow{5}{*}{(7,5)}
& 2  & 4\,697.87 & TL & 81\,099 & 68.155 & 0.063 & 4\,700.39 & 14.5 & 7 & 8 & 69 & 4 & 0.010 & 0.010 \\
& & 4  & 12\,520.61 & TL & 94\,352 & 60.037 & 0.009 & 12\,520.61 & 11.3 & 7 & 6 & 3 & 3 & 0.009 & 0.009 \\
& & 6  & -- & TL & 101\,951 & -- & -- & 19\,950.23 & 25.0 & 9 & 5 & 1\,477 & 3 & 0.005 & 0.005 \\
& & 8  & 22\,720.98 & TL & 83\,668 & 88.231 & 0.009 & 22\,720.98 & 32.9 & 10 & 6 & 2\,008 & 3 & 0.009 & 0.009 \\
& & 10  & 28\,607.54 & TL & 83\,448 & 119.063 & 0.169 & 26\,967.56 & TL & 10 & 6 & 438\,682 & 3 & 6.261 & 6.261 \\
\hline
\multirow{5}{*}{8} & \multirow{5}{*}{(8,6)}
& 2  & 4\,304.97 & TL & 72\,030 & 187.962 & 0.009 & 4\,304.97 & 22.3 & 9 & 4 & 2\,875 & 3 & 0.009 & 0.009 \\
& & 4  & 8\,935.64 & TL & 58\,232 & 124.635 & 0.789 & 9\,005.36 & 156.3 & 10 & 6 & 34\,039 & 4 & 0.008 & 0.008 \\
& & 6  & 16\,825.58 & TL & 78\,669 & 80.544 & 0.010 & 16\,825.58 & 69.7 & 9 & 6 & 9\,955 & 3 & 0.010 & 0.010 \\
& & 8  & 23\,437.24 & TL & 59\,213 & 111.189 & 0.009 & 23\,437.24 & 58.7 & 10 & 7 & 8\,857 & 3 & 0.009 & 0.009 \\
& & 10  & 26\,497.87 & TL & 64\,901 & 150.867 & 0.378 & 24\,058.95 & TL & 10 & 5 & 472\,869 & 2 & 10.554 & 10.554 \\
\hline
\multirow{5}{*}{10} & \multirow{5}{*}{(9,8)}
& 2  & -- & TL & 20\,101 & -- & -- & 6\,090.25 & 14.6 & 11 & 7 & 4 & 4 & 0.001 & 0.001 \\
& & 4  & -- & TL & 24\,739 & -- & -- & 11\,266.46 & 37.9 & 12 & 8 & 1\,069 & 4 & 0.001 & 0.001 \\
& & 6  & -- & TL & 20\,994 & -- & -- & 12\,956.59 & 251.3 & 13 & 10 & 101\,148 & 4 & 0.010 & 0.010 \\
& & 8  & -- & TL & 23\,375 & -- & -- & 16\,360.56 & 162.6 & 15 & 3 & 51\,143 & 3 & 0.010 & 0.010 \\
& & 10  & -- & TL & 22\,710 & -- & -- & 19\,100.33 & TL & 14 & 4 & 605\,719 & 2 & 1.284 & 1.284 \\
\hline
\multirow{5}{*}{15} & \multirow{5}{*}{(14,12)}
& 2  & -- & TL & 5\,029 & -- & -- & 5\,969.11 & 55.7 & 16 & 8 & 1\,181 & 3 & 0.008 & 0.008 \\
& & 4  & -- & TL & 5\,553 & -- & -- & 9\,724.27 & 89.9 & 18 & 7 & 11\,057 & 3 & 0.006 & 0.006 \\
& & 6  & -- & TL & 8\,298 & -- & -- & 16\,765.66 & 274.1 & 20 & 6 & 145\,799 & 3 & 0.008 & 0.008 \\
& & 8  & -- & TL & 11\,292 & -- & -- & 20\,197.03 & TL & 20 & 4 & 395\,161 & 2 & 1.268 & 1.268 \\
& & 10  & -- & TL & 10\,895 & -- & -- & 24\,844.92 & TL & 18 & 7 & 383\,173 & 3 & 0.031 & 0.031 \\
\hline
\multirow{5}{*}{20} & \multirow{5}{*}{(17,13)}
& 2  & -- & TL & 4\,972 & -- & -- & 7\,738.08 & 74.4 & 22 & 10 & 2\,731 & 4 & 0.008 & 0.008 \\
& & 4  & -- & TL & 5\,152 & -- & -- & 13\,774.78 & TL & 24 & 10 & 317\,773 & 3 & 0.722 & 0.722 \\
& & 6  & -- & TL & 6\,026 & -- & -- & 17\,846.48 & TL & 24 & 5 & 53\,778 & 2 & 3.818 & 3.818 \\
& & 8  & -- & TL & 5\,174 & -- & -- & 21\,972.38 & TL & 27 & 4 & 104\,407 & 2 & 1.072 & 1.072 \\
& & 10  & -- & TL & 5\,103 & -- & -- & 30\,952.97 & 618.5 & 25 & 8 & 116\,695 & 3 & 0.010 & 0.010 \\
\hline
\multirow{5}{*}{25} & \multirow{5}{*}{(22,14)}
& 2  & -- & TL & 5\,244 & -- & -- & 9\,466.27 & 55.6 & 26 & 12 & 608 & 3 & 0.009 & 0.009 \\
& & 4  & -- & TL & 4\,650 & -- & -- & 16\,237.35 & 464.2 & 28 & 11 & 54\,036 & 3 & 0.010 & 0.010 \\
& & 6  & -- & TL & 3\,276 & -- & -- & 24\,335.52 & TL & 32 & 6 & 77\,358 & 2 & 0.984 & 0.984 \\
& & 8  & -- & TL & 4\,402 & -- & -- & 28\,328.24 & TL & 33 & 6 & 434\,804 & 2 & 3.232 & 3.232 \\
& & 10  & -- & TL & 4\,960 & -- & -- & 33\,366.85 & TL & 35 & 5 & 132\,914 & 2 & 3.508 & 3.508 \\
\hline
\multirow{5}{*}{30} & \multirow{5}{*}{(24,15)}
& 2  & -- & TL & 2\,643 & -- & -- & 7\,861.46 & 29.0 & 33 & 8 & 374 & 2 & 0.000 & 0.000 \\
& & 4  & -- & TL & 3\,910 & -- & -- & 16\,830.74 & 77.1 & 35 & 9 & 5\,485 & 2 & 0.006 & 0.006 \\
& & 6  & -- & TL & 5\,105 & -- & -- & 25\,325.45 & TL & 33 & 9 & 101\,653 & 2 & 2.208 & 2.208 \\
& & 8  & -- & TL & 4\,177 & -- & -- & 29\,621.06 & TL & 39 & 8 & 67\,672 & 2 & 10.849 & 10.849 \\
& & 10  & -- & TL & 3\,566 & -- & -- & 36\,010.16 & TL & 41 & 0 & 120\,618 & 1 & 13.177 & 13.177 \\
\hline
\multirow{5}{*}{528} & \multirow{5}{*}{(24,24)}
& 2  & -- & TL & 585 & -- & -- & 69\,977.71 & 908.4 & 533 & 97 & 1\,417 & 3 & 0.001 & 0.001 \\
& & 4  & -- & TL & 66 & -- & -- & 178\,591.69 & 209.6 & 537 & 101 & 2\,828 & 3 & 0.000 & 0.000 \\
& & 6  & -- & TL & 2\,279 & -- & -- & 331\,152.45 & 148.9 & 539 & 98 & 652 & 3 & 0.002 & 0.002 \\
& & 8  & -- & TL & 1\,870 & -- & -- & 395\,240.08 & 1355.0 & 536 & 102 & 145\,094 & 4 & 0.006 & 0.006 \\
& & 10  & -- & TL & 1\,171 & -- & -- & 460\,184.66 & 300.5 & 540 & 103 & 1\,676 & 4 & 0.005 & 0.005 \\
\hline
\end{tabular}
\end{sidewaystable}

\paragraph{Path-based heuristics.}
We also evaluate the two path-based heuristics on the complete network;
Table~\ref{tab:rq_bignet_heur} reports the per-instance results, with the certified
gap computed against the same pooled bound $\overline{z}$ as in
Table~\ref{tab:rq_bignet}. Where the optimum is known --- the $36$ instances solved
by an exact formulation --- both heuristics remain accurate: (SD-P) attains it on
$32$ of them and (CC) on $34$, and the largest shortfall of a returned solution is
$0.175\%$. Their value is clearest, however, where the exact formulations fail. 
On the $31$ instances on which (BLS) returns no feasible solution, both heuristics return
one; and on the $20$ instances on which (SD-A) reaches the time limit, (CC) returns,
in at most $548$ seconds, a solution strictly better than the (SD-A) incumbent on $18$
of them, by up to $10\%$. Across the benchmark the certified gap is at most $0.01\%$
on $32$ of the $55$ instances for (SD-P) and $33$ for (CC); the larger gaps fall on the
intermediate demand levels, which neither the exact formulations nor the heuristics
close.

\paragraph{The two heuristics no longer coincide.}
On the subnetwork (Section~\ref{sec:exp_heuristic}), (SD-P) and (CC) returned
essentially identical solutions and (CC) dominated in speed almost uniformly. On the
complete network the two heuristics part ways. Their values now differ by up to
$8.4\%$ (on $\left|\mathcal{K}\right|=30$ with $8$ toll arcs), and, where they differ, (CC) is
generally the stronger: it is faster than (SD-P) on $46$ of the $55$ instances, by a
median factor of $10$; it reaches the time limit on only two instances, against ten
for (SD-P); and on the hard intermediate-demand instances
($\left|\mathcal{K}\right|\in\{20,25,30\}$) it returns the better feasible solution --- for
instance a certified gap of $10.7\%$ against $14.7\%$ on $\left|\mathcal{K}\right|=30$ with $10$
toll arcs, on which (SD-P) reaches the time limit. The one axis on which (SD-P) retains an edge
is full demand: (CC)'s two time-limited instances are both at $\left|\mathcal{K}\right|=528$,
where (SD-P)'s generation terminates on all five. Even there, however, the advantage
is one of termination rather than of quality --- on $\left|\mathcal{K}\right|=528$ with $4$ toll
arcs (CC) matches the exact optimum while (SD-P), though it terminates, misses it by
$0.18\%$. The SOS1 branching that makes (CC) so effective thus scales less readily to
the $528$ commodities of the full-demand case, but without ever degrading the
solutions it returns. Overall, (CC) remains the method of choice across the benchmark,
with (SD-P) a useful complement only at full demand, where it is the more likely of
the two to terminate within the time budget.
\begin{sidewaystable}
\caption{\label{tab:rq_bignet_heur} Path-based heuristics (SD-P) and (CC) on the
complete Sioux Falls network ($24$ nodes, $76$ arcs; time limit TL${}=1800$\,s). The
``gap'' column is the certified gap against the pooled bound $\overline{z}$ of
Table~\ref{tab:rq_bignet}; all other columns are as in Table~\ref{tab:rq4}. 
}
\tiny
\centering
\begin{tabular}{|c|c|c||c|c|c|c|c|c|c||c|c|c|c|c|c|c|}
\hline
\multirow{2}{*}{$\left|\mathcal{K}\right|$} &
\multirow{2}{*}{$(\#o,\#d)$} &
\multirow{2}{*}{$\left|A^{\tt{toll}}\right|$} &
\multicolumn{7}{|c||}{(SD-P)} &
\multicolumn{7}{|c|}{(CC)}
\\
\cline{4-17}
& & &
obj & time & $\#p_i$ & $\#p_g$ & nodes & it & gap &
obj & time & $\#p_i$ & $\#p_g$ & nodes & it & gap \\
\hline
\multirow{5}{*}{4} & \multirow{5}{*}{(4,2)}
& 2  & 7\,829.62 & 1.2 & 5 & 3 & 3 & 3 & 0.008 & 7\,829.62 & 0.1 & 5 & 3 & 3 & 3 & 0.008 \\
& & 4  & 12\,030.89 & 1.7 & 6 & 2 & 3 & 3 & 0.010 & 12\,030.89 & 0.8 & 6 & 2 & 3\,669 & 3 & 0.010 \\
& & 6  & 17\,428.60 & 4.2 & 6 & 3 & 857 & 3 & 0.010 & 17\,428.60 & 1.7 & 6 & 4 & 15\,635 & 3 & 0.010 \\
& & 8  & 26\,046.03 & 1.1 & 6 & 3 & 3 & 3 & 0.000 & 26\,046.03 & 4.5 & 6 & 4 & 43\,357 & 3 & 0.000 \\
& & 10  & 36\,277.75 & 1.9 & 6 & 5 & 4 & 4 & 0.959 & 36\,277.93 & 6.7 & 6 & 5 & 65\,463 & 4 & 0.958 \\
\hline
\multirow{5}{*}{5} & \multirow{5}{*}{(5,3)}
& 2  & 7\,829.58 & 2.2 & 8 & 2 & 3 & 3 & 0.004 & 7\,829.62 & 0.2 & 8 & 2 & 15 & 3 & 0.003 \\
& & 4  & 10\,600.37 & 3.5 & 8 & 3 & 1\,505 & 3 & 0.001 & 10\,600.37 & 0.4 & 8 & 3 & 1\,751 & 3 & 0.001 \\
& & 6  & 13\,630.89 & 2.0 & 7 & 2 & 3 & 3 & 0.056 & 13\,630.89 & 0.8 & 7 & 2 & 3\,663 & 3 & 0.056 \\
& & 8  & 16\,362.66 & 1.9 & 7 & 3 & 3 & 3 & 0.208 & 16\,362.66 & 1.0 & 7 & 3 & 5\,263 & 3 & 0.208 \\
& & 10  & 21\,124.64 & 2.4 & 7 & 4 & 3 & 3 & 0.000 & 21\,124.64 & 3.2 & 7 & 4 & 20\,810 & 3 & 0.000 \\
\hline
\multirow{5}{*}{6} & \multirow{5}{*}{(6,4)}
& 2  & 5\,149.26 & 1.3 & 6 & 7 & 2 & 3 & 0.007 & 5\,149.26 & 0.2 & 6 & 7 & 22 & 3 & 0.007 \\
& & 4  & 12\,978.89 & 3.2 & 7 & 5 & 1\,162 & 3 & 0.008 & 12\,978.89 & 0.2 & 7 & 5 & 241 & 3 & 0.008 \\
& & 6  & 15\,749.63 & 7.6 & 7 & 7 & 3\,659 & 3 & 0.007 & 15\,749.63 & 0.8 & 7 & 7 & 3\,350 & 3 & 0.007 \\
& & 8  & 19\,122.04 & 2.7 & 8 & 6 & 4 & 4 & 0.731 & 19\,122.04 & 3.8 & 8 & 5 & 28\,100 & 3 & 0.731 \\
& & 10  & 24\,519.08 & 9.6 & 8 & 6 & 850 & 3 & 0.229 & 24\,519.08 & 3.7 & 8 & 6 & 18\,046 & 3 & 0.229 \\
\hline
\multirow{5}{*}{7} & \multirow{5}{*}{(7,5)}
& 2  & 4\,700.39 & 2.2 & 7 & 8 & 3 & 4 & 0.010 & 4\,700.39 & 0.3 & 7 & 8 & 1\,057 & 4 & 0.010 \\
& & 4  & 12\,520.61 & 1.6 & 7 & 6 & 2 & 3 & 0.009 & 12\,520.61 & 0.1 & 7 & 6 & 34 & 3 & 0.009 \\
& & 6  & 19\,950.23 & 3.0 & 9 & 5 & 46 & 3 & 0.005 & 19\,950.23 & 0.4 & 9 & 5 & 389 & 3 & 0.005 \\
& & 8  & 22\,720.98 & 16.2 & 10 & 5 & 4\,193 & 3 & 0.009 & 22\,720.98 & 0.6 & 10 & 5 & 1\,048 & 3 & 0.009 \\
& & 10  & 28\,623.49 & TL & 10 & 7 & 1\,005\,375 & 4 & 0.113 & 28\,623.68 & 3.2 & 10 & 7 & 17\,121 & 4 & 0.113 \\
\hline
\multirow{5}{*}{8} & \multirow{5}{*}{(8,6)}
& 2  & 4\,302.51 & 4.3 & 9 & 3 & 4\,497 & 3 & 0.066 & 4\,304.97 & 1.0 & 9 & 4 & 8\,463 & 3 & 0.009 \\
& & 4  & 9\,002.91 & 39.1 & 10 & 5 & 152\,421 & 4 & 0.036 & 9\,002.91 & 3.7 & 10 & 5 & 22\,591 & 4 & 0.036 \\
& & 6  & 16\,823.12 & 9.6 & 9 & 4 & 8\,011 & 3 & 0.024 & 16\,823.12 & 1.8 & 9 & 5 & 8\,311 & 3 & 0.024 \\
& & 8  & 23\,437.24 & 16.3 & 10 & 5 & 8\,222 & 3 & 0.009 & 23\,437.24 & 1.2 & 10 & 6 & 6\,338 & 3 & 0.009 \\
& & 10  & 26\,506.14 & 52.1 & 10 & 7 & 104\,885 & 3 & 0.347 & 26\,506.14 & 41.5 & 10 & 6 & 431\,407 & 3 & 0.347 \\
\hline
\multirow{5}{*}{10} & \multirow{5}{*}{(9,8)}
& 2  & 6\,090.25 & 3.1 & 11 & 7 & 4 & 4 & 0.001 & 6\,090.25 & 0.3 & 11 & 7 & 37 & 4 & 0.001 \\
& & 4  & 11\,266.46 & 6.8 & 12 & 8 & 1\,196 & 4 & 0.001 & 11\,266.46 & 0.8 & 12 & 8 & 2\,205 & 4 & 0.001 \\
& & 6  & 12\,956.59 & 79.5 & 13 & 8 & 163\,838 & 4 & 0.010 & 12\,956.46 & 2.9 & 13 & 6 & 13\,087 & 4 & 0.011 \\
& & 8  & 16\,360.56 & 21.4 & 15 & 5 & 14\,031 & 3 & 0.010 & 16\,360.56 & 1.8 & 15 & 5 & 7\,620 & 3 & 0.010 \\
& & 10  & 19\,245.79 & 319.5 & 14 & 8 & 307\,963 & 3 & 0.519 & 19\,245.79 & 10.9 & 14 & 8 & 58\,025 & 3 & 0.519 \\
\hline
\multirow{5}{*}{15} & \multirow{5}{*}{(14,12)}
& 2  & 5\,969.11 & 9.6 & 16 & 8 & 5\,797 & 3 & 0.008 & 5\,969.11 & 0.4 & 16 & 8 & 826 & 3 & 0.008 \\
& & 4  & 9\,724.27 & 9.5 & 18 & 11 & 2\,854 & 3 & 0.006 & 9\,724.27 & 0.7 & 18 & 11 & 291 & 3 & 0.006 \\
& & 6  & 16\,765.66 & 31.8 & 20 & 6 & 9\,042 & 3 & 0.008 & 16\,765.66 & 4.7 & 20 & 6 & 19\,634 & 3 & 0.008 \\
& & 8  & 20\,281.36 & 40.3 & 20 & 6 & 20\,201 & 3 & 0.847 & 20\,281.36 & 2.7 & 20 & 6 & 9\,383 & 3 & 0.847 \\
& & 10  & 24\,844.10 & 15.6 & 18 & 7 & 4\,801 & 3 & 0.035 & 24\,844.10 & 1.8 & 18 & 7 & 4\,757 & 3 & 0.035 \\
\hline
\multirow{5}{*}{20} & \multirow{5}{*}{(17,13)}
& 2  & 7\,738.08 & 163.9 & 22 & 10 & 30\,305 & 4 & 0.008 & 7\,738.08 & 0.9 & 22 & 10 & 2\,922 & 4 & 0.008 \\
& & 4  & 13\,826.14 & TL & 24 & 11 & 2\,832\,008 & 3 & 0.348 & 13\,826.22 & 2.4 & 24 & 11 & 10\,076 & 3 & 0.347 \\
& & 6  & 17\,839.72 & TL & 24 & 5 & 1\,872\,774 & 2 & 3.857 & 17\,849.09 & 2.1 & 24 & 7 & 8\,831 & 3 & 3.803 \\
& & 8  & 21\,973.41 & TL & 27 & 4 & 498\,026 & 2 & 1.067 & 22\,004.11 & 2.0 & 27 & 6 & 3\,759 & 4 & 0.926 \\
& & 10  & 30\,952.97 & 47.3 & 25 & 8 & 7\,759 & 3 & 0.010 & 30\,952.97 & 1.4 & 25 & 8 & 1\,631 & 3 & 0.010 \\
\hline
\multirow{5}{*}{25} & \multirow{5}{*}{(22,14)}
& 2  & 9\,466.27 & 3.3 & 26 & 15 & 2 & 3 & 0.009 & 9\,466.27 & 0.5 & 26 & 15 & 102 & 3 & 0.009 \\
& & 4  & 16\,237.35 & 179.3 & 28 & 13 & 164\,122 & 3 & 0.010 & 16\,237.35 & 1.3 & 28 & 13 & 3\,165 & 3 & 0.010 \\
& & 6  & 24\,335.57 & TL & 32 & 6 & 242\,083 & 2 & 0.984 & 24\,335.60 & 5.3 & 32 & 7 & 22\,530 & 3 & 0.984 \\
& & 8  & 27\,925.61 & TL & 33 & 5 & 2\,990\,020 & 2 & 4.720 & 28\,329.31 & 26.5 & 33 & 7 & 174\,817 & 3 & 3.228 \\
& & 10  & 33\,370.48 & TL & 35 & 6 & 188\,195 & 2 & 3.497 & 33\,384.44 & 43.2 & 35 & 8 & 151\,251 & 3 & 3.453 \\
\hline
\multirow{5}{*}{30} & \multirow{5}{*}{(24,15)}
& 2  & 7\,861.46 & 20.3 & 33 & 12 & 5\,907 & 3 & 0.000 & 7\,861.46 & 0.6 & 33 & 12 & 163 & 3 & 0.000 \\
& & 4  & 16\,830.74 & 650.8 & 35 & 13 & 732\,516 & 3 & 0.006 & 16\,830.74 & 2.0 & 35 & 13 & 4\,137 & 3 & 0.006 \\
& & 6  & 25\,279.95 & TL & 33 & 12 & 1\,976\,806 & 3 & 2.392 & 25\,736.28 & 3.2 & 33 & 12 & 9\,916 & 3 & 0.576 \\
& & 8  & 27\,918.94 & TL & 39 & 0 & 316\,931 & 1 & 17.607 & 30\,486.34 & 84.0 & 39 & 9 & 312\,775 & 4 & 7.703 \\
& & 10  & 35\,523.31 & TL & 41 & 0 & 315\,996 & 1 & 14.728 & 36\,831.20 & 547.7 & 41 & 6 & 1\,995\,165 & 3 & 10.654 \\
\hline
\multirow{5}{*}{528} & \multirow{5}{*}{(24,24)}
& 2  & 69\,977.71 & 13.1 & 533 & 98 & 3 & 3 & 0.001 & 69\,977.71 & 1587.6 & 533 & 98 & 1\,672\,419 & 3 & 0.001 \\
& & 4  & 178\,278.68 & 1112.4 & 537 & 101 & 7\,756 & 3 & 0.176 & 178\,591.69 & TL & 537 & 96 & 1\,857\,097 & 2 & 0.000 \\
& & 6  & 331\,152.45 & 13.4 & 539 & 97 & 3 & 3 & 0.002 & 331\,152.46 & 986.3 & 539 & 97 & 1\,186\,715 & 3 & 0.002 \\
& & 8  & 395\,239.80 & 13.4 & 536 & 101 & 3 & 3 & 0.006 & 395\,239.80 & 319.4 & 536 & 101 & 446\,803 & 3 & 0.006 \\
& & 10  & 460\,184.65 & 14.2 & 540 & 98 & 3 & 3 & 0.005 & 460\,184.65 & TL & 540 & 98 & 3\,277\,078 & 3 & 0.005 \\
\hline
\end{tabular}
\end{sidewaystable}

\section{Experiments on the Eastern Massachusetts network}
\label{sec:EAM_net}

The Eastern Massachusetts network, with $74$ nodes and $258$ arcs, is the largest that
we consider. Instances are generated as on the Sioux Falls networks --- the same
Wardrop-equilibrium-based selection of the toll arcs (Appendix~\ref{app:toll_arcs}), the
same reference cost structure and toll upper bound
$\tau_a^{\tt up}=\min(c_a^{\tt fft},\tau_a^\infty)$ --- over twelve demand levels
ranging from $\left|\mathcal{K}\right|=4$ to $\left|\mathcal{K}\right|=100$, each
combined with $|A^{\tt toll}|\in\{2,4,6,8,10\}$ toll arcs, for $60$ instances in all;
the first $30$ demands of the nested demand ladder are listed in
Table~\ref{tab:em_od30}. Tables~\ref{tab:em_exact} and~\ref{tab:em_heur} report the
per-instance results for the exact formulations and for the heuristics, with the same
pooled instance-level bound $\overline{z}$ as before.

\begin{table}[h!]
\centering
\caption{\label{tab:em_od30} The first $30$ origin--destination demands used to build
the Eastern Massachusetts instances, in the order in which they enter the nested
demand ladder ($\left|\mathcal{K}\right|=k$ uses demands $1$ to $k$). 
}
 \begin{tabular}{|c|c|c||c|c|c||c|c|c|}
 \hline
\#& $(o_k,s_k)$& demand &\# & $(o_k,s_k)$& demand &\# & $(o_k,s_k)$& demand  \\
 \hline\hline
1&(36,12)&14.33 &11&(58,38)&17.46&21&(66,25)&10.98\\
2&(50,57)&121.44 &12&(54,1)&0.53 &22&(39,35)&6.96\\
3&(1,54)&1.88 &13&(12,55)&146.29 &23&(69,54)&624.56\\
4&(1,36)&11.72 &14&(62,13)&5.23 &24&(54,69)&562.29\\
5&(64,1)&0.55 &15&(2,14)&104.28 &25&(12,56)&146.29\\
6&(57,50)&157.19&16&(55,12)&136.83 &26&(56,12)&136.83\\
7&(38,57)&5.22&17&(36,2)&13.10&27&(1,59)&2.02\\
8&(13,64)&2.01&18&(1,62)&1.94&28&(1,69)&1.97\\
9&(51,10)&11.35&19&(53,2)&4.66&29&(1,67)&1.95\\
10&(10,51)&8.15&20&(25,66)&42.32&30&(1,61)&1.88\\
 \hline
 \end{tabular}
\end{table}

\paragraph{Scalability.}
The larger network drives the gap between the two exact formulations to its extreme.
Within the time limit, (BLS) solves \emph{none} of the $60$ instances and returns a
feasible solution on only $15$ of them; on the remaining $45$ it finds no feasible toll
vector at all within the budget. (SD-A), by contrast, solves $26$ of the $60$ instances
and returns a feasible solution on every one, with a shifted geometric mean solution
time of $392$ seconds, against the full time limit for (BLS). As on the smaller
networks, the instances (SD-A) solves are not simply the smallest: 
it solves all ten instances with $\left|\mathcal{K}\right|\in\{4,5\}$, but only one or two
per demand level thereafter, 
the difficulty again turning on the interaction between commodities and tolled arcs
rather than on size alone.

\paragraph{Gaps at the time limit.}
The certified gaps confirm that the failures are ones of certification rather than of
primal search. 
On the eleven instances that
(SD-A) solves and on which (BLS) returns a feasible solution, that solution is optimal
up to numerical precision, even though (BLS) never certifies it. On the $34$ instances
that neither formulation solves, the pooled certified gap of (SD-A) stays moderate on
most --- a median of $2.1\%$ --- though it reaches $42\%$ on the hardest, reflecting how
loose the available upper bound becomes once the network is large and the demand
intermediate.
\begin{sidewaystable}
\caption{\label{tab:em_exact} Comparison of the exact formulations (BLS) and (SD-A) on the
Eastern Massachusetts network ($74$ nodes, $258$ arcs; time limit TL${}=1800$\,s). The ``objective''
field reports a valid lower bound $\underline{z}$ on $z^\star$; ``gap$_{\tt own}$'' the gap certified
by the formulation alone, and ``gap'' the certified gap $(\overline{z}-\underline{z})/\underline{z}$
obtained with the pooled instance-level bound $\overline{z}$. A dash (--) indicates that no feasible
solution was found.}
\tiny
\centering
\begin{tabular}{|c|c|c||c|c|c|c|c||c|c|c|c|c|c|c|c|}
\hline
\multirow{2}{*}{$\left|\mathcal{K}\right|$} &
\multirow{2}{*}{$(\#o,\#d)$} &
\multirow{2}{*}{$\left|A^{\tt{toll}}\right|$} &
\multicolumn{5}{|c||}{${\tt{(BLS)}}$} &
\multicolumn{8}{|c|}{${\tt{(SD\text{-}A)}}$}
\\
\cline{4-16}
& & & objective & time & nodes & gap$_{\tt own}$ (\%) & gap (\%)
& objective & time & $\#p_i$ & $\#p_g$ & nodes & it & gap$_{\tt own}$ (\%) & gap (\%)\\
\hline
\multirow{5}{*}{4} & \multirow{5}{*}{(3,4)}
& 2  & 3.22 & TL & 503\,972 & 0.941 & 0.000 & 3.22 & 1.5 & 4 & 1 & 2 & 2 & 0.010 & 0.000 \\
& & 4  & 4.48 & TL & 426\,107 & 3.505 & 0.006 & 4.48 & 4.6 & 6 & 1 & 1\,976 & 2 & 0.010 & 0.010 \\
& & 6  & 5.65 & TL & 91\,006 & 14.927 & 0.010 & 5.65 & 32.8 & 6 & 3 & 1\,123 & 2 & 0.010 & 0.010 \\
& & 8  & 6.74 & TL & 104\,643 & 31.428 & 0.005 & 6.74 & 9.9 & 6 & 3 & 875 & 2 & 0.004 & 0.004 \\
& & 10  & 8.65 & TL & 55\,380 & 113.798 & 0.006 & 8.65 & 1.2 & 6 & 1 & 2 & 2 & 0.007 & 0.007 \\
\hline
\multirow{5}{*}{5} & \multirow{5}{*}{(4,5)}
& 2  & 1.99 & TL & 257\,500 & 1.636 & 0.000 & 1.99 & 2.2 & 7 & 1 & 2 & 2 & 0.000 & 0.000 \\
& & 4  & 4.36 & TL & 150\,580 & 5.151 & 0.001 & 4.36 & 4.8 & 7 & 1 & 149 & 2 & 0.001 & 0.001 \\
& & 6  & 4.67 & TL & 194\,846 & 5.445 & 0.008 & 4.67 & 4.3 & 7 & 1 & 580 & 2 & 0.009 & 0.009 \\
& & 8  & 5.36 & TL & 108\,087 & 13.582 & 0.009 & 5.36 & 14.7 & 6 & 4 & 3 & 3 & 0.009 & 0.009 \\
& & 10  & 5.88 & TL & 97\,517 & 48.963 & 0.011 & 5.88 & 107.2 & 8 & 3 & 50\,364 & 2 & 0.010 & 0.010 \\
\hline
\multirow{5}{*}{6} & \multirow{5}{*}{(5,6)}
& 2  & 20.22 & TL & 36\,403 & 32.883 & 0.000 & 20.22 & 51.2 & 7 & 3 & 2\,816 & 3 & 0.001 & 0.001 \\
& & 4  & 37.28 & TL & 23\,776 & 62.598 & 0.720 & 37.43 & TL & 9 & 2 & 191\,967 & 3 & 0.310 & 0.310 \\
& & 6  & 39.42 & TL & 23\,155 & 70.048 & 0.914 & 39.42 & TL & 10 & 1 & 270\,205 & 2 & 0.925 & 0.925 \\
& & 8  & -- & TL & 21\,200 & -- & -- & 41.71 & TL & 10 & 0 & 200\,610 & 1 & 1.888 & 1.888 \\
& & 10  & 39.42 & TL & 21\,649 & 157.619 & 8.840 & 42.02 & TL & 10 & 0 & 216\,374 & 1 & 2.087 & 2.087 \\
\hline
\multirow{5}{*}{7} & \multirow{5}{*}{(6,6)}
& 2  & -- & TL & 17\,019 & -- & -- & 20.22 & 97.8 & 8 & 3 & 2\,681 & 3 & 0.002 & 0.002 \\
& & 4  & -- & TL & 15\,543 & -- & -- & 37.40 & TL & 10 & 1 & 91\,564 & 2 & 0.601 & 0.601 \\
& & 6  & -- & TL & 17\,124 & -- & -- & 39.42 & TL & 11 & 1 & 206\,695 & 2 & 1.021 & 1.021 \\
& & 8  & -- & TL & 18\,548 & -- & -- & 41.71 & TL & 11 & 0 & 179\,151 & 1 & 1.875 & 1.875 \\
& & 10  & 40.52 & TL & 12\,730 & 196.391 & 5.835 & 42.02 & TL & 11 & 0 & 187\,747 & 1 & 2.053 & 2.053 \\
\hline
\multirow{5}{*}{8} & \multirow{5}{*}{(7,7)}
& 2  & -- & TL & 19\,730 & -- & -- & 1.70 & 2.8 & 11 & 4 & 2 & 2 & 0.000 & 0.000 \\
& & 4  & -- & TL & 15\,808 & -- & -- & 7.07 & 11.7 & 12 & 4 & 834 & 2 & 0.002 & 0.002 \\
& & 6  & -- & TL & 10\,911 & -- & -- & 31.04 & TL & 12 & 7 & 141\,244 & 5 & 0.110 & 0.110 \\
& & 8  & -- & TL & 5\,668 & -- & -- & 45.31 & TL & 14 & 0 & 99\,156 & 1 & 6.285 & 6.285 \\
& & 10  & -- & TL & 6\,396 & -- & -- & 46.62 & TL & 14 & 0 & 161\,757 & 1 & 7.276 & 7.276 \\
\hline\hline
\multirow{5}{*}{10} & \multirow{5}{*}{(9,9)}
& 2  & -- & TL & 8\,088 & -- & -- & 1.69 & 4.0 & 13 & 3 & 2 & 2 & 0.006 & 0.006 \\
& & 4  & -- & TL & 5\,504 & -- & -- & 4.23 & 17.1 & 15 & 5 & 573 & 2 & 0.009 & 0.009 \\
& & 6  & -- & TL & 5\,157 & -- & -- & 4.66 & TL & 16 & 0 & 91\,706 & 1 & 42.413 & 42.413 \\
& & 8  & -- & TL & 5\,907 & -- & -- & 10.04 & TL & 17 & 0 & 96\,568 & 1 & 18.016 & 18.016 \\
& & 10  & -- & TL & 4\,306 & -- & -- & 38.90 & TL & 18 & 0 & 102\,937 & 1 & 3.913 & 3.913 \\
\hline
\multirow{5}{*}{15} & \multirow{5}{*}{(14,13)}
& 2  & -- & TL & 2\,491 & -- & -- & 12.80 & 41.3 & 18 & 3 & 507 & 2 & 0.000 & 0.000 \\
& & 4  & -- & TL & 2\,236 & -- & -- & 15.55 & 170.0 & 21 & 4 & 10\,251 & 2 & 0.000 & 0.000 \\
& & 6  & -- & TL & 2\,097 & -- & -- & 14.95 & TL & 22 & 0 & 88\,782 & 1 & 20.962 & 20.962 \\
& & 8  & -- & TL & 2\,066 & -- & -- & 26.76 & TL & 23 & 0 & 75\,690 & 1 & 13.547 & 13.547 \\
& & 10  & -- & TL & 2\,190 & -- & -- & 36.89 & TL & 26 & 0 & 74\,000 & 1 & 10.258 & 10.258 \\
\hline
\multirow{5}{*}{20} & \multirow{5}{*}{(17,16)}
& 2  & -- & TL & 1\,476 & -- & -- & 16.88 & 37.3 & 21 & 3 & 2 & 2 & 0.000 & 0.000 \\
& & 4  & -- & TL & 1\,631 & -- & -- & 19.74 & 419.9 & 24 & 3 & 29\,614 & 2 & 0.000 & 0.000 \\
& & 6  & -- & TL & 1\,775 & -- & -- & 21.36 & TL & 28 & 0 & 125\,621 & 1 & 1.378 & 1.378 \\
& & 8  & -- & TL & 2\,310 & -- & -- & 33.79 & TL & 29 & 0 & 151\,097 & 1 & 1.795 & 1.795 \\
& & 10  & -- & TL & 2\,299 & -- & -- & 41.85 & TL & 30 & 0 & 70\,200 & 1 & 18.194 & 18.194 \\
\hline
\multirow{5}{*}{25} & \multirow{5}{*}{(20,20)}
& 2  & -- & TL & 666 & -- & -- & 20.76 & 295.6 & 26 & 12 & 380 & 4 & 0.000 & 0.000 \\
& & 4  & -- & TL & 542 & -- & -- & 35.71 & TL & 30 & 10 & 63\,960 & 4 & 0.142 & 0.142 \\
& & 6  & -- & TL & 916 & -- & -- & 59.77 & TL & 32 & 10 & 95\,536 & 3 & 1.073 & 1.073 \\
& & 8  & -- & TL & 1\,244 & -- & -- & 63.15 & TL & 34 & 8 & 57\,372 & 3 & 2.160 & 2.160 \\
& & 10  & -- & TL & 449 & -- & -- & 143.90 & TL & 36 & 7 & 94\,146 & 3 & 0.110 & 0.110 \\
\hline
\multirow{5}{*}{30} & \multirow{5}{*}{(21,23)}
& 2  & -- & TL & 32 & -- & -- & 27.81 & 525.2 & 33 & 11 & 9\,527 & 4 & 0.003 & 0.003 \\
& & 4  & -- & TL & 208 & -- & -- & 53.89 & 214.0 & 37 & 9 & 13\,323 & 3 & 0.010 & 0.010 \\
& & 6  & -- & TL & 118 & -- & -- & 69.10 & TL & 41 & 14 & 99\,806 & 3 & 0.391 & 0.391 \\
& & 8  & -- & TL & 18 & -- & -- & 92.38 & TL & 42 & 18 & 101\,691 & 3 & 0.812 & 0.812 \\
& & 10  & -- & TL & 18 & -- & -- & 117.63 & TL & 47 & 9 & 84\,365 & 3 & 2.711 & 2.711 \\
\hline
\multirow{5}{*}{50} & \multirow{5}{*}{(26,29)}
& 2  & -- & TL & 16 & -- & -- & 30.07 & 202.9 & 53 & 18 & 773 & 4 & 0.001 & 0.001 \\
& & 4  & -- & TL & 1 & -- & -- & 65.14 & 816.1 & 59 & 12 & 29\,279 & 3 & 0.010 & 0.010 \\
& & 6  & -- & TL & 1 & -- & -- & 90.30 & TL & 61 & 17 & 61\,099 & 4 & 0.890 & 0.890 \\
& & 8  & -- & TL & 31 & -- & -- & 102.68 & TL & 65 & 7 & 57\,566 & 2 & 10.302 & 10.302 \\
& & 10  & -- & TL & 1 & -- & -- & 118.13 & TL & 72 & 9 & 59\,018 & 2 & 29.360 & 29.360 \\
\hline
\multirow{5}{*}{100} & \multirow{5}{*}{(32,35)}
& 2  & -- & TL & 1 & -- & -- & 26.09 & 323.9 & 109 & 28 & 617 & 3 & 0.010 & 0.010 \\
& & 4  & -- & TL & 1 & -- & -- & 50.01 & TL & 110 & 31 & 30\,190 & 3 & 1.068 & 1.068 \\
& & 6  & -- & TL & 1 & -- & -- & 82.83 & TL & 129 & 30 & 22\,010 & 3 & 23.212 & 23.212 \\
& & 8  & -- & TL & 1 & -- & -- & 99.94 & TL & 140 & 30 & 54\,794 & 3 & 20.063 & 20.063 \\
& & 10  & -- & TL & 1 & -- & -- & 127.68 & TL & 141 & 28 & 34\,094 & 3 & 15.651 & 15.651 \\
\hline
\end{tabular}
\end{sidewaystable}

\paragraph{Path-based heuristics.}
The path-based heuristics extend their usefulness to this network. On the $26$
instances whose optimum is known, both (SD-P) and (CC) attain it, with a worst-case
shortfall of $0.001\%$; and on the $45$ instances on which (BLS) returns nothing, both
return a feasible solution. The two heuristics differ mainly in robustness and cost: 
(CC) is faster than (SD-P) on $58$ of the $60$ instances, with a median speed-up of $13$
over the whole benchmark, 
and reaches the time limit on a single instance, against twenty for (SD-P). On the $34$
instances on which (SD-A) itself reaches the time limit, (CC) returns a solution at
least as good as the (SD-A) incumbent on \emph{all} of them, and (SD-P) on $22$; where
the exact formulations run out of time, the complementarity-based heuristic is thus the
most reliable source of a good feasible toll vector on this network. 
\begin{sidewaystable}
\caption{\label{tab:em_heur} Path-based heuristics (SD-P) and (CC) on the
Eastern Massachusetts network ($74$ nodes, $258$ arcs; time limit TL${}=1800$\,s). The
``gap'' column is the certified gap against the pooled bound $\overline{z}$ of
Table~\ref{tab:em_exact}; all other columns are as in Table~\ref{tab:rq4}. A dash
(--) indicates that no feasible solution was found.}
\tiny
\centering
\begin{tabular}{|c|c|c||c|c|c|c|c|c|c||c|c|c|c|c|c|c|}
\hline
\multirow{2}{*}{$\left|\mathcal{K}\right|$} &
\multirow{2}{*}{$(\#o,\#d)$} &
\multirow{2}{*}{$\left|A^{\tt{toll}}\right|$} &
\multicolumn{7}{|c||}{(SD-P)} &
\multicolumn{7}{|c|}{(CC)}
\\
\cline{4-17}
& & &
obj & time & $\#p_i$ & $\#p_g$ & nodes & it & gap &
obj & time & $\#p_i$ & $\#p_g$ & nodes & it & gap \\
\hline
\multirow{5}{*}{4} & \multirow{5}{*}{(3,4)}
& 2  & 3.22 & 0.1 & 4 & 1 & 0 & 2 & 0.000 & 3.22 & 0.1 & 4 & 1 & 0 & 2 & 0.000 \\
& & 4  & 4.48 & 0.9 & 6 & 1 & 1\,003 & 2 & 0.010 & 4.48 & 0.2 & 6 & 1 & 16 & 2 & 0.010 \\
& & 6  & 5.65 & 1.9 & 6 & 3 & 1\,793 & 2 & 0.010 & 5.65 & 0.2 & 6 & 3 & 8 & 2 & 0.010 \\
& & 8  & 6.74 & 18.9 & 6 & 2 & 132\,703 & 2 & 0.005 & 6.74 & 0.4 & 6 & 2 & 197 & 2 & 0.005 \\
& & 10  & 8.65 & 0.2 & 6 & 1 & 2 & 2 & 0.007 & 8.65 & 0.2 & 6 & 1 & 6 & 2 & 0.007 \\
\hline
\multirow{5}{*}{5} & \multirow{5}{*}{(4,5)}
& 2  & 1.99 & 0.3 & 7 & 1 & 6 & 2 & 0.000 & 1.99 & 0.1 & 7 & 1 & 2 & 2 & 0.000 \\
& & 4  & 4.36 & 0.2 & 7 & 1 & 4 & 2 & 0.001 & 4.36 & 0.2 & 7 & 1 & 10 & 2 & 0.001 \\
& & 6  & 4.67 & 0.5 & 7 & 1 & 197 & 2 & 0.009 & 4.67 & 0.2 & 7 & 1 & 16 & 2 & 0.009 \\
& & 8  & 5.36 & 0.3 & 6 & 3 & 2 & 2 & 0.009 & 5.36 & 0.2 & 6 & 3 & 2 & 2 & 0.009 \\
& & 10  & 5.88 & 3.0 & 8 & 3 & 1\,377 & 2 & 0.011 & 5.88 & 0.3 & 8 & 3 & 14 & 2 & 0.011 \\
\hline
\multirow{5}{*}{6} & \multirow{5}{*}{(5,6)}
& 2  & 20.22 & 1.3 & 7 & 3 & 1\,105 & 3 & 0.000 & 20.22 & 0.2 & 7 & 3 & 3 & 3 & 0.000 \\
& & 4  & 37.44 & 122.0 & 9 & 2 & 331\,417 & 2 & 0.303 & 37.44 & 0.5 & 9 & 2 & 941 & 2 & 0.303 \\
& & 6  & 39.15 & TL & 10 & 0 & 231\,886 & 1 & 1.617 & 39.42 & 8.0 & 10 & 2 & 83\,787 & 2 & 0.915 \\
& & 8  & 41.41 & TL & 10 & 0 & 644\,448 & 1 & 2.631 & 41.78 & 8.1 & 10 & 2 & 61\,653 & 2 & 1.712 \\
& & 10  & 41.76 & TL & 10 & 0 & 373\,967 & 1 & 2.728 & 42.10 & 20.5 & 10 & 2 & 233\,296 & 2 & 1.912 \\
\hline
\multirow{5}{*}{7} & \multirow{5}{*}{(6,6)}
& 2  & 20.22 & 2.4 & 8 & 3 & 542 & 3 & 0.002 & 20.22 & 0.2 & 8 & 3 & 3 & 3 & 0.002 \\
& & 4  & 37.45 & 65.1 & 10 & 2 & 226\,514 & 2 & 0.479 & 37.45 & 0.5 & 10 & 2 & 1\,166 & 2 & 0.479 \\
& & 6  & 39.15 & TL & 11 & 0 & 502\,845 & 1 & 1.716 & 39.42 & 3.1 & 11 & 2 & 20\,179 & 2 & 1.013 \\
& & 8  & 41.41 & TL & 11 & 0 & 494\,961 & 1 & 2.623 & 41.79 & 11.2 & 11 & 2 & 110\,170 & 2 & 1.699 \\
& & 10  & 41.76 & TL & 11 & 0 & 416\,280 & 1 & 2.701 & 42.10 & 33.8 & 11 & 2 & 385\,733 & 2 & 1.878 \\
\hline
\multirow{5}{*}{8} & \multirow{5}{*}{(7,7)}
& 2  & 1.70 & 0.3 & 11 & 3 & 2 & 2 & 0.000 & 1.70 & 0.3 & 11 & 3 & 6 & 2 & 0.000 \\
& & 4  & 7.07 & 1.7 & 12 & 4 & 328 & 3 & 0.002 & 7.07 & 0.4 & 12 & 1 & 226 & 2 & 0.002 \\
& & 6  & 31.04 & 1.9 & 12 & 7 & 18 & 4 & 0.087 & 31.04 & 0.6 & 12 & 7 & 142 & 3 & 0.087 \\
& & 8  & 46.87 & TL & 14 & 4 & 3\,514\,436 & 3 & 2.761 & 46.87 & 32.5 & 14 & 7 & 224\,218 & 4 & 2.743 \\
& & 10  & 49.06 & 708.2 & 14 & 2 & 245\,767 & 2 & 1.953 & 49.06 & 6.0 & 14 & 2 & 37\,730 & 2 & 1.954 \\
\hline
\multirow{5}{*}{10} & \multirow{5}{*}{(9,9)}
& 2  & 1.69 & 0.4 & 13 & 3 & 2 & 2 & 0.006 & 1.69 & 0.3 & 13 & 3 & 8 & 2 & 0.006 \\
& & 4  & 4.23 & 1.0 & 15 & 5 & 2 & 2 & 0.009 & 4.23 & 0.4 & 15 & 5 & 16 & 2 & 0.009 \\
& & 6  & 6.20 & 1.1 & 16 & 4 & 2 & 2 & 7.118 & 6.20 & 0.4 & 16 & 4 & 56 & 2 & 7.118 \\
& & 8  & 11.57 & 2.7 & 17 & 2 & 142 & 2 & 2.430 & 11.57 & 0.8 & 17 & 2 & 413 & 2 & 2.430 \\
& & 10  & 39.03 & TL & 18 & 3 & 777\,378 & 3 & 3.578 & 39.27 & 2.1 & 18 & 6 & 1\,834 & 4 & 2.944 \\
\hline
\multirow{5}{*}{15} & \multirow{5}{*}{(14,13)}
& 2  & 12.80 & 1.0 & 18 & 3 & 13 & 2 & 0.000 & 12.80 & 0.4 & 18 & 3 & 2 & 2 & 0.000 \\
& & 4  & 15.55 & 14.3 & 21 & 4 & 4\,427 & 2 & 0.000 & 15.55 & 0.7 & 21 & 4 & 149 & 2 & 0.001 \\
& & 6  & 15.26 & TL & 22 & 0 & 761\,392 & 1 & 18.467 & 16.91 & 0.7 & 22 & 4 & 142 & 2 & 6.909 \\
& & 8  & 26.95 & TL & 23 & 0 & 2\,403\,284 & 1 & 12.735 & 28.98 & 1.4 & 23 & 4 & 4\,524 & 2 & 4.824 \\
& & 10  & 36.83 & TL & 26 & 0 & 682\,674 & 1 & 10.447 & 38.72 & 2.4 & 26 & 4 & 4\,956 & 3 & 5.044 \\
\hline
\multirow{5}{*}{20} & \multirow{5}{*}{(17,16)}
& 2  & 16.88 & 1.2 & 21 & 3 & 2 & 2 & 0.000 & 16.88 & 0.5 & 21 & 3 & 2 & 2 & 0.000 \\
& & 4  & 19.74 & 8.5 & 24 & 3 & 5\,067 & 2 & 0.000 & 19.74 & 0.6 & 24 & 3 & 2 & 2 & 0.001 \\
& & 6  & 21.36 & 1047.3 & 28 & 1 & 1\,493\,571 & 2 & 1.361 & 21.36 & 0.5 & 28 & 0 & 141 & 1 & 1.363 \\
& & 8  & 33.79 & 7.3 & 29 & 0 & 2\,391 & 1 & 1.783 & 33.79 & 0.6 & 29 & 0 & 195 & 1 & 1.783 \\
& & 10  & 43.63 & 19.4 & 30 & 2 & 12\,217 & 2 & 13.374 & 43.63 & 1.1 & 30 & 2 & 424 & 2 & 13.374 \\
\hline
\multirow{5}{*}{25} & \multirow{5}{*}{(20,20)}
& 2  & 20.76 & 4.5 & 26 & 14 & 4 & 4 & 0.000 & 20.76 & 1.0 & 26 & 14 & 4 & 4 & 0.000 \\
& & 4  & 35.72 & 156.3 & 30 & 13 & 61\,957 & 4 & 0.126 & 35.72 & 3.0 & 30 & 13 & 12\,129 & 4 & 0.126 \\
& & 6  & 60.16 & 305.1 & 32 & 13 & 396\,359 & 4 & 0.423 & 60.16 & 8.8 & 32 & 13 & 41\,254 & 4 & 0.423 \\
& & 8  & 63.01 & TL & 34 & 7 & 937\,934 & 2 & 2.380 & 63.54 & 17.0 & 34 & 10 & 97\,620 & 3 & 1.539 \\
& & 10  & 143.73 & TL & 36 & 0 & 596\,355 & 1 & 0.227 & 143.91 & 19.0 & 36 & 9 & 95\,444 & 3 & 0.107 \\
\hline
\multirow{5}{*}{30} & \multirow{5}{*}{(21,23)}
& 2  & 27.81 & 30.1 & 33 & 13 & 8\,431 & 4 & 0.000 & 27.81 & 1.6 & 33 & 13 & 313 & 4 & 0.003 \\
& & 4  & 53.89 & 18.7 & 37 & 11 & 13\,996 & 3 & 0.010 & 53.89 & 5.2 & 37 & 11 & 21\,247 & 3 & 0.009 \\
& & 6  & 69.10 & TL & 41 & 12 & 1\,946\,090 & 4 & 0.396 & 69.11 & 11.1 & 41 & 12 & 38\,240 & 4 & 0.386 \\
& & 8  & 89.83 & TL & 42 & 10 & 1\,093\,149 & 2 & 3.672 & 92.39 & 43.4 & 42 & 16 & 210\,412 & 4 & 0.802 \\
& & 10  & 116.30 & TL & 47 & 0 & 620\,116 & 1 & 3.887 & 119.61 & 330.6 & 47 & 12 & 1\,581\,413 & 3 & 1.010 \\
\hline
\multirow{5}{*}{50} & \multirow{5}{*}{(26,29)}
& 2  & 30.07 & 31.3 & 53 & 18 & 6\,409 & 4 & 0.000 & 30.07 & 2.5 & 53 & 18 & 1\,655 & 4 & 0.000 \\
& & 4  & 65.14 & 57.0 & 59 & 15 & 67\,294 & 3 & 0.010 & 65.14 & 12.6 & 59 & 15 & 113\,913 & 3 & 0.010 \\
& & 6  & 90.38 & 726.3 & 61 & 16 & 245\,515 & 4 & 0.805 & 90.38 & 7.3 & 61 & 16 & 17\,324 & 4 & 0.805 \\
& & 8  & 104.31 & TL & 65 & 0 & 367\,880 & 1 & 8.583 & 106.12 & 427.9 & 65 & 13 & 2\,203\,071 & 3 & 6.730 \\
& & 10  & 132.82 & TL & 72 & 0 & 429\,314 & 1 & 15.049 & 132.83 & TL & 72 & 10 & 6\,654\,012 & 2 & 15.041 \\
\hline
\multirow{5}{*}{100} & \multirow{5}{*}{(32,35)}
& 2  & 26.10 & 24.5 & 109 & 45 & 1\,086 & 4 & 0.000 & 26.10 & 6.4 & 109 & 45 & 430 & 4 & 0.000 \\
& & 4  & 50.15 & 697.7 & 110 & 48 & 857\,470 & 4 & 0.795 & 50.15 & 246.1 & 110 & 47 & 766\,063 & 4 & 0.794 \\
& & 6  & 101.35 & 626.1 & 129 & 43 & 464\,324 & 4 & 0.701 & 101.35 & 22.5 & 129 & 43 & 27\,933 & 4 & 0.701 \\
& & 8  & 117.82 & TL & 140 & 23 & 482\,746 & 2 & 1.842 & 118.93 & 32.7 & 140 & 43 & 36\,824 & 4 & 0.890 \\
& & 10  & 145.23 & TL & 141 & 23 & 428\,673 & 2 & 1.675 & 146.48 & 41.1 & 141 & 42 & 56\,525 & 4 & 0.806 \\
\hline
\end{tabular}
\end{sidewaystable}

\end{document}